\documentclass{article}

 \usepackage[preprint]{neurips_2026}

\usepackage[utf8]{inputenc} % allow utf-8 input
\usepackage[T1]{fontenc}    % use 8-bit T1 fonts
\usepackage{hyperref}       % hyperlinks
\usepackage{url}            % simple URL typesetting
\usepackage{booktabs}       % professional-quality tables
\usepackage{amsfonts}       % blackboard math symbols
\usepackage{nicefrac}       % compact symbols for 1/2, etc.
\usepackage{microtype}      % microtypography
\usepackage{xcolor}         % colors

\usepackage{amsthm}
\usepackage{amsmath}
\usepackage{amssymb}
\usepackage{algorithm}
\usepackage{algpseudocode}
\usepackage{graphicx}
\usepackage{enumitem}
\usepackage{tabularx}

\newtheorem{theorem}{Theorem}
\newtheorem{proposition}{Proposition}
\newtheorem{assumption}{Assumption}
\newtheorem{corollary}{Corollary}
\newtheorem{definition}{Definition}
\newtheorem{remark}{Remark}
\newtheorem{lemma}{Lemma}

\newcommand{\proofstep}[1]{\par\smallskip\noindent\textbf{#1}\ }

\newcounter{Eitem}
\renewcommand{\theEitem}{E\arabic{Eitem}}
\newcommand{\Etag}[2]{%
  \refstepcounter{Eitem}%
  \medskip\noindent\textbf{(\theEitem)~#1}\label{#2}%
}

\title{A Barrier-Metric First-Order Method for Linearly Constrained Bilevel Optimization}

\author{%
  Tenglong Hong\thanks{Department of Civil and Environmental Engineering,
    University of California, Los Angeles. \texttt{tlhong0802@ucla.edu}.}
  \And
  Paul Grigas\thanks{Department of Industrial Engineering and Operations Research,
    University of California, Berkeley. \texttt{pgrigas@berkeley.edu}.}
}

\begin{document}

\maketitle

\begin{abstract}
We study bilevel optimization with a fixed polyhedral lower feasible set. Such problems are challenging for two reasons: active-set changes can make
the upper objective nonsmooth, and existing hypergradient methods typically
require lower-Hessian inversions or equivalent linear solves, which are
computationally expensive. To address these issues, we adopt a logarithmic barrier smoothing of the lower problem to obtain a
differentiable approximation of the constrained bilevel objective, and develop a proxy-gradient algorithm
for the resulting barrier-smoothed surrogate. The algorithm uses only gradients of the upper and lower objectives; its only second-order object is
the explicit logarithmic barrier Hessian determined by the fixed polyhedral constraints. Barrier smoothing restores differentiability, but Euclidean smoothness constants are not uniformly bounded near the boundary. We therefore develop a local Dikin-geometry analysis in which the barrier-metric provides
an oracle-free curvature scale near the moving lower centers. This leads to \textit{barrier-aware}
schedules that keep the iterates inside locally well-behaved regions. For the barrier-smoothed objective, we prove stationarity rates of \(\widetilde O(K^{-2/3})\) in the deterministic setting and
\(\widetilde O(K^{-2/5})\) under upper-level-only bounded stochastic noise after \(K\) outer iterations, together with quantitative bias
control as the barrier parameter decreases.
\end{abstract}

\section{Introduction}

Constrained bilevel optimization arises in a wide range of learning, control, and decision-making
problems in which an upper-level decision must be chosen while accounting for the optimal response
of a lower-level constrained optimization problem. We consider
\begin{equation}\label{eq:bi-level_constrained}
\min_{x\in X} \; F(x) := f\bigl(x,y^\star(x)\bigr),
\qquad
y^\star(x)\in\arg\min_{y\in Y} g(x,y),
\end{equation}
where \(X\) and \(Y\) are the upper- and lower-level feasible regions. The upper objective depends on
the lower solution both directly and implicitly through the solution map \(x\mapsto y^\star(x)\). In
many applications, the lower variable represents a constrained allocation, flow, or policy, while the
upper variable specifies design, pricing, or calibration decisions. 
Such hierarchical structures arise naturally in meta-learning~\citep{finn2017model},
data reweighting~\citep{ren2018learning}, hyperparameter tuning~\citep{franceschi2018bilevel},
traffic control and toll design~\citep{brotcorne2001bilevel}, and data-driven
decision-making~\citep{donti2017task,elmachtoub2022smart}.

Despite its modeling flexibility, constrained bilevel optimization remains computationally
challenging. Even in smooth settings, direct gradient-based optimization of the upper objective
requires access to the hypergradient of \(F\), which depends on the sensitivity of the lower solution
\(y^\star(x)\) with respect to \(x\). For unconstrained lower problems, this sensitivity can often be
characterized by implicit differentiation, at the cost of evaluating mixed second derivatives and
solving linear systems involving the lower Hessian~\citep{ghadimi2018approximation,ji2021bilevel}.
In constrained problems, however, the situation is more delicate: the lower solution map is shaped
not only by the objective \(g\), but also by the active-set geometry of the feasible region \(Y\). As a
result, the mapping \(x\mapsto y^\star(x)\) may fail to be differentiable, and the corresponding upper
objective \(F(x)=f(x,y^\star(x))\) may inherit this nonsmoothness.

To address this differentiability obstruction, we adopt a barrier-smoothing approach for the lower-level
constrained problem. The logarithmic barrier moves the lower minimizer to the interior of the feasible
region and yields a differentiable barrier-smoothed outer objective \(F_\mu\). However, this smoothing
step introduces a boundary-induced geometric difficulty: near \(\partial Y\), the barrier Hessian blows
up in boundary-normal directions, so Euclidean curvature is not uniformly controlled. Moreover,
Euclidean length no longer reveals the natural local curvature scale: a direction with moderate
\(L_2\) length can still have very large boundary-normal curvature near an active face.
This motivates a local Dikin-geometric analysis. The Dikin metric induced by the barrier Hessian
identifies barrier-adapted neighborhoods in which the barrierized lower problem recovers controlled
curvature and smoothness. It also provides an oracle-free proxy for the curvature at the unknown
lower center: within these neighborhoods, the current barrier geometry is spectrally comparable to the
local Hessian scale at the nearby center. We then design schedules that keep the iterates inside these
well-behaved regions, so that the local curvature and smoothness estimates carry the convergence
analysis through. All stationarity guarantees below use the squared-gradient convention: an
\(\epsilon\)-stationary point of \(F_\mu\) is a point \(x\) satisfying
\(\|\nabla F_\mu(x)\|_2^2\le\epsilon\).

\subsection{Our Contributions}

Our main contributions are as follows.
\begin{itemize}[leftmargin=*]
    \item \textbf{Barrier smoothing for constrained bilevel optimization.}
    We introduce a barrier-smoothed surrogate for constrained bilevel problems with polyhedral lower feasible sets, prove differentiability of the resulting lower solution map and outer objective, and establish quantitative bias-control bounds as the barrier parameter decreases.
    \item \textbf{A barrier-metric first-order method with explicit rates.}
    We develop a proxy-gradient first-order method for the barrier-smoothed problem. The method uses
    only first-order derivatives of upper and lower objectives, while the lower-level problems are preconditioned by the
    explicit logarithmic barrier-metric determined by the polyhedral feasible set. For the barrier-smoothed objective, we prove stationarity rates of \(\widetilde O(K^{-2/3})\) in the deterministic setting and
    \(\widetilde O(K^{-2/5})\) under upper-level-only bounded stochastic noise after \(K\) outer iterations.
    \item \textbf{Dikin local geometry analysis.}
    Barrier smoothing restores differentiability but destroys global Euclidean regularity near the boundary, so the unconstrained \(L_2\)-based first-order analysis no longer applies. We replace this with a local Dikin-geometric analysis: inside suitable barrier-induced neighborhoods, the barrier-metric provides an oracle-free proxy for the unknown curvature at the moving lower centers, and restores local regularity conditions. We further prove explicit \textit{barrier-aware} schedule conditions that keep the iterates inside these well-behaved regions.
    \item \textbf{Numerical validation.}
    We evaluate our method on congestion-toll design and constrained MNIST hypercleaning
    benchmarks. The experiments show that the barrier-metric updates improve stability near active constraints and achieve competitive solution quality with substantially lower per-update cost than exact or implicit-gradient baselines.
\end{itemize}
\subsection{Related Work}

Bilevel optimization has a long history in mathematical programming; see
\citet{colson2007overview} for a broad overview. Classical approaches include approximation
schemes, penalty and value-function reformulations, and implicit differentiation through lower-level
optimality conditions. Modern non-asymptotic analysis for smooth bilevel optimization with a
strongly convex lower problem was initiated by \citet{ghadimi2018approximation}, and was
subsequently refined through approximate implicit differentiation, iterative differentiation,
single-loop schemes, and variance-reduction or momentum techniques
\citep{ji2021bilevel,hong2023two,chen2021closing,chen2022single,dagreou2022framework,
khanduri2021near,yang2021provably,yang2023achieving}. These methods provide important
stationarity guarantees, but typically rely on Hessian-vector products, inverse-Hessian
approximations, or differentiating through the lower-level solution procedure.
A recent line of work seeks to avoid second-order information altogether. In the unconstrained
setting, \citet{liu2022bome} propose a simple first-order penalty method, and
\citet{kwon2023fully} develop a fully first-order stochastic bilevel method with explicit rates in
deterministic, upper-level-stochastic, and fully stochastic regimes. 

Bilevel optimization with lower-level constraints has also received substantial recent attention.
For linearly constrained lower problems, \citet{khanduri2023linearly} propose a smoothed implicit
gradient method, while \citet{kornowski2024first} develop first-order methods with finite-time
Goldstein-stationarity guarantees. Related stochastic first-order developments include
\citet{phan2025bridging} and \citet{shen2026a}, and other constrained bilevel works use
proximal, penalty, or Lagrangian-value-function reformulations to target KKT-type or nonsmooth
stationarity notions \citep{yao2024constrained,jiang2024primal}. Among existing approaches,
barrier and smoothing ideas are especially close to ours: \citet{tsaknakis2023implicit} and
\citet{jiang2024barrier} show that barrier reformulations can restore differentiability in constrained
bilevel problems, but rely on implicit-gradient/KKT-style sensitivity information or adaptive barrier
frameworks. We instead combine barrier smoothing with the proxy-gradient philosophy of
\citet{kwon2023fully}. Their analysis relies on global Euclidean smoothness in the unconstrained lower problem; after barrier smoothing,
those Euclidean constants are not uniformly controlled near the boundary, therefore we develop a local Dikin-geometry analysis that provides
near-boundary curvature control and certifies \textit{barrier-aware} schedules for our first-order method.

\section{Problem Formulation and Algorithm}\label{sec:barrier smoothing}
Consider the constrained bilevel problem~\eqref{eq:bi-level_constrained}. Throughout the paper, we
focus on lower-level constraints and take the upper variable to be unconstrained:
\(X=\mathbb R^{d_x}\). The lower feasible set is a fixed polytope \( Y=\{y\in\mathbb R^{d_y}:Ay\le b\}, \)
independent of \(x\), and we assume that \(Y\) is compact, full-dimensional, and has nonempty
interior. We also assume that the original outer objective is bounded below:
\(F^{\star}:=\inf_{x\in X}F(x)>-\infty . \)
We impose the following standing regularity assumptions on \(f\) and \(g\).

\begin{assumption}[Objective regularity]
\label{ass:euclidean}
\leavevmode
\begin{enumerate}[leftmargin=*, label=\textup{(A\arabic*)}, ref=\textup{(A\arabic*)}]
    \item\label{ass:A1} \(f\) and \(g\) are jointly smooth in \((x,y)\), with constants
    \(\ell_{f,1}\) and \(\ell_{g,1}\).
    \item\label{ass:A2} For every \(x\), the map \(y\mapsto g(x,y)\) is \(\rho_g\)-strongly convex for some \(\rho_g>0\).
    \item\label{ass:A3} \(f,g\in C^2\), and \(\nabla^2 f,\nabla^2 g\) are jointly Lipschitz in
    \((x,y)\), with constants \(\ell_{f,2}\) and \(\ell_{g,2}\).
    \item\label{ass:A4} The first derivatives used in lower-level and outer analysis are bounded: for all
    \((x,y)\),
    \(\|\nabla_x f(x,y)\|_2\vee \|\nabla_y f(x,y)\|_2\le \ell_{f,0}\),
    \(\|\nabla_x g(x,y)\|_2\le \ell_{g,0}\).
\end{enumerate}
\end{assumption}
Under the strong convexity condition in \ref{ass:A2}, the lower minimizer \(y^\star(x)\) is unique
for every \(x\), so the lower solution map and the induced upper objective are well defined.

\subsection{Barrier Smoothing}
The main difficulty of the constrained formulation is that the lower solution map \(x\mapsto y^\star(x)\) may fail to be differentiable because of active-set changes in \(Y\). To recover a differentiable surrogate problem, we replace the constrained lower problem by an interior barrier approximation. Since \(Y\) is polyhedral with nonempty interior, we introduce the logarithmic barrier
\(\phi(y):=-\sum_{i=1}^m \log\bigl(b_i-a_i^\top y\bigr)\)
and define the barrierized lower objective \(\psi_\mu(x,y):=g(x,y)+\mu\phi(y), y\in\operatorname{int}(Y).\)
This gives the barrier-smoothed lower minimizer \(y_\mu^\star(x)\in\arg\min_{y\in\operatorname{int}(Y)}\psi_\mu(x,y),\)
and the corresponding barrier-smoothed outer objective \(F_\mu(x):=f(x,y_\mu^\star(x)).\)
The barrier surrogate restores differentiability of the lower solution map and therefore of the induced outer objective.

\begin{theorem}
\label{thm:barrier_smoothing}
Suppose Assumption~\ref{ass:euclidean} holds. Then, for every \(x\),
\begin{enumerate}[leftmargin=*]
    \item [\textup{(i)}] the barrierized lower problem has a unique minimizer
    \(y_\mu^\star(x)\in\operatorname{int}(Y)\);

    \item [\textup{(ii)}] the map \(x\mapsto y_\mu^\star(x)\) is differentiable;

    \item [\textup{(iii)}] the barrier-smoothed outer objective \(F_\mu(x)=f(x,y_\mu^\star(x))\) is differentiable,
    with
    \[
    \nabla F_\mu(x)
    =
    \nabla_x f(x,y_\mu^\star(x))
    -
    \nabla_{xy}^2\psi_\mu(x,y_\mu^\star(x))^\top
    \bigl(\nabla_{yy}^2\psi_\mu(x,y_\mu^\star(x))\bigr)^{-1}
    \nabla_y f(x,y_\mu^\star(x));
    \]

    \item [\textup{(iv)}] the barrier-smoothed model is a consistent surrogate for the original constrained lower
    problem. Namely, if \(y^\star(x)\) denotes the minimizer of the original constrained lower
    problem, then there exist finite constants \(C_g,C_y,C_F>0\) such that
    \[
    0 \le g(x,y_\mu^\star(x))-g(x,y^\star(x))\le C_g\mu,
    \]
    \[
    \|y_\mu^\star(x)-y^\star(x)\|_2\le C_y\sqrt{\mu}, \qquad |F_\mu(x)-F(x)|\le C_F\sqrt{\mu}
    \]
\end{enumerate}
\end{theorem}

Theorem~\ref{thm:barrier_smoothing} shows that barrier smoothing restores the differentiability
needed for gradient-based optimization while remaining faithful to the original constrained bilevel
problem: as \(\mu\downarrow0\), the barrierized lower solution and the induced outer objective
approach their constrained counterparts.

\subsection{A Barrier-Metric First-Order Algorithm}

Although Theorem~\ref{thm:barrier_smoothing} makes \(F_\mu\) differentiable, directly using
\(\nabla F_\mu(x)\) is still computationally unattractive. Computing the hypergradient requires an
accurate solution of the barrierized lower problem and the sensitivity formula involves mixed
derivatives together with an inverse barrierized lower Hessian. We therefore avoid direct
hypergradient computation and instead introduce a first-order proxy-gradient construction, following
the proxy-tracking philosophy of \citet{kwon2023fully}.

We introduce the barrierized lower value function $\psi_\mu^\star(x)
:=
\min_{y\in\operatorname{int}(Y)} \psi_\mu(x,y),$
and define the penalized first-order surrogate
\begin{equation}\label{eq:L_lam_mu}
L_{\lambda,\mu}(x,y)
:=
f(x,y)+\lambda\bigl(\psi_\mu(x,y)-\psi_\mu^\star(x)\bigr),
\qquad y\in\operatorname{int}(Y), \notag
\end{equation}
together with its minimizer:
\(y_{\lambda,\mu}^\star(x)
\in
\arg\min_{y\in\operatorname{int}(Y)} L_{\lambda,\mu}(x,y).\)

The point \(y_{\lambda,\mu}^\star(x)\) serves as a first-order surrogate for the exact barrierized
response \(y_\mu^\star(x)\). The key point is that \(L_{\lambda,\mu}\) admits simple first-order derivatives: $\nabla_y L_{\lambda,\mu}(x,y)
=
\nabla_y f(x,y)+\lambda \nabla_y \psi_\mu(x,y),$
since \(\psi_\mu^\star(x)\) does not depend on \(y\). Likewise, if \(z\) tracks the exact lower
minimizer and \(y\) tracks the proxy minimizer, a natural first-order outer direction is
\begin{equation}\label{eq:qx_def_intro}
q^x(x;y,z,\lambda)
:=
\nabla_x f(x,y)+\lambda\bigl(\nabla_x\psi_\mu(x,y)-\nabla_x\psi_\mu(x,z)\bigr). \notag
\end{equation}
Because the barrier \(\phi\) is independent of \(x\), we have $\nabla_x\psi_\mu(x,y)=\nabla_x g(x,y),$ and therefore $q^x(x;y,z,\lambda)
=
\nabla_x f(x,y)+\lambda\bigl(\nabla_x g(x,y)-\nabla_x g(x,z)\bigr).$
Thus both the inner and outer updates can be formed using only first-order information of
\(f\) and \(g\).

Algorithm~\ref{alg:dikin_sbo} implements this construction with two lower trackers. At outer
iteration \(k\), the variable \(z_k\) tracks the exact barrierized lower minimizer \(y_\mu^\star(x_k)\)
of \(\psi_\mu(x_k,\cdot)\), while \(y_k\) tracks the surrogate minimizer
\(y_{\lambda_k,\mu}^\star(x_k)\) of \(L_{\lambda_k,\mu}(x_k,\cdot)\). The multiplier
\(\lambda_k\) reduces the bias between the first-order proxy direction \(q_k^x\) and the true
hypergradient \(\nabla F_\mu(x_k)\). The inner-loop length \(T\) determines the contraction budget
of the two trackers, while the ratio \(\xi\) controls how far the outer iterate moves in one outer
round. In the analysis we also write \(\beta_k:=\alpha_k\lambda_k\) for the effective step size of the
\(y\)-tracker. Relative to the unconstrained proxy-gradient method of \citet{kwon2023fully}, the main difference is the lower-tracking geometry. Instead
of Euclidean gradient steps, the two inner trackers use logarithmic barrier-metrics
\(\nabla^2\phi(z_k)\) and \(\nabla^2\phi(y_k)\), held fixed during the \(T\) inner steps of outer
iteration \(k\). These metrics are explicit from the polyhedral constraints. For
\(Y=\{y:Ay\le b\}\), with slacks \(s(y):=b-Ay\), we have
\(\nabla\phi(y)=A^\top s(y)^{-1},
\nabla^2\phi(y)=A^\top\operatorname{Diag}(s(y)^{-2})A .\)
Thus the lower-tracker preconditioners use only the current slacks and the known constraint matrix
\(A\), rather than Hessians of \(f\) or \(g\), lower-Hessian inversions, or implicit differentiation.
For box constraints this metric is diagonal; for a general polytope, the preconditioned step can be
computed efficiently by solving a linear system with
\(A^\top\operatorname{Diag}(s(y)^{-2})A\), or by applying this matrix through \(A\) and \(A^\top\).
This explicit preconditioner makes the lower updates implementable without problem Hessians, but it
also changes the scale on which step sizes must be certified. In the unconstrained setting, global
Euclidean smoothness constants certify the lower, proxy, and outer step sizes. After barrier
smoothing, these global Euclidean scales are no longer reliable near the boundary.
Section~\ref{sec:dikin_geometry} develops the local Dikin geometry and \textit{barrier-aware} schedules
needed to keep the two trackers in well-behaved regions.

\begin{algorithm}[h]
\caption{Barrier-Metric First-Order Algorithm}
\label{alg:dikin_sbo}
\begin{algorithmic}[1]
\Require iteration budget \(K\), inner-loop length \(T\), outer ratio \(\xi>0\),
initial point \(x_0\in\mathbb R^{d_x}\), interior initializations \(y_0,z_0\in\operatorname{int}(Y)\),
initial multiplier \(\lambda_0>0\), schedules \(\{\lambda_k\}_{k\ge 0}\), \(\{\alpha_k\}_{k\ge 0}\), \(\{\gamma_k\}_{k\ge 0}\),
\(\{\delta_k\}_{k\ge 0}\)
\For{$k=0,1,\dots,K-1$}
    \State \(z_k^{(0)} \gets z_k,\quad y_k^{(0)} \gets y_k\)
    \For{$t=0,1,\dots,T-1$}
        \State \(z_k^{(t+1)} \gets z_k^{(t)} - \gamma_k  \nabla^2\phi(z_k)^{-1} \nabla_y \psi_\mu\!\left(x_k,z_k^{(t)}\right)\)
        \State \(y_k^{(t+1)} \gets y_k^{(t)} - \alpha_k \nabla^2\phi(y_k)^{-1} \nabla_y L_{\lambda_k,\mu}\!\left(x_k,y_k^{(t)}\right)\)
        % \Comment{$\nabla_y L_{\lambda_k,\mu}(x_k,y)=\nabla_y f(x_k,y)+\lambda_k \nabla_y\psi_\mu(x_k,y)$}
    \EndFor
    \State \(z_{k+1} \gets z_k^{(T)},\qquad y_{k+1} \gets y_k^{(T)}\)
    \State \(q_k^x \gets \nabla_x f(x_k,y_{k+1}) + \lambda_k\!\left(\nabla_x\psi_\mu(x_k,y_{k+1})-\nabla_x\psi_\mu(x_k,z_{k+1})\right)\)
    \State \(x_{k+1} \gets x_k - \xi\alpha_k\, q_k^x\)
    \State \(\lambda_{k+1} \gets \lambda_k + \delta_k\)
\EndFor
\State \Return \(\{x_k,y_k,z_k,\lambda_k\}_{k=0}^{K}\)
\end{algorithmic}
\end{algorithm}

\section{Convergence Analysis under Dikin Geometry}
\label{sec:dikin_geometry}

Section~\ref{sec:barrier smoothing} introduced the barrier-smoothed surrogate and
Algorithm~\ref{alg:dikin_sbo}. The remaining difficulty is geometric. Barrier smoothing restores
differentiability, but the logarithmic barrier creates boundary-induced curvature. Indeed, for
\(\phi(y)=-\sum_{i=1}^m\log s_i(y)\), \(s_i(y):=b_i-a_i^\top y\), we have
\(\nabla^2\phi(y)=\sum_{i=1}^m a_i a_i^\top/s_i(y)^2\). Hence, as \(y\) approaches an active face,
the curvature in the corresponding boundary-normal direction can become arbitrarily large. As a
result, the barrierized lower objective, and the proxy objective built from it, need not satisfy the
global Euclidean regularity used in unconstrained first-order bilevel analysis; in particular, global
Euclidean estimates for minimizer stability, outer smoothness, and proxy-gradient bias are no longer
available. We replace these global Euclidean estimates by local Dikin regularity. The Dikin metric induced by
the barrier Hessian gives the natural local scale near the boundary: directions that look moderate in
Euclidean norm can have large boundary-normal curvature when some slack is small, and the Dikin
norm records this curvature. If the exact or proxy lower center were known, the barrier Hessian at that center would set the natural local step-size scale. In our setting these centers are unknown and move with \(x\) and \(\lambda\); however, inside suitable Dikin neighborhoods the barrier geometry at the current iterate is spectrally comparable to the geometry at the center. Thus the barrier-metric provides an oracle-free local
curvature proxy for certifying first-order schedules.  Figure~\ref{fig:geometric_illustration} illustrates our approach's mechanism. The analysis in the remainder of this section follows this geometry. We first establish local Dikin estimates for curvature,
smoothness, minimizer stability, and proxy-gradient bias. We then use these estimates to design
\textit{barrier-aware} schedules under which the tracker contraction dominates center drift, keeping the
iterates inside the Dikin neighborhoods where the local estimates are valid. This closes the local
analysis and yields the deterministic convergence guarantee.

\begin{figure}[h]
    \centering
    \includegraphics[width=.8\linewidth]{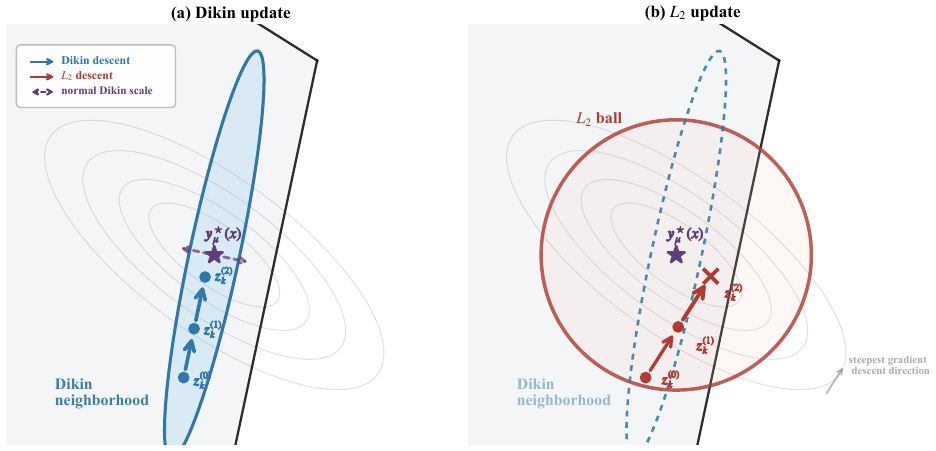}
    \caption{
    Geometric illustration.
    (a) The Dikin update is scaled by the local barrier-metric, which contracts motion in the
    boundary-normal direction as the center approaches an active face.
    (b) The Euclidean \(L_2\) update is isotropic and can move too far in the boundary-normal direction,
    causing the iterate to leave the Dikin neighborhood and approach the boundary, where feasibility and
    step-size control can break down.
    }
\label{fig:geometric_illustration}
\end{figure}

\subsection{Dikin Geometry}
\label{subsec:dikin_local}

\paragraph{Notations.} For any \(x\in\mathbb R^{d_x}\), we attach two barrier-metrics:
\(H_\mu^\star(x):=\nabla^2\phi(y_\mu^\star(x))\) at the exact barrierized minimizer and
\(H_{\lambda,\mu}^\star(x):=\nabla^2\phi(y_{\lambda,\mu}^\star(x))\) at the proxy minimizer. For an anchor \(c\in\{y_\mu^\star(x),y_{\lambda,\mu}^\star(x)\}\), let \(H_c=\nabla^2\phi(c)\), and define
\(\|u\|_c^2:=\langle H_cu,u\rangle\) and
\(\|w\|_{c,*}^2:=\langle H_c^{-1}w,w\rangle\). For a linear map \(M\), we use the induced norms
\(\|M\|_{2\to c,*}:=\sup_{\|u\|_2=1}\|Mu\|_{c,*}\) and
\(\|M\|_{c\to c,*}:=\sup_{\|u\|_c=1}\|Mu\|_{c,*}\), and analogously for other norm pairs. Fix a radius parameter \(\eta\in(0,1/2)\), at a fixed outer point \(x\), we define the exact and proxy Dikin neighborhoods
\(T_{\eta,\mu}(x):=\{y\in\operatorname{int}(Y):\|y-y_\mu^\star(x)\|_{y_\mu^\star(x)}\le \eta\}\) and
\(T_{\eta,\mu}^{\lambda}(x):=\{y\in\operatorname{int}(Y):\|y-y_{\lambda,\mu}^\star(x)\|_{y_{\lambda,\mu}^\star(x)}\le \eta\}\).
We will also use the enlarged buffer neighborhoods \(T_{2\eta,\mu}(x)\) and
\(T_{2\eta,\mu}^{\lambda}(x)\), obtained by replacing \(\eta\) with \(2\eta\).
The exact neighborhood \(T_{\eta,\mu}(x)\) is the local region used for the \(z\)-tracker analysis, while the
proxy neighborhood \(T_{\eta,\mu}^{\lambda}(x)\) is used for the \(y\)-tracker analysis. The target neighborhoods are the regions in which the trackers will be maintained; the \(2\eta\)-buffer neighborhoods provide room for anchor changes when the centers
move between outer iterations. Along the algorithmic trajectory \(\{x_k,\lambda_k\}\), the moving families \(\{T_{\eta,\mu}(x_k)\}_k\) and \(\{T_{\eta,\mu}^{\lambda_k}(x_k)\}_k\) form the exact and proxy Dikin tubes. The results in this subsection are static local-geometry
statements: for each fixed \(x\), we derive Dikin regularity around the exact barrierized center and
transfer it to the proxy center when \(\lambda\) is sufficiently large. These estimates do not by
themselves imply that the algorithmic iterates stay in the corresponding neighborhoods. That closure is proved in Section~\ref{subsec:tube_and_conv}, where a \textit{barrier-aware} schedule is shown to
keep the trackers inside the target neighborhoods. With this closure deferred, we now derive the local Dikin regularity needed for the analysis.

\begin{proposition}[Local Dikin regularity]
\label{prop:derived_dikin_regularity}
Suppose Assumption~\ref{ass:euclidean} holds. Fix a neighborhood radius
\(\eta\in(0,1/2)\). Then there exist finite constants
\(\rho_{\psi}^{\eta},\rho_{\psi,\lambda}^{\eta}>0\) and
\(\ell_{\psi,1}^{\eta},\ell_{\psi,2}^{\eta},
\ell_{\psi,\lambda,1}^{\eta},\ell_{\psi,\lambda,2}^{\eta},
\ell_{f,0}^{\eta},\ell_{f,1}^{\eta},\ell_{f,2}^{\eta}<\infty\),
such that, for every \(x\in\mathbb R^{d_x}\) the following hold.
\begin{enumerate}[leftmargin=*]
\item[\textup{(i)}]
On the neighborhood \(T_{\eta,\mu}(x)\), the barrierized lower objective is locally well conditioned in
the anchored Dikin metric:
\(\rho_{\psi}^{\eta}H_\mu^\star(x)\preceq
\nabla_{yy}^2\psi_\mu(x,y)\preceq
\ell_{\psi,1}^{\eta}H_\mu^\star(x)\) for every \(y\in T_{\eta,\mu}(x)\).
Moreover, \(\nabla_{xy}^2\psi_\mu\) is locally bounded, and
\(\nabla_{yy}^2\psi_\mu,\nabla_{xy}^2\psi_\mu\) are locally Lipschitz on
\(T_{\eta,\mu}(x)\) in the anchored Dikin norms, with constants
\(\ell_{\psi,1}^{\eta},\ell_{\psi,2}^{\eta}\).

\item[\textup{(ii)}]
The upper-level derivatives of \(f\) are locally controlled on \(T_{\eta,\mu}(x)\):
\(\nabla_y f\), \(\nabla_{xy}^2 f\), and \(\nabla_{yy}^2 f\) are locally bounded, and the
corresponding Hessian terms are locally Lipschitz, in the anchored Dikin norms, with constants
\(\ell_{f,0}^{\eta},\ell_{f,1}^{\eta},\ell_{f,2}^{\eta}\).

\item[\textup{(iii)}]
On the proxy neighborhood \(T_{\eta,\mu}^{\lambda}(x)\), the barrierized lower objective inherits the same
type of local Dikin regularity with respect to the proxy anchor \(H_{\lambda,\mu}^\star(x)\):
\(\rho_{\psi,\lambda}^{\eta}H_{\lambda,\mu}^\star(x)\preceq
\nabla_{yy}^2\psi_\mu(x,y)\preceq
\ell_{\psi,\lambda,1}^{\eta}H_{\lambda,\mu}^\star(x)\) for every
\(y\in T_{\eta,\mu}^{\lambda}(x)\), together with the analogous mixed-Hessian and
Lipschitz-Hessian bounds in the proxy anchored Dikin norm, with constants
\(\ell_{\psi,\lambda,1}^{\eta},\ell_{\psi,\lambda,2}^{\eta}\).
\end{enumerate}
\end{proposition}

The fully expanded statement, including all mixed-Hessian and Lipschitz-Hessian bounds in both the
exact and proxy anchored Dikin norms, is given in Appendix~\ref{app:dikin_regularity}.
Proposition~\ref{prop:derived_dikin_regularity} is the local barrier-geometry substitute for the
global Euclidean regularity used in unconstrained first-order bilevel analysis. By self-concordant
metric change, if \(y\in T_{\eta,\mu}(x)\), then
\((1-\eta)^2H_\mu^\star(x)\preceq\nabla^2\phi(y)\preceq(1-\eta)^{-2}H_\mu^\star(x)\), and an
analogous comparison holds on the proxy neighborhood with anchor \(H_{\lambda,\mu}^\star(x)\). Thus,
after normalization by the local barrier-metric, the barrierized lower problem recovers controlled
curvature and smoothness inside the exact and proxy neighborhoods. In this sense, the Dikin metric provides
an oracle-free local curvature scale for the moving lower centers, replacing the global Euclidean
smoothness constants available in the unconstrained setting. These local estimates are the ingredients needed for the first-order analysis. In particular,
Proposition~\ref{prop:local_consequences} below uses them to establish stability of the exact and
proxy minimizer maps, local smoothness of the barrier-smoothed objective \(F_\mu\), and the
\(O(1/\lambda)\) proxy-gradient bias bound:

\begin{proposition}[Local consequences of Dikin regularity]
\label{prop:local_consequences}
Suppose Assumption~\ref{ass:euclidean} holds and let \(\eta\) be as in
Proposition~\ref{prop:derived_dikin_regularity}. Then there exist finite constants
\(\ell_{\ast,0}^\eta,\ell_{\lambda,0}^\eta,L_F^\eta,c_x^\eta\) such that the following hold.

\begin{enumerate}[leftmargin=*]
    \item[\textup{(i)}] \textbf{Local stability of the proxy minimizer map.}
    For any \(x_1,x_2\in X\) and  \(\lambda_1,\lambda_2\ge \lambda_0\),
    \begin{equation}\label{eq:prop_local_proxy}
    \|y_{\lambda_2,\mu}^\star(x_2)-y_{\lambda_1,\mu}^\star(x_1)\|_{y_{\lambda_1,\mu}^\star(x_1)}
    \le
    \frac{2\ell_{f,0}^\eta}{\rho_{\psi}^\eta}
    \left|\frac{1}{\lambda_2}-\frac{1}{\lambda_1}\right|
    +
    \ell_{\lambda,0}^\eta \|x_2-x_1\|_2. \notag
    \end{equation}
    \item[\textup{(ii)}] \textbf{Local stability of the exact minimizer map.}
    For any \(x_1,x_2\in X\),
    \begin{equation}\label{eq:prop_local_exact}
    \|y_\mu^\star(x_2)-y_\mu^\star(x_1)\|_{y_\mu^\star(x_1)}
    \le
    \ell_{\ast,0}^\eta \|x_2-x_1\|_2. \notag
    \end{equation}

    \item[\textup{(iii)}] \textbf{Local smoothness of the barrier-smoothed outer objective.}
    The barrier-smoothed outer objective \(F_\mu(x)=f(x,y_\mu^\star(x))\) is differentiable on \(X\), and for any \(x_1,x_2\in X\),
    \begin{equation}\label{eq:prop_local_smooth}
    \|\nabla F_\mu(x_2)-\nabla F_\mu(x_1)\|_2
    \le
    L_F^\eta \|x_2-x_1\|_2. \notag
    \end{equation}

    \item[\textup{(iv)}] \textbf{Local proxy-gradient bias.}
    Let \(C_{\lambda,\mu}^\star(x):=\min_{y\in\operatorname{int}(Y)} L_{\lambda,\mu}(x,y)\).
    Then for every \(x\in X\) and every \(\lambda\ge \lambda_0\),
    \begin{equation}\label{eq:prop_local_bias}
    \|\nabla F_\mu(x)-\nabla C_{\lambda,\mu}^\star(x)\|_2
    \le
    \frac{c_x^\eta}{\lambda}. \notag
    \end{equation}
\end{enumerate}
\end{proposition}

\subsection{Barrier-aware Schedules and Convergence Analysis}
\label{subsec:tube_and_conv}

\newcommand{\Cd}{(\ell_{f,0}/\lambda_0+2\ell_{g,0})}

The local estimates from Section~\ref{sec:dikin_geometry} are conditional: they hold only while the
exact and proxy trackers remain inside the corresponding Dikin neighborhoods. The remaining task is
therefore algorithmic closure. We need schedules under which the local contraction generated by the
two inner trackers is strong enough to preserve these neighborhoods throughout the run. This is the
role of the \textit{barrier-aware} schedule introduced below.

\paragraph{Schedule design principle.}
The local Dikin regularity from Proposition~\ref{prop:derived_dikin_regularity} turns curvature into
fixed-center contraction: the exact tracker contracts on the scale \(\rho_{\psi}^\eta\gamma_k\),
while the proxy tracker contracts on the effective scale \(\rho_{\psi}^\eta\beta_k\), where
\(\beta_k:=\alpha_k\lambda_k\). These contractions must compensate for the fact that the exact
center changes with \(x_k\), while the proxy center changes with both \(x_k\) and \(\lambda_k\).
Thus the schedule must simultaneously enforce inner contraction, slow multiplier growth, and
controlled outer motion so that center drift and self-concordant anchor-switch effects remain
dominated. The following \textit{barrier-aware} step-size conditions are explicit sufficient conditions for this
closure:
\begin{definition}[\textit{Barrier-aware} schedule]
\label{def:barrier_aware}
Let \(\beta_k:=\alpha_k\lambda_k\). We say that
\(\{\alpha_k,\gamma_k,\lambda_k,\delta_k\}_{k\ge0}\) is \textbf{barrier-aware} if, for all \(k\),
\begin{align}
\lambda_0 &\ge
\max\left\{
\frac{2\ell_{f,1}^{\eta}}{\rho_{\psi}^{\eta}},
\frac{8\ell_{f,0}^{\eta}}{\eta\rho_{\psi}^{\eta}}
\right\},
\qquad
\beta_k\le \gamma_k\le
\min\left\{
\frac{1}{4\ell_{\psi,1}^{\eta}},
\frac{1}{4T\rho_{\psi}^{\eta}}
\right\},
\qquad
\alpha_k\le \frac{1}{2\xi L_F^{\eta}},
\tag{S1}\label{eq:main_sched_s1}\\
\frac{\delta_k}{\lambda_k}
&\le
\frac{T\rho_{\psi}^{\eta}}{16}\beta_k,
\qquad
\delta_k:=\lambda_{k+1}-\lambda_k\ge0,
\tag{S2}\label{eq:main_sched_s2}\\
\tfrac{\xi}{T}
&\le
\min\Biggl\{
\tfrac{\eta\rho_{\psi}^{\eta}}{16\ell_{*,0}^{\eta}\Cd},\,
\tfrac{\eta\rho_{\psi}^{\eta}}{64\ell_{\lambda,0}^{\eta}\Cd},\,
\,\tfrac{c_\xi\rho_\psi^\eta}
{\max\{\ell_{\psi,1}^\eta(\ell_{*,0}^\eta)^2,\,\ell_{*,1}^\eta\sqrt{\max(\ell_{g,0}^2,\ell_{f,0}^2)}\}}
\Biggr\}.
\tag{S3}\label{eq:main_sched_s3}
\end{align}
Here \(c_\xi>0\) is a sufficiently small absolute numerical constant determined by the
outer-descent and tracker-recursion analyses (cf.\ Proposition~\ref{prop:outer_descent} and
Appendix~\ref{app:tracker}).
\end{definition}
Condition~\eqref{eq:main_sched_s1} combines proxy initialization, fixed-center contraction, and
outer descent compatibility. The lower bound on \(\lambda_0\) makes the proxy objective locally
well conditioned and keeps the proxy center close enough to the exact center for proxy neighborhood transfer;
\(\gamma_k\) controls the exact tracker, \(\beta_k=\alpha_k\lambda_k\) controls the proxy tracker, and
\(\alpha_k\) controls the outer descent step. Condition~\eqref{eq:main_sched_s2} controls multiplier
growth, ensuring that \(\lambda_k\) does not change faster than the proxy tracker can adapt.
Condition~\eqref{eq:main_sched_s3} is the tube-closure condition: its first two terms control the
motion of the exact and proxy centers, while its last term controls self-concordant anchor-switch
effects. Here \(\ell_{*,1}^{\eta}\) denotes the local Dikin Lipschitz constant of \(x\mapsto y_\mu^\star(x)\), used in the anchor-switch term. Together, \eqref{eq:main_sched_s1}-\eqref{eq:main_sched_s3} ensure that inner-loop contraction dominates center drift and
keeps the iterates inside the local Dikin neighborhoods. Appendix~\ref{app:tube_invariance} derives
these sufficient conditions from the one-step tracker recursions in Appendix~\ref{app:tracker}. The following theorem formalizes this closure step: once the warm starts are initialized inside the
target neighborhoods, a \textit{barrier-aware} schedule keeps all exact and proxy tracker iterates inside their
corresponding neighborhoods, so the local estimates from Proposition~\ref{prop:derived_dikin_regularity}
remain valid along the full trajectory.

\begin{theorem}[Tube invariance]
\label{thm:tube_invariance}
Suppose Assumption~\ref{ass:euclidean} holds, fix an integer \(T\ge 1\), and let
\(\{\alpha_k,\gamma_k,\lambda_k,\delta_k\}_{k\ge 0}\) be \textit{barrier-aware} in the sense of
Definition~\ref{def:barrier_aware}. If the initial warm starts satisfy
\(z_0\in T_{\eta,\mu}(x_0)\) and \(y_0\in T_{\eta,\mu}^{\lambda_0}(x_0)\),
then for every outer iteration \(k\), all inner iterates of the exact tracker remain in
\(T_{\eta,\mu}(x_k)\), all inner iterates of the proxy tracker remain in
\(T_{\eta,\mu}^{\lambda_k}(x_k)\), and the next warm starts satisfy
\(z_{k+1}\in T_{\eta,\mu}(x_{k+1})\) and \(y_{k+1}\in T_{\eta,\mu}^{\lambda_{k+1}}(x_{k+1})\).
In particular, the moving exact and proxy Dikin neighborhoods are forward invariant.
\end{theorem}

Theorem~\ref{thm:tube_invariance} upgrades the local Dikin estimates to trajectory-wise guarantees:
under a \textit{barrier-aware} schedule, the exact and proxy trackers remain inside the target neighborhoods, so
Propositions~\ref{prop:derived_dikin_regularity}--\ref{prop:local_consequences} apply throughout
the run. With this invariant in place, a Lyapunov descent argument yields the deterministic
convergence guarantee below; the proof is deferred to Appendix~\ref{app:deterministic}.

\begin{theorem}[Deterministic convergence]
\label{thm:deterministic_poly}
Suppose Assumption~\ref{ass:euclidean} holds, \(\inf_{x\in\mathbb R^{d_x}}F_\mu(x)>-\infty\),
the inner-loop length \(T\ge 1\) is a fixed integer, the initial warm starts satisfy
\(z_0\in T_{\eta,\mu}(x_0)\) and \(y_0\in T_{\eta,\mu}^{\lambda_0}(x_0)\), and the schedule
\(\{\alpha_k,\gamma_k,\lambda_k,\delta_k\}_{k\ge 0}\) is \textit{barrier-aware}
(Definition~\ref{def:barrier_aware}) with the polynomial profile
\begin{equation}
\label{eq:main_poly_schedule}
\alpha_k=\frac{\alpha_0}{(k+k_0)^{1/3}},
\qquad
\gamma_k=\gamma_0,
\qquad
\lambda_k=\lambda_0\Bigl(\frac{k+k_0}{k_0}\Bigr)^{1/3},
\qquad
\delta_k=\lambda_{k+1}-\lambda_k, \notag
\end{equation}
where the constants \(\alpha_0,\gamma_0,\lambda_0,k_0, \xi, T\) are chosen so that
\eqref{eq:main_sched_s1}--\eqref{eq:main_sched_s3} hold. Then the exact and proxy target neighborhoods are
forward invariant, and the iterates generated by Algorithm~\ref{alg:dikin_sbo} satisfy
\(\min_{0\le k\le K-1}\|\nabla F_\mu(x_k)\|_2^2=\widetilde O(K^{-2/3})\).
Equivalently, Algorithm~\ref{alg:dikin_sbo} reaches an \(\epsilon\)-stationary point of the barrier-smoothed objective
\(F_\mu\) in \(\widetilde O(\epsilon^{-3/2})\) outer iterations.
\end{theorem}

% These are fixed-\(\mu\) rates for \(F_\mu\); since the local Dikin constants and admissible radius \(\eta\) may depend on \(\mu\), converting them to stationarity rates for the original nonsmoothed objective \(F\) requires additional analysis.
\begin{remark}[Practical schedule and initialization]
The \textit{barrier-aware} conditions \eqref{eq:main_sched_s1}--\eqref{eq:main_sched_s3} should be
interpreted as schedule-certification conditions rather than plug-in tuning rules. The polynomial
schedules in Theorem~\ref{thm:deterministic_poly} identify the scaling law for
\((\alpha_k,\gamma_k,\lambda_k,\delta_k)\), while
\eqref{eq:main_sched_s1}--\eqref{eq:main_sched_s3} specify the structural inequalities needed for
contraction, controlled center drift, and tube closure. Some quantities entering these conditions,
such as the geometry of \(Y\), the barrier parameter \(\mu\), and coarse Euclidean bounds on \(f\)
and \(g\), are explicit or can be estimated. The local Dikin constants around the moving exact and
proxy centers, such as \(\rho_{\psi}^{\eta},\ell_{\psi,1}^{\eta},L_F^{\eta},
\ell_{*,0}^{\eta}\), and \(\ell_{\lambda,0}^{\eta}\), are primarily analysis constants and are not directly
available online. Accordingly, in practice we use the theorem as a geometry-aware design principle:
we keep the polynomial schedule shape, tune the leading constants empirically, and preserve the
structural relations \(\beta_k=\alpha_k\lambda_k\le\gamma_k\), slow multiplier growth, and a small
outer-to-inner ratio \(\xi/T\). The warm-start requirement can be enforced offline by fixing
\(x=x_0\) and solving the exact and proxy barrierized lower problems to the target Dikin accuracy;
standard first-order methods for logarithmically homogeneous barriers, such as generalized
Frank--Wolfe schemes \citep{zhao2023analysis}, provide one such initialization route.
\end{remark}

\subsection{Stochastic Extension}
\label{sec:stoch_upper}

We also consider an upper-level stochastic variant in which \(f\) is accessed through unbiased
stochastic first-order oracles with bounded variance, while \(g\) and the barrierized lower objective
\(\psi_\mu\) remain deterministic. The algorithm is identical to Algorithm~\ref{alg:dikin_sbo}, except that
the \(f\)-gradients in the proxy tracker and outer direction are replaced by stochastic estimates.
The precise oracle assumptions and pathwise tube-maintenance conditions are stated in
Appendix~\ref{app:stochastic_proofs}.

\begin{theorem}[Stochastic convergence]
\label{thm:stoch_upper}
Suppose Assumption~\ref{ass:euclidean}, the lower-boundedness condition
\(\inf_{x\in\mathbb R^{d_x}}F_\mu(x)>-\infty\), and the stochastic oracle assumptions in
Appendix~\ref{app:stochastic_proofs} hold, the inner-loop length \(T\ge 1\) is a fixed integer,
the initial warm starts satisfy
\(z_0\in T_{\eta,\mu}(x_0)\) and \(y_0\in T_{\eta,\mu}^{\lambda_0}(x_0)\), and the schedule
\(\alpha_k=\alpha_0/(k+k_0)^{3/5}\), \(\gamma_k=\gamma_0/(k+k_0)^{2/5}\),
\(\lambda_k=\lambda_0\bigl((k+k_0)/k_0\bigr)^{1/5}\), \(\delta_k=\lambda_{k+1}-\lambda_k\),
is \textit{barrier-aware} in the deterministic sense and additionally satisfies the stochastic proxy tube-invariance condition stated in Appendix~\ref{app:stochastic_proofs}. Then the exact and proxy target neighborhoods remain forward invariant, and
\(\min_{0\le k\le K-1}\mathbb E\|\nabla F_\mu(x_k)\|_2^2=\widetilde O(K^{-2/5})\).
Equivalently, Algorithm~\ref{alg:dikin_sbo} reaches an \(\epsilon\)-stationary point of the
barrier-smoothed objective \(F_\mu\) in \(\widetilde O(\epsilon^{-5/2})\) outer iterations.
\end{theorem}

\section{Numerical Experiment}
\label{sec:experiments}
\paragraph{Motivating example.}
We first illustrate the local geometry behind the analysis on the exact barrierized lower problem.
Consider a synthetic two-dimensional polytope
\(Y=\{y\in\mathbb R^2:Ay\le b\}\), equipped with the logarithmic barrier
\(\phi(y)=-\sum_{i=1}^m\log s_i(y)\), where \(s_i(y)=b_i-a_i^\top y\). For each
outer parameter \(x\in[0,1]\), let
\(g(x,y)=\frac12(y-c(x))^\top Q(y-c(x))\), where \(Q\succ0\) and the path \(c(x)\) is chosen so
that the exact barrierized minimizer \(y_\mu^\star(x)\) moves toward an active face of \(Y\). We define \(\psi_\mu(x,y)=g(x,y)+\mu\phi(y)\), and consider the exact tracker update 
\(z_k^{(t+1)}
=
z_k^{(t)}
-
\gamma_k \nabla^2\phi(z_k)^{-1}\nabla_y\psi_\mu(x_k,z_k^{(t)}).\)
Figure~\ref{fig:motivating_example} illustrates the resulting Dikin tube invariance mechanism.
The Dikin neighborhood
\(T_{\eta,\mu}(x)=\{y\in\operatorname{int}(Y):\|y-y_\mu^\star(x)\|_{y_\mu^\star(x)}\le\eta\}\)
is shaped by the barrier Hessian at the exact center. As \(y_\mu^\star(x)\) approaches an active
face, the neighborhood becomes anisotropic and shrinks in the boundary-normal direction. The barrier
metric used by the tracker is locally comparable to this center metric inside the neighborhood, while the
\textit{barrier-aware} schedule keeps the tracker \(z_k\) inside these moving local regions.

\begin{figure}[h]
    \centering
    \includegraphics[width=\linewidth]{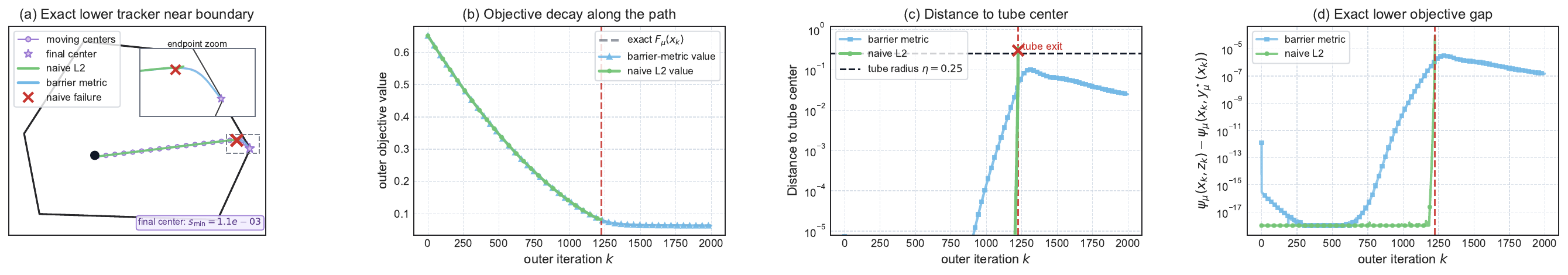}
    \caption{Boundary stability experiment for the exact lower tracker. 
(a) The exact barrierized center \(y_\mu^\star(x_k)\) moves toward an active face; the barrier-metric
tracker remains feasible while the Euclidean tracker exits near the boundary.
(b) Both methods decrease the exact barrier-smoothed objective before the Euclidean tracker exits.
(c) The anchored Dikin error shows that the barrier-metric tracker stays inside the tube, while the
Euclidean tracker crosses the tube radius.
(d) The exact lower objective gap remains controlled for the barrier-metric tracker after the boundary
event.
}
\label{fig:motivating_example}
\end{figure}

\paragraph{Congestion toll design.}
We consider a congestion-toll benchmark inspired by Wardrop equilibrium and Beckmann's
convex traffic-assignment formulation~\citep{wardrop1952road,beckmann1956studies}. The upper variable \(x\in\mathbb R^n\) parameterizes tolls on \(n\) corridors, and the lower variable \(y\in\mathbb R^n\)
denotes the induced corridor-flow response. The lower feasible region is
\(Y(\tau):=\{y\in\mathbb R^n:0\le y\le u,\ Cy\le \tau d,\ \mathbf 1^\top y\le D\}\), where \(u\)
gives corridor capacities, \(Cy\le \tau d\) encodes shared bottleneck capacities, and
\(\tau\in(0,1]\) controls tightness. Each row of \(C\) represents a shared bottleneck, \(C_{r,i}\) indicates whether corridor \(i\) uses
bottleneck \(r\), and \(d_r\) is the corresponding capacity. For fixed tolls \(x\), the regularized traffic-response model is
\(y^\star(x)\in\arg\min_{y\in Y(\tau)} g(x,y)\), with
\(g(x,y):=\ell^\top y+\frac12 y^\top Qy+x^\top y+\frac{\kappa}{2}(D-\mathbf 1^\top y)^2\).
Here \(\ell^\top y\) is baseline travel cost, \(\frac12y^\top Qy\) is a regularized congestion cost,
\(x^\top y\) is the toll payment, and the final term penalizes unmet demand. The toll authority
evaluates the induced response through
\(f(x,y):=\ell^\top y+\frac12y^\top Qy+\beta(D-\mathbf 1^\top y)^2
+\rho_{\mathrm{rev}}(x^\top y-R_{\mathrm{tar}})^2+\frac{\rho_x}{2}\|x\|_2^2\),
which balances system cost, demand coverage, target revenue, and toll regularization. We compare Algorithm~\ref{alg:dikin_sbo} (Barrier-Metric First-Order, BMFO) with four
baselines: an implicit-gradient barrier method (Exact-HG-Barrier), a differentiable convex-layer
hypergradient method based on \citet{agrawal2019differentiable} (CVXPYLayer), the smoothed
implicit-gradient method of \citet{khanduri2023linearly} (DSIGD), and the first-order constrained
bilevel method of \citet{kornowski2024first} (F2CBA). Figure~\ref{fig:toll_design} reports both
runtime and solution quality under a \(2\)-second optimization wall-clock budget. Panel~(a) shows
seconds per completed outer update. Panel~(b) reports the normalized
original-objective gap achieved within the same budget. BMFO has substantially lower update cost
at larger dimensions and remains competitive in objective gap, while several exact or implicit
baselines become too expensive to complete many outer updates under the fixed budget.
\begin{figure}[h]
    \centering
    \includegraphics[width=1\linewidth]{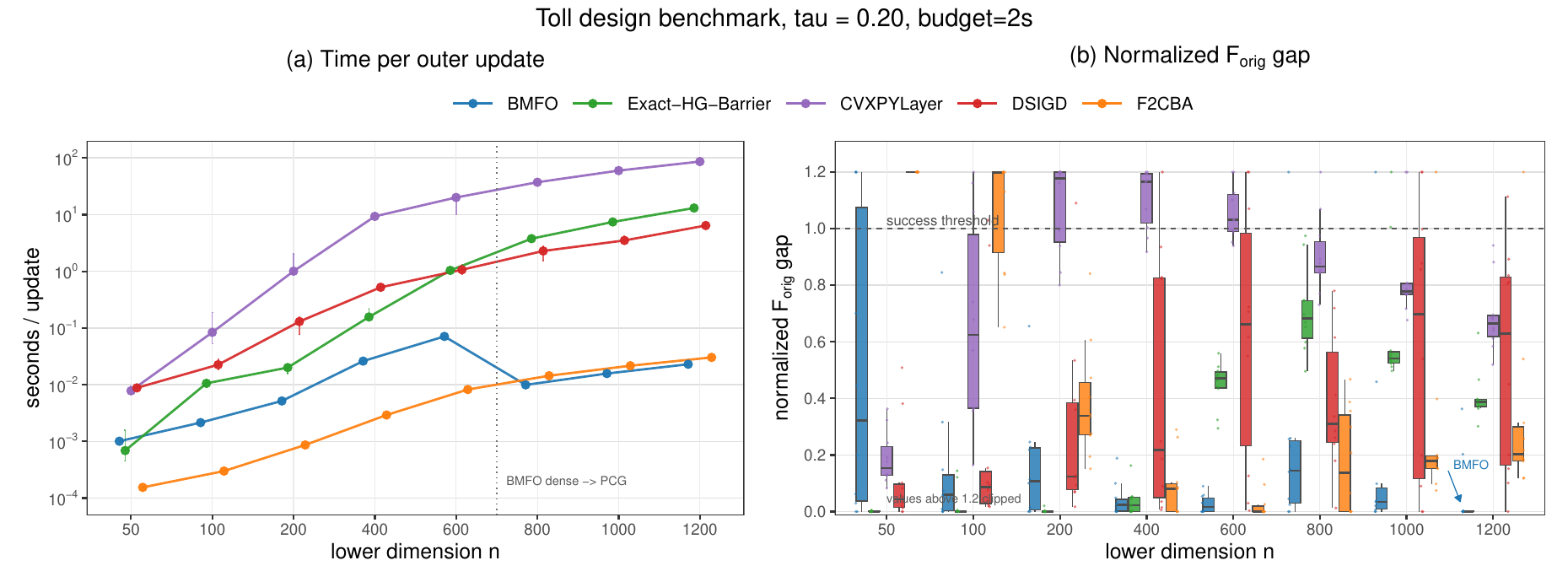}
\caption{
Congestion-toll benchmark with bottleneck tightness \(\tau=0.20\) and a \(2\)-second optimization
wall-clock budget per method. 
(a) Seconds per completed outer update, computed as total optimization time divided by the number
of completed outer updates. If one atomic
outer update exceeds the budget, the method may still appear with seconds/update \(>2\) because
the runner checks the budget only between completed outer updates.
(b) Normalized original-objective gap under the same wall-clock budget. Lower is better; values
above \(1.2\) are clipped for visualization, and budget-censored or one-update runs are included.
}
\label{fig:toll_design}
\end{figure}

\begin{figure}[h]
    \centering
    \includegraphics[width=.8\linewidth]{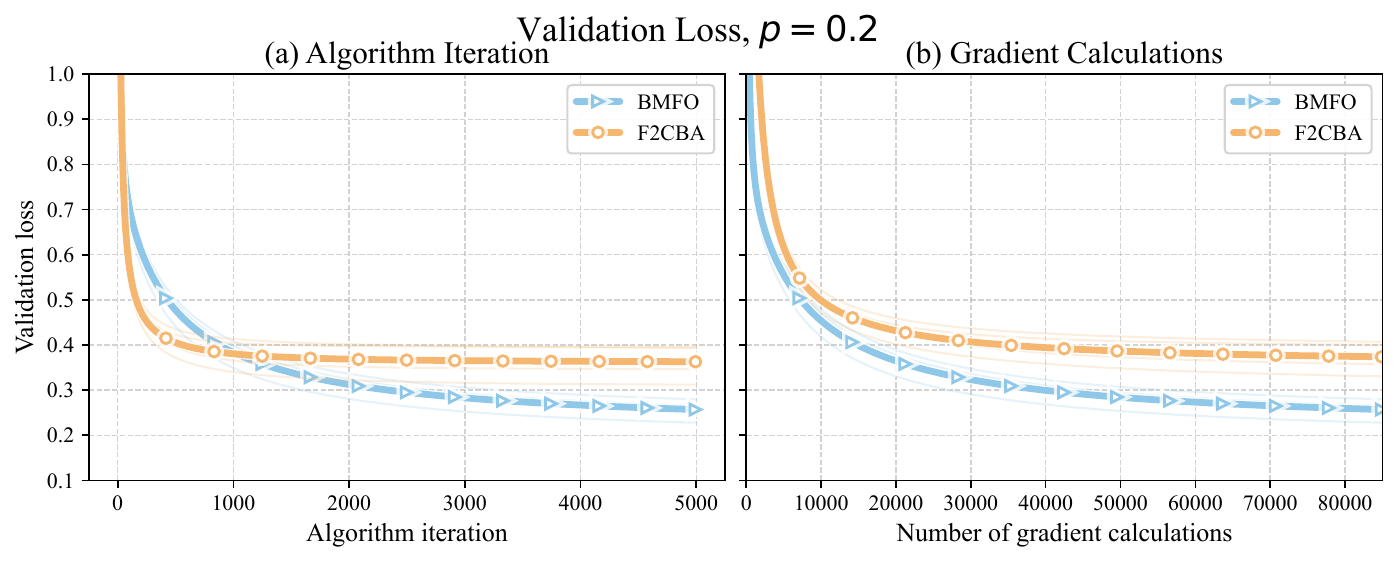}
    \caption{
Constrained MNIST hypercleaning with label corruption probability \(p=0.2\). Panel~(a) reports validation cross-entropy versus algorithm iterations, and panel~(b)
reports validation cross-entropy versus the number of gradient evaluations. BMFO decreases the
validation loss faster and attains a lower final loss than the first-order constrained bilevel baseline
F2CBA.
}
\label{fig:mnist_hypercleaning}
\end{figure}

\paragraph{Constrained MNIST hypercleaning.}
We consider a constrained variant of the MNIST hypercleaning task~\citep{deng2012mnist}. The
training set \(\mathcal D_{\rm train}=\{(\tilde\xi_i,\tilde y_i)\}_{i=1}^n\) has labels corrupted with
probability \(p\), while the validation set
\(\mathcal D_{\rm val}=\{(\xi_j,y_j)\}_{j=1}^m\) is clean. The upper variable
\(\lambda\in\mathbb R^n\) assigns weight \(\sigma(\lambda_i)\in(0,1)\) to training example \(i\),
and the lower variable \(w\) is the softmax-regression parameter. The constrained bilevel problem is
\(
\min_{\lambda\in\mathbb R^n} F(\lambda):=f(\lambda,w^\star(\lambda)),
\quad
w^\star(\lambda)\in\arg\min_{w\in\mathcal W_\tau} g(\lambda,w),\)
where \( f(\lambda,w):=\frac1m\sum_{j=1}^m \ell_{\rm CE}(\xi_j,y_j;w),
g(\lambda,w):=\frac1n\sum_{i=1}^n \sigma(\lambda_i)\ell_{\rm CE}(\tilde\xi_i,\tilde y_i;w)
+\frac{\rho}{2}\|w\|_2^2 .\)
Unlike standard hypercleaning, the lower parameter is constrained. We impose box constraints and
patchwise signed-mass constraints:
\(\mathcal W_\tau
:=
\left\{
w:\ |w_{q,c}|\le R,\ 
\left|\sum_{q\in\mathcal P_r}w_{q,c}\right|\le \tau b_{r,c}
\ \text{for all }q,c,r
\right\},\)
where \(\mathcal P_r\) denotes a pixel patch and \(\tau\in(0,1]\) controls constraint tightness. Further
implementation details are given in Appendix~\ref{app:experiments}. We compare
Algorithm~\ref{alg:dikin_sbo} (BMFO) with the first-order constrained bilevel method of
\citet{kornowski2024first} (F2CBA). Figure~\ref{fig:mnist_hypercleaning} shows that BMFO reduces
validation loss faster both per iteration and per gradient evaluation, and reaches a lower final loss
under the same constrained hypercleaning setup.

\section{Discussion}
\label{sec:discussion}

We develop a barrier-metric first-order method for constrained bilevel optimization with polyhedral
lower feasible sets. The present theory is limited to fixed compact polyhedral lower feasible sets, strongly convex lower
objectives, and stochasticity only in the upper-level objective. Moreover, the guarantees are for
\(F_\mu\), with approximation to the original constrained problem controlled by the barrier-bias
bounds. Future work includes \(x\)-dependent or coupled constraints, weaker lower curvature and
fully stochastic lower-level oracles.

\section*{Acknowledgements}
This research was supported, in part, by the NSF AI Institute for Advances in Optimization Award 2112533.

\bibliographystyle{plainnat}
\bibliography{references}

%%%%%%%%%%%%%%%%%%%%%%%%%%%%%%%%%%%%%%%%%%%%%%%%%%%%%%%%%%%%

\clearpage

\appendix

\section{Notation and Constants}\label{app:notation}

This appendix collects the constants used throughout the analysis. The constants
fall into three groups: \emph{Euclidean constants} (introduced in
Assumption~\ref{ass:euclidean}), \emph{local Dikin constants} from the
local Dikin regularity (Proposition~\ref{prop:derived_dikin_regularity} and its
expanded form in Appendix~\ref{app:dikin_regularity}), and \emph{stochastic
constants} from the upper-level stochastic oracle assumptions
(Assumptions~\ref{ass:stoch_oracle}--\ref{ass:stoch_local_bounds}). Entries
in the third column are explicit admissible values derived in the corresponding
proofs (upper bounds for smoothness/Lipschitz constants, lower bound for the
curvature constant $\rho_\psi^\eta$).

\begin{table}[h]
\centering
\small
\renewcommand{\arraystretch}{1.25}
\caption{Constants used throughout the analysis.}
\begin{tabularx}{\textwidth}{|l|X|l|}
\hline
\textbf{Symbol} & \textbf{Meaning} & \textbf{Explicit value} \\
\hline
$\ell_{f,0}$ & Bound of $\|\nabla_x f\|, \|\nabla_y f\|$ & $\cdot$ \\
\hline
$\ell_{f,1}$ & Smoothness of $f$ & $\cdot$ \\
\hline
$\ell_{f,2}$ & Hessian-continuity of $f$  & $\cdot$ \\
\hline
$\ell_{g,0}$ & Bound of $\|\nabla_x g\|$ & $\cdot$ \\
\hline
$\ell_{g,1}$ & Smoothness of $g$ & $\cdot$ \\
\hline
$\ell_{g,2}$ & Hessian-continuity of $g$ & $\cdot$ \\
\hline
$\rho_g$ & Strong convexity of $g$ in $y$ & $\cdot$ \\
\hline
$\rho_\psi^\eta$ & Lower curvature of $\psi_\mu$ in anchored Dikin metric on $T_{2\eta,\mu}$ & $\mu(1-2\eta)^2$ \\
\hline
$\ell_{\psi,1}^\eta$ & Smoothness of $\psi_\mu$ in anchored Dikin metric on $T_{2\eta,\mu}$ & $\kappa_\phi^2\ell_{g,1}+\mu(1-2\eta)^{-2}$ \\
\hline
$\ell_{\psi,2}^\eta$ & Hessian-continuity of $\psi_\mu$ in anchored Dikin metric & $\cdot$ \\
\hline
$\ell_{f,0}^\eta$ & Local Dikin dual-norm bound on $\nabla_y f$ & $\kappa_\phi\,\ell_{f,0}$ \\
\hline
$\ell_{f,1}^\eta$ & Local Dikin smoothness of $f$ & $\kappa_\phi^2\ell_{f,1}$ \\
\hline
$\ell_{f,2}^\eta$ & Local Dikin Hessian-continuity of $f$ & $\cdot$ \\
\hline
$C_f^\eta$ & Proxy-center transfer radius in Dikin metric & $2\,\ell_{f,0}^\eta/\rho_\psi^\eta$ \\
\hline
$\lambda_\eta$ & Multiplier threshold for proxy regularity & $\cdot$ \\
\hline
$\ell_{\lambda,0}^\eta$ & Lipschitzness of $y_{\lambda,\mu}^\star(x)$ in $x$ (for $\lambda\ge\lambda_\eta$) & $3\ell_{\psi,1}^\eta/\rho_\psi^\eta$ \\
\hline
$\ell_{\lambda,1}^\eta$ & Local Dikin Lipschitz constant of $\nabla y_{\lambda,\mu}^\star(x)$ & $\cdot$ \\
\hline
$\ell_{*,0}^\eta$ & Lipschitzness of $y_\mu^\star(x)$ in $x$, enlarged to be at least \(1\) & $1+3\ell_{\psi,1}^\eta/\rho_\psi^\eta$ \\
\hline
$\ell_{*,1}^\eta$ & Local Dikin Lipschitz constant of $\nabla y_\mu^\star(x)$ & $\cdot$ \\
\hline
$L_F^\eta$ & Local smoothness of $F_\mu$ in Euclidean norm &
$\Bigl(\ell_{f,1}+\tfrac{\ell_{\psi,1}^\eta\ell_{f,1}^\eta}{\rho_\psi^\eta}+\tfrac{2\ell_{\psi,1}^\eta\ell_{\psi,2}^\eta\ell_{f,0}^\eta}{(\rho_\psi^\eta)^2}\Bigr)\ell_{*,0}^\eta$ \\
\hline
$c_x^\eta$ & Local proxy-gradient bias constant &
$\tfrac{2\ell_{\psi,1}^\eta\ell_{f,0}^\eta}{(\rho_\psi^\eta)^2}\bigl(\ell_{f,1}^\eta+\tfrac{\ell_{\psi,2}^\eta\ell_{f,0}^\eta}{\rho_\psi^\eta}\bigr)$ \\
\hline
$\sigma_{f,x}^2$ & Variance of stochastic $\nabla_x f$ oracle (Euclidean) & $\cdot$ \\
\hline
$\sigma_{f,y}^2$ & Variance of stochastic $\nabla_y f$ oracle (Euclidean) & $\cdot$ \\
\hline
$\sigma_{f,y,\eta}^2$ & Variance of stochastic $\nabla_y f$ oracle in anchored Dikin dual norm & $\kappa_\phi^2\,\sigma_{f,y}^2$ \\
\hline
$\sigma_\eta^2$ & Combined variance & $\sigma_{f,x}^2+\sigma_{f,y,\eta}^2$ \\
\hline
$\bar\ell_{f,x,0}^\eta$ & A.s.\ Euclidean bound on sampled $\nabla_x f$ & $\cdot$ \\
\hline
$\bar\ell_{f,y,0}^\eta$ & A.s.\ Dikin dual-norm bound on sampled $\nabla_y f$ & $\cdot$ \\
\hline
$\bar\sigma_y^\eta$ & Pathwise inner-noise budget in Dikin dual norm & $\ell_{f,0}^\eta+\bar\ell_{f,y,0}^\eta$ \\
\hline
$\kappa_\phi$ & Euclidean--Dikin conversion factor (depends only on $Y$) & $\cdot$ \\
\hline
\end{tabularx}
\label{tab:constants}
\end{table}

\noindent
\textbf{Conventions.}
Local Dikin constants depend on $\eta$ and $\mu$; the $\mu$-dependence is
suppressed in notation throughout. After
Proposition~\ref{prop:derived_dikin_regularity}, the exact-anchor and
proxy-anchor versions of $\rho_\psi^\eta$, $\ell_{\psi,1}^\eta$, and
$\ell_{\psi,2}^\eta$ are unified by decreasing $\rho_\psi^\eta$ and increasing
$\ell_{\psi,1}^\eta$, $\ell_{\psi,2}^\eta$ as needed, so that the same symbols
hold on both buffer neighborhoods $T_{2\eta,\mu}(x)$ and
$T_{2\eta,\mu}^\lambda(x)$. Similarly,
$\ell_{f,0}^\eta$, $\ell_{f,1}^\eta$, $\ell_{f,2}^\eta$ are enlarged to dominate
both anchors.

\section{Proofs for Section~\ref{sec:barrier smoothing}}
\label{app:barrier_setup}

\subsection{Proof of Theorem~\ref{thm:barrier_smoothing}}
\label{app:proof_barrier_smoothing}

\begin{proof}
\proofstep{(i) Existence and uniqueness of \(y_\mu^\star(x)\).}
Fix \(x\). By Assumption~\ref{ass:euclidean}, the map \(y\mapsto g(x,y)\) is
\(\rho_g\)-strongly convex. Since the logarithmic barrier \(\phi\) is convex on
\(\operatorname{int}(Y)\), the barrierized objective
\(\psi_\mu(x,y)=g(x,y)+\mu\phi(y)\) is also \(\rho_g\)-strongly convex on
\(\operatorname{int}(Y)\). Hence it has at most one minimizer.
It remains to show existence. Since \(Y\) is a compact full-dimensional polytope, there exists
\(\bar y\in\operatorname{int}(Y)\). Consider a minimizing sequence
\(\{y_j\}_{j\ge1}\subset\operatorname{int}(Y)\) for \(\psi_\mu(x,\cdot)\). Because \(Y\) is compact,
a subsequence converges to some \(\bar y_\infty\in Y\). The barrier satisfies
\(\phi(y)\to+\infty\) whenever \(y\to\partial Y\) from the interior, while
\(\psi_\mu(x,\bar y)<\infty\). Therefore the limit point \(\bar y_\infty\) cannot lie on
\(\partial Y\). Hence \(\bar y_\infty\in\operatorname{int}(Y)\), and by continuity of
\(\psi_\mu(x,\cdot)\) on \(\operatorname{int}(Y)\), this point attains the minimum. Thus the
barrierized lower problem has a minimizer \(y_\mu^\star(x)\in\operatorname{int}(Y)\). By strong
convexity, this minimizer is unique.

\proofstep{(ii) Differentiability of \(x\mapsto y_\mu^\star(x)\).}
Since \(g\) is \(C^2\) and \(\phi\) is \(C^\infty\) on \(\operatorname{int}(Y)\), the mapping
\((x,y)\mapsto \nabla_y\psi_\mu(x,y)\) is \(C^1\) on \(X\times\operatorname{int}(Y)\). The
first-order optimality condition for the barrierized lower problem is
\begin{equation*}
\nabla_y\psi_\mu\bigl(x,y_\mu^\star(x)\bigr)=0.
\end{equation*}
Moreover,
\begin{equation*}
\nabla_{yy}^2\psi_\mu\bigl(x,y_\mu^\star(x)\bigr)
=
\nabla_{yy}^2g\bigl(x,y_\mu^\star(x)\bigr)
+
\mu\nabla^2\phi\bigl(y_\mu^\star(x)\bigr).
\end{equation*}
By Assumption~\ref{ass:euclidean}, \(\nabla_{yy}^2g(x,\cdot)\succeq \rho_g I\), and since
\(\phi\) is convex, \(\nabla^2\phi(y)\succeq0\) on \(\operatorname{int}(Y)\). Therefore
\begin{equation*}
\nabla_{yy}^2\psi_\mu\bigl(x,y_\mu^\star(x)\bigr)\succeq \rho_g I,
\end{equation*}
so the Jacobian with respect to \(y\) is invertible. The implicit function theorem implies that
\(x\mapsto y_\mu^\star(x)\) is differentiable, with
\begin{equation*}
D y_\mu^\star(x)
=
-
\Bigl(\nabla_{yy}^2\psi_\mu\bigl(x,y_\mu^\star(x)\bigr)\Bigr)^{-1}
\nabla_{yx}^2\psi_\mu\bigl(x,y_\mu^\star(x)\bigr).
\end{equation*}

\proofstep{(iii) Differentiability of \(F_\mu\) and the hypergradient formula.}
By definition, \(F_\mu(x)=f\bigl(x,y_\mu^\star(x)\bigr)\). Applying the chain rule gives
\begin{equation*}
\nabla F_\mu(x)
=
\nabla_x f\bigl(x,y_\mu^\star(x)\bigr)
+
\bigl(Dy_\mu^\star(x)\bigr)^\top
\nabla_y f\bigl(x,y_\mu^\star(x)\bigr).
\end{equation*}
Substituting the expression for \(D y_\mu^\star(x)\) yields
\begin{align*}
\nabla F_\mu(x)
&=
\nabla_x f\bigl(x,y_\mu^\star(x)\bigr)
-
\nabla_{xy}^2\psi_\mu\bigl(x,y_\mu^\star(x)\bigr)^\top
\Bigl(\nabla_{yy}^2\psi_\mu\bigl(x,y_\mu^\star(x)\bigr)\Bigr)^{-1}
\nabla_y f\bigl(x,y_\mu^\star(x)\bigr),
\end{align*}
which is the claimed formula.

\proofstep{(iv(a)) Lower-objective bias.}
Let \(y^\star(x)\) denote the minimizer of the original constrained lower problem. Since
\(y^\star(x)\) minimizes \(g(x,\cdot)\) over \(Y\), and
\(y_\mu^\star(x)\in\operatorname{int}(Y)\subseteq Y\), we have
\begin{equation*}
0\le g\bigl(x,y_\mu^\star(x)\bigr)-g\bigl(x,y^\star(x)\bigr).
\end{equation*}
For the upper bound, the first-order optimality condition for the barrierized lower problem gives
\begin{equation*}
\nabla_y g\bigl(x,y_\mu^\star(x)\bigr)
+
\mu\nabla\phi\bigl(y_\mu^\star(x)\bigr)
=
0.
\end{equation*}
By convexity of \(g(x,\cdot)\), for any \(y\in Y\),
\begin{align*}
g(x,y)
&\ge
g\bigl(x,y_\mu^\star(x)\bigr)
+
\left\langle
\nabla_y g\bigl(x,y_\mu^\star(x)\bigr),
y-y_\mu^\star(x)
\right\rangle  \\
&=
g\bigl(x,y_\mu^\star(x)\bigr)
-
\mu
\left\langle
\nabla\phi\bigl(y_\mu^\star(x)\bigr),
y-y_\mu^\star(x)
\right\rangle .
\end{align*}
For any \(z\in\operatorname{int}(Y)\) and any \(y\in Y\), using
\(s_i(y)=b_i-a_i^\top y\), we have
\begin{align*}
\langle\nabla\phi(z),y-z\rangle
&=
\sum_{i=1}^m \frac{a_i^\top(y-z)}{s_i(z)}
=
\sum_{i=1}^m \frac{s_i(z)-s_i(y)}{s_i(z)}
=
m-\sum_{i=1}^m \frac{s_i(y)}{s_i(z)}
\le m,
\end{align*}
because \(s_i(y)\ge0\) for every \(y\in Y\). Taking \(z=y_\mu^\star(x)\) and
\(y=y^\star(x)\) gives
\begin{equation*}
g\bigl(x,y^\star(x)\bigr)
\ge
g\bigl(x,y_\mu^\star(x)\bigr)-m\mu.
\end{equation*}
Therefore
\begin{equation*}
0
\le
g\bigl(x,y_\mu^\star(x)\bigr)-g\bigl(x,y^\star(x)\bigr)
\le
m\mu.
\end{equation*}
Thus the first bias bound holds with \(C_g:=m\).

\proofstep{(iv(b)) Distance between \(y_\mu^\star(x)\) and \(y^\star(x)\).}
By strong convexity of \(g(x,\cdot)\),
\begin{align*}
g\bigl(x,y_\mu^\star(x)\bigr)
&\ge
g\bigl(x,y^\star(x)\bigr)
+
\left\langle
\nabla_y g\bigl(x,y^\star(x)\bigr),
y_\mu^\star(x)-y^\star(x)
\right\rangle  \\
&\qquad
+
\frac{\rho_g}{2}
\|y_\mu^\star(x)-y^\star(x)\|_2^2.
\end{align*}
Since \(y^\star(x)\) minimizes \(g(x,\cdot)\) over the convex set \(Y\), and
\(y_\mu^\star(x)\in Y\), the variational inequality gives
\begin{equation*}
\left\langle
\nabla_y g\bigl(x,y^\star(x)\bigr),
y_\mu^\star(x)-y^\star(x)
\right\rangle
\ge 0.
\end{equation*}
Hence
\begin{equation*}
g\bigl(x,y_\mu^\star(x)\bigr)-g\bigl(x,y^\star(x)\bigr)
\ge
\frac{\rho_g}{2}\|y_\mu^\star(x)-y^\star(x)\|_2^2.
\end{equation*}
Combining this with the lower-objective bias bound gives
\begin{equation*}
\frac{\rho_g}{2}\|y_\mu^\star(x)-y^\star(x)\|_2^2
\le
m\mu,
\end{equation*}
and therefore
\begin{equation*}
\|y_\mu^\star(x)-y^\star(x)\|_2
\le
\sqrt{\frac{2m}{\rho_g}}\,\sqrt{\mu}.
\end{equation*}
Thus the distance bound holds with \(C_y:=\sqrt{2m/\rho_g}\).

\proofstep{(iv(c)) Outer-objective bias.}
By definition,
\begin{equation*}
F_\mu(x)-F(x)
=
f\bigl(x,y_\mu^\star(x)\bigr)-f\bigl(x,y^\star(x)\bigr).
\end{equation*}
Using the gradient bound \(\|\nabla_y f(x,y)\|_2\le \ell_{f,0}\) from
Assumption~\ref{ass:euclidean} and the mean-value theorem in the \(y\)-variable,
\begin{equation*}
|F_\mu(x)-F(x)|
\le
\ell_{f,0}
\|y_\mu^\star(x)-y^\star(x)\|_2.
\end{equation*}
Applying the distance bound above yields
\begin{equation*}
|F_\mu(x)-F(x)|
\le
\ell_{f,0}\sqrt{\frac{2m}{\rho_g}}\sqrt{\mu}.
\end{equation*}
Thus the outer-objective bias bound holds with
\(C_F:=\ell_{f,0}\sqrt{2m/\rho_g}\). The proof is complete.
\end{proof}

\section{Proofs for Local Dikin Regularity}
\label{app:dikin_regularity}

This section proves Proposition~\ref{prop:derived_dikin_regularity}. The main text states the
local Dikin regularity result in a compressed form. For the proofs below, we first record the
fully expanded version of the exact- and proxy-neighborhood bounds. They are derived from Assumption~\ref{ass:euclidean}, compact polyhedral geometry,
and self-concordance of the logarithmic barrier.

\subsection{Expanded Form of Proposition~\ref{prop:derived_dikin_regularity}}\label{app:expanded_form}
Fix \(\eta\in(0,\frac{1}{2})\), up to changing \(\eta\)-dependent constants, the following bounds hold
uniformly for \(x\in X\).

\Etag{Exact neighborhood regularity of \(\psi_\mu\).}{ex:E1}
For every \(y\in T_{2\eta,\mu}(x)\),
\[
\rho_\psi^\eta H_\mu^\star(x)
\preceq
\nabla_{yy}^2\psi_\mu(x,y)
\preceq
\ell_{\psi,1}^\eta H_\mu^\star(x),
\qquad
\|\nabla_{xy}^2\psi_\mu(x,y)\|_{2\to y_\mu^\star(x),*}
\le
\ell_{\psi,1}^\eta .
\]
Moreover, for every \(y_1,y_2\in T_{2\eta,\mu}(x)\),
\[
\|\nabla_{yy}^2\psi_\mu(x,y_1)-\nabla_{yy}^2\psi_\mu(x,y_2)\|_{y_\mu^\star(x)\to y_\mu^\star(x),*}
\le
\ell_{\psi,2}^\eta
\|y_1-y_2\|_{y_\mu^\star(x)},
\]
and
\[
\|\nabla_{xy}^2\psi_\mu(x,y_1)-\nabla_{xy}^2\psi_\mu(x,y_2)\|_{2\to y_\mu^\star(x),*}
\le
\ell_{\psi,2}^\eta
\|y_1-y_2\|_{y_\mu^\star(x)}.
\]

\Etag{Exact neighborhood regularity of \(f\).}{ex:E2}
For every \(y\in T_{2\eta,\mu}(x)\),
\[
\|\nabla_y f(x,y)\|_{y_\mu^\star(x),*}\le \ell_{f,0}^\eta,\qquad
\|\nabla_{xy}^2 f(x,y)\|_{2\to y_\mu^\star(x),*}\le \ell_{f,1}^\eta,
\]
and
\[
\|\nabla_{yy}^2 f(x,y)\|_{y_\mu^\star(x)\to y_\mu^\star(x),*}\le \ell_{f,1}^\eta.
\]
The corresponding Hessian-difference bounds hold with constant \(\ell_{f,2}^\eta\).

\Etag{Proxy-center transfer.}{ex:E3}
There exists \(\lambda_\eta<\infty\) and finite constant \(C_f^\eta\) such that, for every \(\lambda\ge\lambda_\eta\),
\[
\|y_{\lambda,\mu}^\star(x)-y_\mu^\star(x)\|_{y_\mu^\star(x)}
\le
\frac{C_f^\eta}{\lambda}
\le
\eta .
\]
Consequently, \(H_{\lambda,\mu}^\star(x)\) and \(H_\mu^\star(x)\) are spectrally comparable, and the
proxy neighborhood is contained in a slightly enlarged exact buffer neighborhood.

\Etag{Proxy neighborhood  regularity.}{ex:E4}
For every \(\lambda\ge\lambda_\eta\), the proxy neighborhood
\(T_{2\eta,\mu}^{\lambda}(x)\) is contained in an enlarged exact buffer neighborhood, and the proxy anchor
\(H_{\lambda,\mu}^\star(x):=\nabla^2\phi(y_{\lambda,\mu}^\star(x))\) is spectrally comparable to
\(H_\mu^\star(x)\). Consequently, after changing only \(\eta\)-dependent constants, the exact neighborhood
bounds transfer to the proxy neighborhood. In particular, there exist finite constants
\(\rho_{\psi,\lambda}^{\eta}>0\) and
\(\ell_{\psi,\lambda,1}^{\eta},\ell_{\psi,\lambda,2}^{\eta}<\infty\) such that, for every
\(y\in T_{2\eta,\mu}^{\lambda}(x)\),
\[
\rho_{\psi,\lambda}^{\eta}H_{\lambda,\mu}^\star(x)
\preceq
\nabla_{yy}^2\psi_\mu(x,y)
\preceq
\ell_{\psi,\lambda,1}^{\eta}H_{\lambda,\mu}^\star(x),
\qquad
\|\nabla_{xy}^2\psi_\mu(x,y)\|_{2\to y_{\lambda,\mu}^\star(x),*}
\le
\ell_{\psi,\lambda,1}^{\eta}.
\]
Moreover, for all \(y_1,y_2\in T_{2\eta,\mu}^{\lambda}(x)\),
\[
\|\nabla_{yy}^2\psi_\mu(x,y_1)-\nabla_{yy}^2\psi_\mu(x,y_2)\|_{y_{\lambda,\mu}^\star(x)\to y_{\lambda,\mu}^\star(x),*}
\le
\ell_{\psi,\lambda,2}^{\eta}
\|y_1-y_2\|_{y_{\lambda,\mu}^\star(x)},
\]
and
\[
\|\nabla_{xy}^2\psi_\mu(x,y_1)-\nabla_{xy}^2\psi_\mu(x,y_2)\|_{2\to y_{\lambda,\mu}^\star(x),*}
\le
\ell_{\psi,\lambda,2}^{\eta}
\|y_1-y_2\|_{y_{\lambda,\mu}^\star(x)}.
\]
The corresponding \(f\)-bounds also hold on \(T_{2\eta,\mu}^{\lambda}(x)\) with respect to the proxy
anchor: after increasing
\(\ell_{f,0}^{\eta},\ell_{f,1}^{\eta},\ell_{f,2}^{\eta}\) if necessary,
\[
\|\nabla_y f(x,y)\|_{y_{\lambda,\mu}^\star(x),*}\le \ell_{f,0}^{\eta},\qquad
\|\nabla_{xy}^2 f(x,y)\|_{2\to y_{\lambda,\mu}^\star(x),*}\le \ell_{f,1}^{\eta},
\]
and
\[
\|\nabla_{yy}^2 f(x,y)\|_{y_{\lambda,\mu}^\star(x)\to y_{\lambda,\mu}^\star(x),*}
\le
\ell_{f,1}^{\eta},
\]
with the analogous Lipschitz-Hessian bounds controlled by \(\ell_{f,2}^{\eta}\).

\subsection{Self-concordant Metric Change and Euclidean--Dikin Conversion}

The two lemmas in this subsection collect the self-concordant metric-change and Euclidean–Dikin conversion bounds that we invoke throughout the appendix.

\begin{lemma}[Anchor-switch bound in anchored Dikin norms]
\label{lem:anchor_switch}
Let \(y_1,y_2\in \operatorname{int}(Y)\), and define
\(r(y_1,y_2):=\|y_1-y_2\|_{y_1}\). If \(r(y_1,y_2)<1\), then
\[
(1-r(y_1,y_2))^2H(y_1)
\preceq
H(y_2)
\preceq
(1-r(y_1,y_2))^{-2}H(y_1),
\]
where \(H(y):=\nabla^2\phi(y)\). Equivalently, for every \(u\),
\[
\|u\|_{y_2}^2
\le
(1-r(y_1,y_2))^{-2}\|u\|_{y_1}^2.
\]
The corresponding dual-norm bound also holds:
\[
\|w\|_{y_2,*}^2
\le
(1-r(y_1,y_2))^{-2}\|w\|_{y_1,*}^2.
\]
\end{lemma}

\begin{proof}
Let \(h:=y_2-y_1\) and \(r:=\|h\|_{y_1}\). For each constraint \(i\),
\[
\frac{|a_i^\top h|}{s_i(y_1)}
\le
\left(\sum_{j=1}^m\frac{(a_j^\top h)^2}{s_j(y_1)^2}\right)^{1/2}
=
\|h\|_{y_1}
=
r.
\]
Hence
\[
(1-r)s_i(y_1)
\le
s_i(y_2)
=
s_i(y_1)-a_i^\top h
\le
(1+r)s_i(y_1).
\]
Since \(r<1\), all slacks \(s_i(y_2)\) are positive. Therefore, for every \(u\),
\[
u^\top H(y_2)u
=
\sum_{i=1}^m\frac{(a_i^\top u)^2}{s_i(y_2)^2}
\le
(1-r)^{-2}
\sum_{i=1}^m\frac{(a_i^\top u)^2}{s_i(y_1)^2}
=
(1-r)^{-2}u^\top H(y_1)u.
\]
Similarly,
\[
u^\top H(y_2)u
\ge
(1+r)^{-2}u^\top H(y_1)u
\ge
(1-r)^2u^\top H(y_1)u,
\]
where the last inequality uses \((1+r)^{-2}\ge (1-r)^2\) for \(r\in[0,1)\). This proves the matrix
comparison. The primal norm bound follows immediately. The dual norm bound follows by inverting
the matrix comparison:
\[
H(y_2)^{-1}\preceq (1-r)^{-2}H(y_1)^{-1}.
\]
\end{proof}

\begin{lemma}[Euclidean--Dikin norm conversion]
\label{lem:euc_dikin}
Let \(Y=\{y\in\mathbb R^{d_y}:Ay\le b\}\) be a compact full-dimensional polytope, and let
\[
\phi(y)=-\sum_{i=1}^m\log s_i(y),
\qquad
s_i(y)=b_i-a_i^\top y,
\qquad
y\in\operatorname{int}(Y).
\]
Then there exists a constant \(\kappa_\phi\ge 1\), depending only on \(Y\), such that for every
\(y\in\operatorname{int}(Y)\),
\[
\|u\|_2\le \kappa_\phi\|u\|_y,
\qquad
\|w\|_{y,*}\le \kappa_\phi\|w\|_2,
\qquad
\forall u,w.
\]
In particular, the same bounds hold for every exact or proxy anchor.
\end{lemma}

\begin{proof}
Let \(\bar s_i:=\max_{y\in Y}s_i(y)\). Since \(Y\) is compact and has nonempty interior, each
\(\bar s_i\) is finite and positive. For every \(y\in\operatorname{int}(Y)\), we have
\(0<s_i(y)\le \bar s_i\), and therefore
\[
\nabla^2\phi(y)
=
\sum_{i=1}^m\frac{a_i a_i^\top}{s_i(y)^2}
\succcurlyeq
\sum_{i=1}^m\frac{a_i a_i^\top}{\bar s_i^2}
=
A^\top\operatorname{Diag}(\bar s^{-2})A.
\]
The matrix \(A^\top\operatorname{Diag}(\bar s^{-2})A\) is positive definite. Indeed, if there were a
nonzero vector \(v\) with \(Av=0\), then for any \(y\in Y\), the line \(y+tv\) would satisfy
\(A(y+tv)=Ay\le b\) for all \(t\in\mathbb R\), contradicting compactness of \(Y\). Hence
\(A\) has full column rank, and the displayed matrix is positive definite.

Let
\[
\rho_\phi
:=
\lambda_{\min}\!\left(A^\top\operatorname{Diag}(\bar s^{-2})A\right)>0.
\]
Then \(\nabla^2\phi(y)\succeq \rho_\phi I\) for every \(y\in\operatorname{int}(Y)\). Therefore,
\[
\|u\|_y^2
=
u^\top\nabla^2\phi(y)u
\ge
\rho_\phi\|u\|_2^2,
\]
so \(\|u\|_2\le \rho_\phi^{-1/2}\|u\|_y\). Also,
\[
\|w\|_{y,*}^2
=
w^\top\nabla^2\phi(y)^{-1}w
\le
\rho_\phi^{-1}\|w\|_2^2,
\]
so \(\|w\|_{y,*}\le \rho_\phi^{-1/2}\|w\|_2\). Taking
\(\kappa_\phi:=\max\{1,\rho_\phi^{-1/2}\}\) gives the claim.
\end{proof}

\begin{remark}[Absorbing the Euclidean--Dikin conversion factor]
Lemma~\ref{lem:euc_dikin} provides a fixed geometric constant \(\kappa_\phi\) that converts
between Euclidean norms and anchored Dikin norms. In all statements below that are written in
Euclidean norm, we may absorb this factor into the corresponding local Dikin constants. We keep the intrinsic Dikin constants \(\rho_\psi^\eta,\ell_{\psi,1}^\eta,\ell_{\psi,2}^\eta,\ldots\)
visible, and absorb \(\kappa_\phi\) only into the derived Euclidean-facing constants.
\end{remark}

\subsection{Proof of Proposition~\ref{prop:derived_dikin_regularity}}
\label{app:proof_dikin_regularity}
\begin{proof}
Fix \(x\in X\). We prove the expanded bounds in
Appendix~\ref{app:expanded_form}; the main-text Proposition~\ref{prop:derived_dikin_regularity} follows immediately by
restricting from the buffer neighborhoods to the target neighborhoods and absorbing fixed \(\eta\)-dependent constants.
\proofstep{(i) Exact neighborhood regularity of \(\psi_\mu\).}
Let \(y\in T_{2\eta,\mu}(x)\). By Lemma~\ref{lem:anchor_switch},
\[
(1-2\eta)^2H_\mu^\star(x)
\preceq
\nabla^2\phi(y)
\preceq
(1-2\eta)^{-2}H_\mu^\star(x).
\]
Since
\(\nabla_{yy}^2\psi_\mu(x,y)=\nabla_{yy}^2g(x,y)+\mu\nabla^2\phi(y)\), and
\(\nabla_{yy}^2g(x,y)\succeq \rho_g I\), we obtain
\[
\nabla_{yy}^2\psi_\mu(x,y)
\succeq
\mu(1-2\eta)^2H_\mu^\star(x).
\]
Thus the lower Dikin-curvature bound holds with
\(\rho_\psi^\eta:=\mu(1-2\eta)^2\). For the upper bound, Assumption~\ref{ass:euclidean} gives
\(\nabla_{yy}^2g(x,y)\preceq \ell_{g,1}I\), and Lemma~\ref{lem:anchor_switch}
implies \(I\preceq \kappa_\phi^2H_\mu^\star(x)\). Hence
\[
\nabla_{yy}^2\psi_\mu(x,y)
\preceq
\bigl(\kappa_\phi^2\ell_{g,1}+\mu(1-2\eta)^{-2}\bigr)H_\mu^\star(x).
\]
After increasing constants, this gives the upper bound with \(\ell_{\psi,1}^\eta<\infty\).
Because the barrier does not depend on \(x\), \(\nabla_{xy}^2\psi_\mu=\nabla_{xy}^2g\). The mixed
Hessian bound follows from Assumption~\ref{ass:euclidean} and the Euclidean--Dikin conversion:
\[
\|\nabla_{xy}^2\psi_\mu(x,y)\|_{2\to y_\mu^\star(x),*}
\le
\kappa_\phi\ell_{g,1}.
\]
The Lipschitz-Hessian bounds follow similarly from the Lipschitzness of \(\nabla^2g\), the
self-concordant third-derivative bounds for the logarithmic barrier on \(T_{2\eta,\mu}(x)\), and
the norm conversions in Lemma~\ref{lem:euc_dikin}. Hence there exists
\(\ell_{\psi,2}^\eta<\infty\) such that the exact-neighborhood Lipschitz-Hessian bounds in
\eqref{ex:E1} hold.
\proofstep{(ii) Exact neighborhood regularity of \(f\).}
The \(f\)-bounds in \eqref{ex:E2} follow directly from Assumption~\ref{ass:euclidean} and
Lemma~\ref{lem:anchor_switch}. In particular,
\[
\|\nabla_y f(x,y)\|_{y_\mu^\star(x),*}\le \kappa_\phi\ell_{f,0},
\qquad
\|\nabla_{xy}^2 f(x,y)\|_{2\to y_\mu^\star(x),*}\le \kappa_\phi\ell_{f,1},
\]
and
\[
\|\nabla_{yy}^2 f(x,y)\|_{y_\mu^\star(x)\to y_\mu^\star(x),*}
\le
\kappa_\phi^2\ell_{f,1}.
\]
The corresponding Hessian-difference bounds follow from the Lipschitzness of \(\nabla^2 f\) and
the same norm conversions. Enlarging constants gives finite
\(\ell_{f,0}^\eta,\ell_{f,1}^\eta,\ell_{f,2}^\eta\).

\proofstep{(iii(a)) Proxy-center proximity.}
We first prove the proxy-center transfer bound in \eqref{ex:E3}. Let
\[
D_Y:=\sup_{u,v\in Y}\|u-v\|_2<\infty,
\qquad
\Delta_f:=\ell_{f,0}D_Y .
\]
Since \(Y\) is compact and \(\|\nabla_y f(x,y)\|_2\le \ell_{f,0}\), we have
\[
|f(x,y)-f(x,y')|\le \Delta_f
\qquad
\text{for all }x\in X,\ y,y'\in Y.
\]
By optimality of \(y_{\lambda,\mu}^\star(x)\) for \(L_{\lambda,\mu}(x,\cdot)\), and using
\(\psi_\mu^\star(x)=\psi_\mu(x,y_\mu^\star(x))\),
\[
f(x,y_{\lambda,\mu}^\star(x))
+
\lambda\bigl(\psi_\mu(x,y_{\lambda,\mu}^\star(x))-\psi_\mu(x,y_\mu^\star(x))\bigr)
\le
f(x,y_\mu^\star(x)).
\]
Therefore
\begin{equation}
\label{eq:proxy_center_objective_gap}
\psi_\mu(x,y_{\lambda,\mu}^\star(x))-\psi_\mu(x,y_\mu^\star(x))
\le
\frac{\Delta_f}{\lambda}.
\end{equation}

We next show that \(y_{\lambda,\mu}^\star(x)\) enters the exact neighborhood for all sufficiently large
\(\lambda\). Suppose, to the contrary, that
\[
r_\lambda(x):=
\|y_{\lambda,\mu}^\star(x)-y_\mu^\star(x)\|_{y_\mu^\star(x)}
>
\eta .
\]
Define the point on the exact Dikin sphere
\[
\widehat y
:=
y_\mu^\star(x)
+
\frac{\eta}{r_\lambda(x)}
\bigl(y_{\lambda,\mu}^\star(x)-y_\mu^\star(x)\bigr).
\]
Then \(\|\widehat y-y_\mu^\star(x)\|_{y_\mu^\star(x)}=\eta\), so by the exact neighborhood curvature bound,
\[
\psi_\mu(x,\widehat y)-\psi_\mu(x,y_\mu^\star(x))
\ge
\frac{\rho_\psi^\eta}{2}\eta^2.
\]
Because \(\widehat y\) lies on the line segment between \(y_\mu^\star(x)\) and
\(y_{\lambda,\mu}^\star(x)\), convexity of \(\psi_\mu(x,\cdot)\) gives
\[
\psi_\mu(x,y_{\lambda,\mu}^\star(x))-\psi_\mu(x,y_\mu^\star(x))
\ge
\psi_\mu(x,\widehat y)-\psi_\mu(x,y_\mu^\star(x))
\ge
\frac{\rho_\psi^\eta}{2}\eta^2.
\]
This contradicts \eqref{eq:proxy_center_objective_gap} whenever
\[
\lambda\ge \frac{2\Delta_f}{\rho_\psi^\eta\eta^2}.
\]
Hence, for all such \(\lambda\),
\[
y_{\lambda,\mu}^\star(x)\in T_{\eta,\mu}(x).
\]

Once \(y_{\lambda,\mu}^\star(x)\in T_{\eta,\mu}(x)\), we can use the strong monotonicity of
\(\nabla_y\psi_\mu(x,\cdot)\) in the anchored Dikin metric. The optimality conditions are
\[
\nabla_y\psi_\mu(x,y_\mu^\star(x))=0,
\qquad
\nabla_y f(x,y_{\lambda,\mu}^\star(x))
+
\lambda\nabla_y\psi_\mu(x,y_{\lambda,\mu}^\star(x))
=
0.
\]
Therefore,
\begin{align*}
\rho_\psi^\eta
\|y_{\lambda,\mu}^\star(x)-y_\mu^\star(x)\|_{y_\mu^\star(x)}^2
&\le
\left\langle
\nabla_y\psi_\mu(x,y_{\lambda,\mu}^\star(x))
-
\nabla_y\psi_\mu(x,y_\mu^\star(x)),
\,y_{\lambda,\mu}^\star(x)-y_\mu^\star(x)
\right\rangle \\
&=
-\frac{1}{\lambda}
\left\langle
\nabla_y f(x,y_{\lambda,\mu}^\star(x)),
\,y_{\lambda,\mu}^\star(x)-y_\mu^\star(x)
\right\rangle \\
&\le
\frac{1}{\lambda}
\|\nabla_y f(x,y_{\lambda,\mu}^\star(x))\|_{y_\mu^\star(x),*}
\|y_{\lambda,\mu}^\star(x)-y_\mu^\star(x)\|_{y_\mu^\star(x)}.
\end{align*}
Using the exact neighborhood \(f\)-bound,
\[
\|\nabla_y f(x,y_{\lambda,\mu}^\star(x))\|_{y_\mu^\star(x),*}
\le
\ell_{f,0}^\eta,
\]
and canceling one factor gives
\[
\|y_{\lambda,\mu}^\star(x)-y_\mu^\star(x)\|_{y_\mu^\star(x)}
\le
\frac{\ell_{f,0}^\eta}{\lambda\rho_\psi^\eta}.
\]
Thus \eqref{ex:E3} holds with, for instance,
\[
C_f^\eta:=\frac{2\ell_{f,0}^\eta}{\rho_\psi^\eta}
\]
after increasing the constant if necessary. Finally, choose
\[
\lambda_\eta
:=
\max\left\{
\frac{2\Delta_f}{\rho_\psi^\eta\eta^2},
\frac{C_f^\eta}{\eta}
\right\}.
\]
Then for every \(\lambda\ge\lambda_\eta\),
\[
\|y_{\lambda,\mu}^\star(x)-y_\mu^\star(x)\|_{y_\mu^\star(x)}
\le
\frac{C_f^\eta}{\lambda}
\le
\eta .
\]
The constants are uniform over \(x\in X\), because
\(\Delta_f,\rho_\psi^\eta,\ell_{f,0}^\eta\) are uniform.

\proofstep{(iii(b)) Proxy neighborhood transfer.}
We now prove the proxy neighborhood regularity bounds in \eqref{ex:E4}. The preceding step gives, for every
\(\lambda\ge\lambda_\eta\),
\[
\|y_{\lambda,\mu}^\star(x)-y_\mu^\star(x)\|_{y_\mu^\star(x)}
\le
\eta .
\]
By the self-concordant metric-change inequality,
\begin{equation}
\label{eq:proxy_exact_anchor_comparison}
(1-\eta)^2H_\mu^\star(x)
\preceq
H_{\lambda,\mu}^\star(x)
\preceq
(1-\eta)^{-2}H_\mu^\star(x).
\end{equation}
Thus the exact and proxy anchors are spectrally comparable. Moreover, if
\(y\in T_{2\eta,\mu}^{\lambda}(x)\), then another application of the same metric-change inequality,
now around the proxy anchor, gives
\begin{equation}
\label{eq:proxy_tube_barrier_metric_change}
(1-2\eta)^2H_{\lambda,\mu}^\star(x)
\preceq
\nabla^2\phi(y)
\preceq
(1-2\eta)^{-2}H_{\lambda,\mu}^\star(x).
\end{equation}
This is the key local metric comparison on the proxy neighborhood.

We first prove the proxy curvature bound for \(\psi_\mu\). Since
\[
\nabla_{yy}^2\psi_\mu(x,y)
=
\nabla_{yy}^2g(x,y)+\mu\nabla^2\phi(y),
\]
and \(\nabla_{yy}^2g(x,y)\succeq 0\), \eqref{eq:proxy_tube_barrier_metric_change} implies
\[
\nabla_{yy}^2\psi_\mu(x,y)
\succeq
\mu(1-2\eta)^2H_{\lambda,\mu}^\star(x).
\]
Thus the lower curvature bound holds with
\[
\rho_{\psi,\lambda}^\eta:=\mu(1-2\eta)^2.
\]
For the upper bound, Assumption~\ref{ass:euclidean} gives
\(\nabla_{yy}^2g(x,y)\preceq \ell_{g,1}I\). By the Euclidean--Dikin conversion,
\(I\preceq \kappa_\phi^2H_{\lambda,\mu}^\star(x)\). Combining this with
\eqref{eq:proxy_tube_barrier_metric_change} yields
\[
\nabla_{yy}^2\psi_\mu(x,y)
\preceq
\bigl(\kappa_\phi^2\ell_{g,1}
+
\mu(1-2\eta)^{-2}\bigr)H_{\lambda,\mu}^\star(x).
\]
After increasing constants, this gives
\[
\rho_{\psi,\lambda}^\eta H_{\lambda,\mu}^\star(x)
\preceq
\nabla_{yy}^2\psi_\mu(x,y)
\preceq
\ell_{\psi,\lambda,1}^\eta H_{\lambda,\mu}^\star(x),
\qquad
y\in T_{2\eta,\mu}^{\lambda}(x).
\]

The mixed-Hessian bound follows similarly. Since the barrier does not depend on \(x\),
\(\nabla_{xy}^2\psi_\mu=\nabla_{xy}^2g\). Hence Assumption~\ref{ass:euclidean} and the
Euclidean--Dikin conversion give
\[
\|\nabla_{xy}^2\psi_\mu(x,y)\|_{2\to y_{\lambda,\mu}^\star(x),*}
\le
\kappa_\phi\ell_{g,1}.
\]
Thus the proxy mixed-Hessian bound holds after increasing
\(\ell_{\psi,\lambda,1}^\eta\).

It remains to justify the Lipschitz-Hessian bounds on the proxy neighborhood. Let
\(y_1,y_2\in T_{2\eta,\mu}^{\lambda}(x)\). For the \(g\)-part, the joint Lipschitzness of
\(\nabla^2g\), together with the Euclidean--Dikin norm conversion, gives
\[
\|\nabla_{yy}^2g(x,y_1)-\nabla_{yy}^2g(x,y_2)\|_{y_{\lambda,\mu}^\star(x)\to y_{\lambda,\mu}^\star(x),*}
\le
C_g^\eta\|y_1-y_2\|_{y_{\lambda,\mu}^\star(x)}
\]
and
\[
\|\nabla_{xy}^2g(x,y_1)-\nabla_{xy}^2g(x,y_2)\|_{2\to y_{\lambda,\mu}^\star(x),*}
\le
C_g^\eta\|y_1-y_2\|_{y_{\lambda,\mu}^\star(x)}
\]
for a finite \(\eta\)-dependent constant \(C_g^\eta\). For the barrier part, the logarithmic barrier is
self-concordant, so on the Dikin ball \(T_{2\eta,\mu}^{\lambda}(x)\) its Hessian is locally Lipschitz
in the proxy Dikin norm. In particular, there exists a finite constant \(C_\phi^\eta\) such that
\[
\|\nabla^2\phi(y_1)-\nabla^2\phi(y_2)\|_{y_{\lambda,\mu}^\star(x)\to y_{\lambda,\mu}^\star(x),*}
\le
C_\phi^\eta\|y_1-y_2\|_{y_{\lambda,\mu}^\star(x)} .
\]
Combining the \(g\)-part and the barrier part gives the desired Lipschitz-Hessian bound for
\(\nabla_{yy}^2\psi_\mu\). Since \(\nabla_{xy}^2\psi_\mu=\nabla_{xy}^2g\), the corresponding
mixed-Hessian Lipschitz bound follows from the preceding \(g\)-part estimate. Hence there exists a
finite constant \(\ell_{\psi,\lambda,2}^\eta<\infty\) such that all proxy neighborhood Lipschitz-Hessian bounds
in \eqref{ex:E4} hold.

The \(f\)-bounds on \(T_{2\eta,\mu}^{\lambda}(x)\) are obtained in exactly the same way, but without
the barrier term. Assumption~\ref{ass:euclidean} gives Euclidean bounds and Euclidean
Lipschitz-Hessian bounds for \(f\), while Lemma~\ref{lem:euc_dikin} converts them to the proxy
anchored Dikin norms. After increasing
\(\ell_{f,0}^\eta,\ell_{f,1}^\eta,\ell_{f,2}^\eta\), we obtain
\[
\|\nabla_y f(x,y)\|_{y_{\lambda,\mu}^\star(x),*}\le \ell_{f,0}^\eta,
\qquad
\|\nabla_{xy}^2 f(x,y)\|_{2\to y_{\lambda,\mu}^\star(x),*}\le \ell_{f,1}^\eta,
\]
and
\[
\|\nabla_{yy}^2 f(x,y)\|_{y_{\lambda,\mu}^\star(x)\to y_{\lambda,\mu}^\star(x),*}
\le
\ell_{f,1}^\eta,
\]
together with the analogous Lipschitz-Hessian bounds controlled by \(\ell_{f,2}^\eta\).

Finally, \eqref{eq:proxy_exact_anchor_comparison} also shows explicitly that the exact and proxy
anchors are uniformly comparable for all \(\lambda\ge\lambda_\eta\). If desired, one may also view
the proxy neighborhood as lying in an enlarged exact neighborhood: for
\(y\in T_{2\eta,\mu}^{\lambda}(x)\),
\[
\|y-y_\mu^\star(x)\|_{y_\mu^\star(x)}
\le
\|y_{\lambda,\mu}^\star(x)-y_\mu^\star(x)\|_{y_\mu^\star(x)}
+
\|y-y_{\lambda,\mu}^\star(x)\|_{y_\mu^\star(x)}
\le
\eta+\frac{2\eta}{1-\eta}.
\]
Thus the proxy neighborhood is contained in an enlarged exact Dikin neighborhood with radius
\(\eta+2\eta/(1-\eta)\). The proof above, however, works directly in the proxy anchor and therefore
only requires the proxy Dikin ball condition \(2\eta<1\). This proves \eqref{ex:E4}.

\end{proof}

\begin{remark}[Admissible multiplier range]
\label{rem:admissible_multiplier_range}
The proxy neighborhood regularity part of Proposition~\ref{prop:derived_dikin_regularity} requires the
initial multiplier to lie in a sufficiently large, problem-dependent range. More precisely, the proof
above exhibits a finite threshold \(\lambda_\eta<\infty\), depending only on the tube radius \(\eta\),
the barrier parameter \(\mu\), the geometry of \(Y\), and the local bounds in Assumption~\ref{ass:euclidean},
such that the proxy-center transfer and proxy neighborhood regularity bounds hold for every
\(\lambda\ge\lambda_0\) whenever \(\lambda_0\ge\lambda_\eta\). This threshold is not part of the \textit{barrier-aware} schedule design. It is a static admissibility condition
for the local proxy geometry. Once \(\lambda_0\ge\lambda_\eta\), the monotonicity requirement
\(\delta_k=\lambda_{k+1}-\lambda_k\ge0\) guarantees \(\lambda_k\ge\lambda_\eta\) for all outer
iterations \(k\). Thus all proxy neighborhood estimates used in the tracker and Lyapunov analysis remain
valid along the entire multiplier sequence. In the main text, this one-time condition is absorbed into
the phrase ``admissible multiplier range,'' while Definition~\ref{def:barrier_aware} records only the
dynamic schedule conditions controlling contraction, multiplier growth, and tube closure.
\end{remark}

\subsection{Local Dikin Curvature of the Proxy Objective}
\label{app:proxy_objective_curvature}

For notational simplicity in the sequel, we use common exact/proxy neighborhood constants. That is, after
Proposition~\ref{prop:derived_dikin_regularity}, we decrease \(\rho_\psi^\eta\) if necessary and
increase \(\ell_{\psi,1}^\eta,\ell_{\psi,2}^\eta\) if necessary so that the exact and proxy neighborhoods
regularity bounds both hold with the same constants whenever \(\lambda\ge\lambda_\eta\). In
particular, for every \(y\in T_{2\eta,\mu}^{\lambda}(x)\),
\begin{equation*}
\rho_\psi^\eta H_{\lambda,\mu}^\star(x)
\preceq
\nabla_{yy}^2\psi_\mu(x,y)
\preceq
\ell_{\psi,1}^\eta H_{\lambda,\mu}^\star(x),
\end{equation*}
with the corresponding mixed-Hessian and Lipschitz-Hessian bounds in the proxy anchored Dikin
norm. We also enlarge \(\ell_{f,0}^\eta,\ell_{f,1}^\eta,\ell_{f,2}^\eta\) if necessary so that the
\(f\)-bounds hold on both exact and proxy buffer tubes. All constants remain finite and depend only
on the quantities specified in Proposition~\ref{prop:derived_dikin_regularity}.

Using these unified constants, we can now read off the curvature of the proxy objective \(L_{\lambda,\mu}=f+\lambda\psi_\mu\) directly on the proxy buffer neighborhood. This is the form we use later to certify contraction of the \(y\)-tracker.

\begin{lemma}[Local curvature of \(L_{\lambda,\mu}\) on the proxy neighborhood]
\label{lem:proxy_objective_curvature}
Suppose the conclusions of Proposition~\ref{prop:derived_dikin_regularity} hold and
\[
\lambda\ge \lambda_\eta,
\qquad
\lambda\ge \frac{2\ell_{f,1}^\eta}{\rho_\psi^\eta}.
\]
Then, for every \(x\in X\) and every
\(y\in T_{2\eta,\mu}^{\lambda}(x)\),
\begin{equation*}
\frac{\lambda\rho_\psi^\eta}{2}H_{\lambda,\mu}^\star(x)
\preceq
\nabla_{yy}^2L_{\lambda,\mu}(x,y)
\preceq
\bigl(\lambda\ell_{\psi,1}^\eta+\ell_{f,1}^\eta\bigr)
H_{\lambda,\mu}^\star(x).
\end{equation*}
In particular, \(y\mapsto L_{\lambda,\mu}(x,y)\) is strongly convex on the proxy buffer neighborhood in the
proxy anchored Dikin metric.
\end{lemma}

\begin{proof}
By definition,
\begin{equation*}
\nabla_{yy}^2L_{\lambda,\mu}(x,y)
=
\nabla_{yy}^2 f(x,y)
+
\lambda\nabla_{yy}^2\psi_\mu(x,y).
\end{equation*}
On \(T_{2\eta,\mu}^{\lambda}(x)\), the common-constant convention gives
\(\rho_\psi^\eta H_{\lambda,\mu}^\star(x)\preceq
\nabla_{yy}^2\psi_\mu(x,y)\preceq
\ell_{\psi,1}^\eta H_{\lambda,\mu}^\star(x)\). The \(f\)-regularity bound gives, in quadratic-form
order,
\begin{equation*}
-\ell_{f,1}^\eta H_{\lambda,\mu}^\star(x)
\preceq
\nabla_{yy}^2 f(x,y)
\preceq
\ell_{f,1}^\eta H_{\lambda,\mu}^\star(x).
\end{equation*}
Therefore,
\begin{equation*}
\nabla_{yy}^2L_{\lambda,\mu}(x,y)
\preceq
\bigl(\lambda\ell_{\psi,1}^\eta+\ell_{f,1}^\eta\bigr)
H_{\lambda,\mu}^\star(x),
\end{equation*}
and
\begin{equation*}
\nabla_{yy}^2L_{\lambda,\mu}(x,y)
\succeq
\bigl(\lambda\rho_\psi^\eta-\ell_{f,1}^\eta\bigr)
H_{\lambda,\mu}^\star(x).
\end{equation*}
Since \(\lambda\ge 2\ell_{f,1}^\eta/\rho_\psi^\eta\), we have
\(\lambda\rho_\psi^\eta-\ell_{f,1}^\eta\ge \lambda\rho_\psi^\eta/2\), which proves the claim.
\end{proof}

\section{Proof of Proposition~\ref{prop:local_consequences}}
\label{app:local_consequences}

This appendix proves the local consequences of Proposition~\ref{prop:derived_dikin_regularity}.
Throughout this section, all estimates are understood on the maintained exact and proxy Dikin neighborhoods,
and all constants are the corresponding local Dikin constants from Proposition~\ref{prop:derived_dikin_regularity}.

\paragraph{Proof roadmap.}
Proposition~\ref{prop:local_consequences}\textup{(i)} is the stability result for the proxy minimizer
map \((x,\lambda)\mapsto y_{\lambda,\mu}^\star(x)\). It combines the proxy curvature Lemma~\ref{lem:proxy_objective_curvature} for
\(L_{\lambda,\mu}\) with the local \(xy\)-cross control of \(f\) and \(\psi_\mu\), together with the
observation that the \(\lambda\)-dependence enters the optimality system only through the term
\(\lambda\nabla_y\psi_\mu(x,y)\). This yields a decomposition of the proxy minimizer drift into an
\(x\)-motion term and a \(1/\lambda\)-bias term. Proposition~\ref{prop:local_consequences}\textup{(ii)}
is the corresponding stability result for the exact barrierized minimizer map
\(x\mapsto y_\mu^\star(x)\). Proposition~\ref{prop:local_consequences}\textup{(iii)} then upgrades the stability of
\(y_\mu^\star(x)\) to local smoothness of the barrier-smoothed outer objective
\(F_\mu(x)=f(x,y_\mu^\star(x))\). The proof differentiates the exact barrierized minimizer map
through the optimality condition, controls its Jacobian by the local curvature and cross-regularity of
\(\psi_\mu\), and combines this with the local bounds for \(f\). Thus the Dikin estimates recover,
inside the maintained neighborhoods, the same type of outer smoothness estimate that is globally available in
the Euclidean unconstrained setting. Finally, Proposition~\ref{prop:local_consequences}\textup{(iv)} quantifies the error incurred by
replacing the true hypergradient \(\nabla F_\mu(x)\) with the first-order proxy based on
\(y_{\lambda,\mu}^\star(x)\) and \(y_\mu^\star(x)\). Its proof uses the envelope identity for
\(C_{\lambda,\mu}^\star(x):=\min_{y\in\operatorname{int}(Y)}L_{\lambda,\mu}(x,y)\), together with
Proposition~\ref{prop:local_consequences}\textup{(i)} and the local smoothness bounds of
\(\psi_\mu\) and \(f\), to show that the proxy-gradient bias decays at rate \(O(1/\lambda)\).

\subsection{Proxy Minimizer-map Stability}
\label{sec:proxy minimizer-map}

\begin{proof}
We prove the local estimate for admissible pairs
\((x_1,\lambda_1),(x_2,\lambda_2)\) such that
\(y^\star_{\lambda_2,\mu}(x_2)\in T_{2\eta,\mu}^{\lambda_1}(x_1)\). This is the case used in the
algorithmic analysis, and the condition is verified for consecutive algorithmic centers by the
tube-maintenance argument. By Proposition~\ref{prop:derived_dikin_regularity} and
Lemma~\ref{lem:proxy_objective_curvature}, for all \(x\in X\) and
\(y\in T_{2\eta,\mu}^{\lambda_1}(x)\),
\[
\nabla^2_{yy}L_{\lambda_1,\mu}(x,y)
=
\nabla^2_{yy} f(x,y)+\lambda_1\nabla^2_{yy}\psi_\mu(x,y)
\succeq
\frac{\lambda_1\rho_\psi^\eta}{2}H_{\lambda_1,\mu}^\star(x).
\]
Hence \(L_{\lambda_1,\mu}(x,\cdot)\) is
\(\lambda_1\rho_\psi^\eta/2\)-strongly convex with respect to the Dikin metric induced by
\(H_{\lambda_1,\mu}^\star(x)\).

By strong convexity,
\[
\frac{\lambda_1\rho_\psi^\eta}{2}
\bigl\|
y_{\lambda_2,\mu}^\star(x_2)-y_{\lambda_1,\mu}^\star(x_1)
\bigr\|_{y_{\lambda_1,\mu}^\star(x_1)}
\le
\bigl\|
\nabla_y L_{\lambda_1,\mu}(x_1,y_{\lambda_2,\mu}^\star(x_2))
\bigr\|_{y_{\lambda_1,\mu}^\star(x_1),*}.
\]

We expand the gradient:
\[
\nabla_y L_{\lambda_1,\mu}(x_1,y_{\lambda_2,\mu}^\star(x_2))
=
\nabla_y f(x_1,y_{\lambda_2,\mu}^\star(x_2))
+
\lambda_1\nabla_y\psi_\mu(x_1,y_{\lambda_2,\mu}^\star(x_2)).
\]
Add and subtract terms at \((x_2,y_{\lambda_2,\mu}^\star(x_2))\):
\begin{align*}
\nabla_y L_{\lambda_1,\mu}(x_1,y_{\lambda_2,\mu}^\star(x_2))
&=
\bigl(
\nabla_y f(x_1,y_{\lambda_2,\mu}^\star(x_2))
-
\nabla_y f(x_2,y_{\lambda_2,\mu}^\star(x_2))
\bigr) \\
&\quad
+
\lambda_1
\bigl(
\nabla_y\psi_\mu(x_1,y_{\lambda_2,\mu}^\star(x_2))
-
\nabla_y\psi_\mu(x_2,y_{\lambda_2,\mu}^\star(x_2))
\bigr) \\
&\quad
+
\bigl(
\nabla_y f(x_2,y_{\lambda_2,\mu}^\star(x_2))
+
\lambda_1\nabla_y\psi_\mu(x_2,y_{\lambda_2,\mu}^\star(x_2))
\bigr).
\end{align*}
Since \(y_{\lambda_2,\mu}^\star(x_2)\) satisfies the optimality condition
\[
\nabla_y f(x_2,y_{\lambda_2,\mu}^\star(x_2))
+
\lambda_2\nabla_y\psi_\mu(x_2,y_{\lambda_2,\mu}^\star(x_2))
=0,
\]
the last term becomes
\[
(\lambda_1-\lambda_2)
\nabla_y\psi_\mu(x_2,y_{\lambda_2,\mu}^\star(x_2)).
\]
Using the stationarity relation and the bound
\(\|\nabla_y f(x,y)\|_{y,*}\le \ell_{f,0}^\eta\), we have
\[
\bigl\|
\nabla_y\psi_\mu(x_2,y_{\lambda_2,\mu}^\star(x_2))
\bigr\|_{y_{\lambda_1,\mu}^\star(x_1),*}
=
\frac{1}{\lambda_2}
\bigl\|
\nabla_y f(x_2,y_{\lambda_2,\mu}^\star(x_2))
\bigr\|_{y_{\lambda_1,\mu}^\star(x_1),*}
\le
\frac{\ell_{f,0}^\eta}{\lambda_2}.
\]

Applying the local Dikin Lipschitz-in-\(x\) bounds in the Dikin dual norm,
\[
\bigl\|
\nabla_y f(x_1,y_{\lambda_2,\mu}^\star(x_2))
-
\nabla_y f(x_2,y_{\lambda_2,\mu}^\star(x_2))
\bigr\|_{y_{\lambda_1,\mu}^\star(x_1),*}
\le
\ell_{f,1}^\eta\|x_2-x_1\|_2,
\]
and
\[
\bigl\|
\nabla_y\psi_\mu(x_1,y_{\lambda_2,\mu}^\star(x_2))
-
\nabla_y\psi_\mu(x_2,y_{\lambda_2,\mu}^\star(x_2))
\bigr\|_{y_{\lambda_1,\mu}^\star(x_1),*}
\le
\ell_{\psi,1}^\eta\|x_2-x_1\|_2.
\]

Combining these estimates,
\begin{align*}
\bigl\|
\nabla_y L_{\lambda_1,\mu}(x_1,y_{\lambda_2,\mu}^\star(x_2))
\bigr\|_{y_{\lambda_1,\mu}^\star(x_1),*}
&\le
\ell_{f,1}^\eta\|x_2-x_1\|_2
+
\lambda_1\ell_{\psi,1}^\eta\|x_2-x_1\|_2 \\
&\quad
+
|\lambda_2-\lambda_1|\frac{\ell_{f,0}^\eta}{\lambda_2}.
\end{align*}
Substituting and dividing by \(\lambda_1\rho_\psi^\eta/2\) yields
\begin{align*}
\bigl\|
y_{\lambda_2,\mu}^\star(x_2)-y_{\lambda_1,\mu}^\star(x_1)
\bigr\|_{y_{\lambda_1,\mu}^\star(x_1)}
&\le
\frac{2|\lambda_2-\lambda_1|}{\lambda_1\lambda_2}
\frac{\ell_{f,0}^\eta}{\rho_\psi^\eta} \\
&\quad
+
\frac{2(\lambda_1\ell_{\psi,1}^\eta+\ell_{f,1}^\eta)}
{\lambda_1\rho_\psi^\eta}
\|x_2-x_1\|_2 .
\end{align*}
Equivalently,
\[
\bigl\|
y_{\lambda_2,\mu}^\star(x_2)-y_{\lambda_1,\mu}^\star(x_1)
\bigr\|_{y_{\lambda_1,\mu}^\star(x_1)}
\le
\frac{2\ell_{f,0}^{\eta}}{\rho_\psi^\eta}
\left|
\frac1{\lambda_2}-\frac1{\lambda_1}
\right|
+
\ell_{\lambda,0}^{\eta}\|x_2-x_1\|_2,
\]
for some
\[
\ell_{\lambda,0}^{\eta}
\le
\frac{2(\lambda_1\ell_{\psi,1}^{\eta}+\ell_{f,1}^{\eta})}
{\lambda_1\rho_\psi^\eta},
\qquad
\text{or}
\qquad
\ell_{\lambda,0}^{\eta}
\le
\frac{3\ell_{\psi,1}^{\eta}}{\rho_\psi^\eta}.
\]
The latter simplified bound follows because the admissible multiplier range gives
\(\lambda_1\ge 2\ell_{f,1}^{\eta}/\rho_\psi^\eta\), while
\(\rho_\psi^\eta\le \ell_{\psi,1}^{\eta}\); hence
\(\ell_{f,1}^{\eta}/\lambda_1\le \ell_{\psi,1}^{\eta}/2\).

In addition, note that
\[
y_\mu^\star(x)=\lim_{\lambda\to\infty}y_{\lambda,\mu}^\star(x).
\]
Taking \(x_1=x_2=x\), \(\lambda_2=\lambda\), and letting
\(\lambda_1\to\infty\) in the bound above, using continuity of the barrier Hessian at the limiting
anchor, gives that for any \(x\in X\) and any finite \(\lambda\ge\lambda_0\),
\begin{equation}\label{eq:dikin_penalty_bias}
\bigl\|
y_{\lambda,\mu}^\star(x)-y_\mu^\star(x)
\bigr\|_{y_\mu^\star(x)}
\le
\frac{2\,\ell_{f,0}^{\eta}}{\lambda\rho_\psi^\eta}.
\end{equation}
\end{proof}

\subsection{Exact Minimizer-map Stability}

\begin{proof}
We obtain the exact minimizer stability as the limiting case of the proxy minimizer stability.
Apply Proposition~\ref{prop:local_consequences}\textup{(i)} with
\(\lambda_1=\lambda_2=\lambda\). For every \(\lambda\ge\lambda_0\),
\[
\|y_{\lambda,\mu}^\star(x_2)-y_{\lambda,\mu}^\star(x_1)\|_{y_{\lambda,\mu}^\star(x_1)}
\le
\ell_{\lambda,0}^{\eta}\|x_2-x_1\|_2 .
\]
By the proxy--exact proximity bound,
\(y_{\lambda,\mu}^\star(x_i)\to y_\mu^\star(x_i)\) as \(\lambda\to\infty\) for \(i=1,2\).
Since the anchors remain in the interior and the barrier Hessian is continuous on \(\operatorname{int}(Y)\),
we also have
\[
H_{\lambda,\mu}^\star(x_1)
=
\nabla^2\phi(y_{\lambda,\mu}^\star(x_1))
\to
\nabla^2\phi(y_\mu^\star(x_1))
=
H_\mu^\star(x_1).
\]
Therefore,
\[
\|y_{\lambda,\mu}^\star(x_2)-y_{\lambda,\mu}^\star(x_1)\|_{y_{\lambda,\mu}^\star(x_1)}
\to
\|y_\mu^\star(x_2)-y_\mu^\star(x_1)\|_{y_\mu^\star(x_1)} .
\]
Passing to the limit in the proxy stability bound yields
\[
\|y_\mu^\star(x_2)-y_\mu^\star(x_1)\|_{y_\mu^\star(x_1)}
\le
\ell_{\lambda,0}^{\eta}\|x_2-x_1\|_2 .
\]
Thus Proposition~\ref{prop:local_consequences}\textup{(ii)} holds with
\(\ell_{*,0}^{\eta}:=\ell_{\lambda,0}^{\eta}\), after increasing the constant if necessary.
\end{proof}

\subsection{Local Smoothness of the Barrier-Smoothed Outer Objective}

\begin{proof}
Since \(\nabla_y\psi_\mu(x,y_\mu^\star(x))=0\) and \(\psi_\mu\) is strongly convex in \(y\), the
implicit function theorem applies and yields
\[
Dy_\mu^\star(x)
=
-\bigl(\nabla^2_{yy}\psi_\mu(x,y_\mu^\star(x))\bigr)^{-1}
\nabla^2_{xy}\psi_\mu(x,y_\mu^\star(x)).
\]
By the chain rule,
\begin{align}
\nabla F_\mu(x)
&=
\nabla_x f(x,y_\mu^\star(x))
+
\bigl(Dy_\mu^\star(x)\bigr)^\top \nabla_y f(x,y_\mu^\star(x))
\nonumber\\
&=
\nabla_x f(x,y_\mu^\star(x))
-
\nabla^2_{xy}\psi_\mu(x,y_\mu^\star(x))^\top
\bigl(\nabla^2_{yy}\psi_\mu(x,y_\mu^\star(x))\bigr)^{-1}
\nabla_y f(x,y_\mu^\star(x)).
\label{eq:local_smooth_gradF}
\end{align}
By Proposition~\ref{prop:local_consequences}\textup{(ii)} and the Euclidean--Dikin norm conversion,
after enlarging \(\ell_{*,0}^\eta\) if necessary, we have
\begin{equation}
\label{eq:local_smooth_y_lip}
\|y_\mu^\star(x_2)-y_\mu^\star(x_1)\|_2
\le
\ell_{*,0}^\eta\|x_2-x_1\|_2.
\end{equation}
From \eqref{eq:local_smooth_gradF},
\[
\nabla F_\mu(x_2)-\nabla F_\mu(x_1)
=
\underbrace{
\Bigl(\nabla_x f(x_2,y_\mu^\star(x_2))
-
\nabla_x f(x_1,y_\mu^\star(x_1))\Bigr)
}_{(I)}
\]
\[
-
\underbrace{
\Bigl(
\nabla^2_{xy}\psi_\mu(x_2,y_\mu^\star(x_2))^\top
\bigl(\nabla^2_{yy}\psi_\mu(x_2,y_\mu^\star(x_2))\bigr)^{-1}
\nabla_y f(x_2,y_\mu^\star(x_2))
\Bigr)
}_{\text{first implicit term}}
\]
\[
+
\underbrace{
\Bigl(
\nabla^2_{xy}\psi_\mu(x_1,y_\mu^\star(x_1))^\top
\bigl(\nabla^2_{yy}\psi_\mu(x_1,y_\mu^\star(x_1))\bigr)^{-1}
\nabla_y f(x_1,y_\mu^\star(x_1))
\Bigr)
}_{\text{second implicit term}}.
\]
Equivalently, writing the difference of the two implicit terms as \((II)\), we have
\[
\nabla F_\mu(x_2)-\nabla F_\mu(x_1)=(I)-(II).
\]

\smallskip
\noindent\emph{Term \((I)\).}
By the Hessian bound for \(f\) in Assumption~\ref{ass:euclidean},
\[
\|(I)\|_2
\le
\ell_{f,1}
\Bigl(\|x_2-x_1\|_2+\|y_\mu^\star(x_2)-y_\mu^\star(x_1)\|_2\Bigr)
\le
\ell_{f,1}\left(1+\ell_{*,0}^\eta\right)\|x_2-x_1\|_2.
\]
With a slight abuse of notation, we absorb \(1\) into \(\ell_{*,0}^\eta\) to obtain
\[
\|(I)\|_2
\le
\ell_{f,1}\ell_{*,0}^\eta \|x_2-x_1\|_2.
\]

\smallskip
\noindent\emph{Term \((II)\).}
Add and subtract intermediate terms:
\begin{align*}
(II)
&=
\Bigl(
\nabla^2_{xy}\psi_\mu(x_2,y_\mu^\star(x_2))
-
\nabla^2_{xy}\psi_\mu(x_1,y_\mu^\star(x_1))
\Bigr)^\top
\bigl(\nabla^2_{yy}\psi_\mu(x_2,y_\mu^\star(x_2))\bigr)^{-1}
\nabla_y f(x_2,y_\mu^\star(x_2))
\\
&\quad
+
\nabla^2_{xy}\psi_\mu(x_1,y_\mu^\star(x_1))^\top
\Bigl(
\bigl(\nabla^2_{yy}\psi_\mu(x_2,y_\mu^\star(x_2))\bigr)^{-1}
-
\bigl(\nabla^2_{yy}\psi_\mu(x_1,y_\mu^\star(x_1))\bigr)^{-1}
\Bigr)
\nabla_y f(x_2,y_\mu^\star(x_2))
\\
&\quad
+
\nabla^2_{xy}\psi_\mu(x_1,y_\mu^\star(x_1))^\top
\bigl(\nabla^2_{yy}\psi_\mu(x_1,y_\mu^\star(x_1))\bigr)^{-1}
\Bigl(
\nabla_y f(x_2,y_\mu^\star(x_2))
-
\nabla_y f(x_1,y_\mu^\star(x_1))
\Bigr)
\\
&=:\ (IIa)+(IIb)+(IIc).
\end{align*}

We will use the following duality consequence of the cross-regularity bound in
Proposition~\ref{prop:derived_dikin_regularity}: for any
\((x,y)\in X\times T_{\eta,\mu}(x)\) and any \(v\in\mathbb R^{d_y}\),
\begin{equation}
\label{eq:local_smooth_cross_duality}
\|\nabla^2_{xy}\psi_\mu(x,y)^\top v\|_2
\le
\|\nabla^2_{xy}\psi_\mu(x,y)\|_{2\to y_\mu^\star(x),*}\,
\|v\|_{y_\mu^\star(x)}
\le
\ell_{\psi,1}^\eta\,\|v\|_{y_\mu^\star(x)}.
\end{equation}

\smallskip
\noindent\emph{Bound for \((IIa)\).}
By duality in the anchored Dikin geometry at \(y_\mu^\star(x_1)\),
\[
\|(IIa)\|_2
\le
\Bigl\|
\nabla^2_{xy}\psi_\mu(x_2,y_\mu^\star(x_2))
-
\nabla^2_{xy}\psi_\mu(x_1,y_\mu^\star(x_1))
\Bigr\|_{2\to y_\mu^\star(x_1),*}
\]
\[
\qquad\qquad\cdot
\Bigl\|
\bigl(\nabla^2_{yy}\psi_\mu(x_2,y_\mu^\star(x_2))\bigr)^{-1}
\nabla_y f(x_2,y_\mu^\star(x_2))
\Bigr\|_{y_\mu^\star(x_1)}.
\]
Using cross-Hessian Lipschitzness along the minimizer graph in
Proposition~\ref{prop:derived_dikin_regularity},
\[
\Bigl\|
\nabla^2_{xy}\psi_\mu(x_2,y_\mu^\star(x_2))
-
\nabla^2_{xy}\psi_\mu(x_1,y_\mu^\star(x_1))
\Bigr\|_{2\to y_\mu^\star(x_1),*}
\le
\ell_{\psi,2}^\eta\|x_2-x_1\|_2.
\]
It remains to bound
\[
\Bigl\|
\bigl(\nabla^2_{yy}\psi_\mu(x_2,y_\mu^\star(x_2))\bigr)^{-1}
\nabla_y f(x_2,y_\mu^\star(x_2))
\Bigr\|_{y_\mu^\star(x_1)}.
\]
By local Dikin curvature and anchor-switch equivalence,
\[
\Bigl\|
\bigl(\nabla^2_{yy}\psi_\mu(x_2,y_\mu^\star(x_2))\bigr)^{-1}
\nabla_y f(x_2,y_\mu^\star(x_2))
\Bigr\|_{y_\mu^\star(x_1)}
\le
\frac{1}{\rho_\psi^\eta}
\|\nabla_y f(x_2,y_\mu^\star(x_2))\|_{y_\mu^\star(x_2),*}
\le
\frac{\ell^\eta_{f,0}}{\rho_\psi^\eta}.
\]
Therefore,
\begin{equation}
\label{eq:local_smooth_IIa}
\|(IIa)\|_2
\le
\frac{\ell_{\psi,2}^\eta\,\ell_{f,0}^\eta}{\rho_\psi^\eta}
\|x_2-x_1\|_2.
\end{equation}

\smallskip
\noindent\emph{Bound for \((IIb)\).}
We have
\[
(IIb)
=
\nabla^2_{xy}\psi_\mu(x_1,y_\mu^\star(x_1))^\top
\Bigl(
\bigl(\nabla^2_{yy}\psi_\mu(x_2,y_\mu^\star(x_2))\bigr)^{-1}
-
\bigl(\nabla^2_{yy}\psi_\mu(x_1,y_\mu^\star(x_1))\bigr)^{-1}
\Bigr)
\nabla_y f(x_2,y_\mu^\star(x_2)).
\]
Applying \eqref{eq:local_smooth_cross_duality} at
\((x_1,y_\mu^\star(x_1))\) yields
\[
\|(IIb)\|_2
\le
\ell_{\psi,1}^\eta
\Bigl\|
\Bigl(
\bigl(\nabla^2_{yy}\psi_\mu(x_2,y_\mu^\star(x_2))\bigr)^{-1}
-
\bigl(\nabla^2_{yy}\psi_\mu(x_1,y_\mu^\star(x_1))\bigr)^{-1}
\Bigr)
\nabla_y f(x_2,y_\mu^\star(x_2))
\Bigr\|_{y_\mu^\star(x_1)}.
\]
Using the resolvent identity,
\[
\bigl(\nabla^2_{yy}\psi_\mu(x_2,y_\mu^\star(x_2))\bigr)^{-1}
-
\bigl(\nabla^2_{yy}\psi_\mu(x_1,y_\mu^\star(x_1))\bigr)^{-1}
\]
\[
=
\bigl(\nabla^2_{yy}\psi_\mu(x_2,y_\mu^\star(x_2))\bigr)^{-1}
\Bigl(
\nabla^2_{yy}\psi_\mu(x_1,y_\mu^\star(x_1))
-
\nabla^2_{yy}\psi_\mu(x_2,y_\mu^\star(x_2))
\Bigr)
\bigl(\nabla^2_{yy}\psi_\mu(x_1,y_\mu^\star(x_1))\bigr)^{-1}.
\]
Therefore,
\[
\Bigl\|
\Bigl(
\bigl(\nabla^2_{yy}\psi_\mu(x_2,y_\mu^\star(x_2))\bigr)^{-1}
-
\bigl(\nabla^2_{yy}\psi_\mu(x_1,y_\mu^\star(x_1))\bigr)^{-1}
\Bigr)
\nabla_y f(x_2,y_\mu^\star(x_2))
\Bigr\|_{y_\mu^\star(x_1)}
\]
\[
\le
\Bigl\|
\bigl(\nabla^2_{yy}\psi_\mu(x_2,y_\mu^\star(x_2))\bigr)^{-1}
\Bigr\|_{y_\mu^\star(x_1)\leftarrow y_\mu^\star(x_1),*}
\]
\[
\quad\cdot
\Bigl\|
\nabla^2_{yy}\psi_\mu(x_1,y_\mu^\star(x_1))
-
\nabla^2_{yy}\psi_\mu(x_2,y_\mu^\star(x_2))
\Bigr\|_{y_\mu^\star(x_1),*\to y_\mu^\star(x_1)}
\]
\[
\quad\cdot
\Bigl\|
\bigl(\nabla^2_{yy}\psi_\mu(x_1,y_\mu^\star(x_1))\bigr)^{-1}
\nabla_y f(x_2,y_\mu^\star(x_2))
\Bigr\|_{y_\mu^\star(x_1)}.
\]
By local Dikin curvature and anchor-switch,
\[
\Bigl\|
\bigl(\nabla^2_{yy}\psi_\mu(x_2,y_\mu^\star(x_2))\bigr)^{-1}
\Bigr\|_{y_\mu^\star(x_1)\leftarrow y_\mu^\star(x_1),*}
\le
\frac{1}{\rho_\psi^\eta},
\]
and
\[
\Bigl\|
\bigl(\nabla^2_{yy}\psi_\mu(x_1,y_\mu^\star(x_1))\bigr)^{-1}
\nabla_y f(x_2,y_\mu^\star(x_2))
\Bigr\|_{y_\mu^\star(x_1)}
\le
\frac{1}{\rho_\psi^\eta}
\|\nabla_y f(x_2,y_\mu^\star(x_2))\|_{y_\mu^\star(x_1),*}
\le
\frac{\ell_{f,0}^\eta}{\rho_\psi^\eta}.
\]
Finally, by the local Hessian Lipschitzness and the drift bound
\eqref{eq:local_smooth_y_lip},
\[
\Bigl\|
\nabla^2_{yy}\psi_\mu(x_1,y_\mu^\star(x_1))
-
\nabla^2_{yy}\psi_\mu(x_2,y_\mu^\star(x_2))
\Bigr\|_{y_\mu^\star(x_1),*\to y_\mu^\star(x_1)}
\]
\[
\le
\ell_{\psi,2}^\eta
\Bigl(
\|x_2-x_1\|_2
+
\|y_\mu^\star(x_2)-y_\mu^\star(x_1)\|_{y_\mu^\star(x_1)}
\Bigr)
\le
\ell_{\psi,2}^\eta\ell_{*,0}^\eta\|x_2-x_1\|_2.
\]
Putting these together gives
\begin{equation}
\label{eq:local_smooth_IIb}
\|(IIb)\|_2
\le
\frac{\ell_{\psi,1}^\eta\,\ell_{\psi,2}^\eta\,\ell_{f,0}^\eta}
{(\rho_\psi^\eta)^2}
\ell_{*,0}^\eta
\|x_2-x_1\|_2.
\end{equation}

\smallskip
\noindent\emph{Bound for \((IIc)\).}
We have
\[
(IIc)
=
\nabla^2_{xy}\psi_\mu(x_1,y_\mu^\star(x_1))^\top
\bigl(\nabla^2_{yy}\psi_\mu(x_1,y_\mu^\star(x_1))\bigr)^{-1}
\Bigl(
\nabla_y f(x_2,y_\mu^\star(x_2))
-
\nabla_y f(x_1,y_\mu^\star(x_1))
\Bigr).
\]
Applying \eqref{eq:local_smooth_cross_duality} at
\((x_1,y_\mu^\star(x_1))\) gives
\[
\|(IIc)\|_2
\le
\ell_{\psi,1}^\eta
\Bigl\|
\bigl(\nabla^2_{yy}\psi_\mu(x_1,y_\mu^\star(x_1))\bigr)^{-1}
\Bigl(
\nabla_y f(x_2,y_\mu^\star(x_2))
-
\nabla_y f(x_1,y_\mu^\star(x_1))
\Bigr)
\Bigr\|_{y_\mu^\star(x_1)}
\]
\[
\le
\frac{\ell_{\psi,1}^\eta}{\rho_\psi^\eta}
\Bigl\|
\nabla_y f(x_2,y_\mu^\star(x_2))
-
\nabla_y f(x_1,y_\mu^\star(x_1))
\Bigr\|_{y_\mu^\star(x_1),*}.
\]
Using joint smoothness of \(f\) in \((x,y)\),
\[
\Bigl\|
\nabla_y f(x_2,y_\mu^\star(x_2))
-
\nabla_y f(x_1,y_\mu^\star(x_1))
\Bigr\|_{y_\mu^\star(x_1),*}
\le
\ell_{f,1}^\eta
\Bigl(
\|x_2-x_1\|_2
+
\|y_\mu^\star(x_2)-y_\mu^\star(x_1)\|_2
\Bigr).
\]
Then \eqref{eq:local_smooth_y_lip} yields
\[
\|y_\mu^\star(x_2)-y_\mu^\star(x_1)\|_2
\le
\ell_{*,0}^\eta\|x_2-x_1\|_2.
\]
Therefore,
\begin{equation}
\label{eq:local_smooth_IIc}
\|(IIc)\|_2
\le
\frac{\ell_{\psi,1}^\eta\,\ell_{f,1}^\eta}{\rho_\psi^\eta}
\ell_{*,0}^\eta
\|x_2-x_1\|_2.
\end{equation}

Combining the bounds for \((I)\), \eqref{eq:local_smooth_IIa},
\eqref{eq:local_smooth_IIb}, and \eqref{eq:local_smooth_IIc}, we obtain
\[
\|\nabla F_\mu(x_2)-\nabla F_\mu(x_1)\|_2
\le
L_F^\eta\|x_2-x_1\|_2,
\]
where an explicit admissible choice is
\begin{equation}
\label{eq:LF_explicit}
L_F^\eta
:=
\ell_{f,1}\ell_{*,0}^\eta
+
\frac{\ell_{\psi,2}^\eta\,\ell_{f,0}^\eta}{\rho_\psi^\eta}
+
\frac{\ell_{\psi,1}^\eta\,\ell_{f,1}^\eta}{\rho_\psi^\eta}
\ell_{*,0}^\eta
+
\frac{\ell_{\psi,1}^\eta\,\ell_{\psi,2}^\eta\,\ell_{f,0}^\eta}
{(\rho_\psi^\eta)^2}
\ell_{*,0}^\eta .
\end{equation}
Using the fact \(\ell_{\psi,1}^\eta/\rho_\psi^\eta\ge 1\), we get the simpler sufficient bound
\begin{equation}
\label{eq:LF_simplified}
L_F^\eta
:=
\left(
\ell_{f,1}
+
\frac{\ell_{\psi,1}^\eta\,\ell_{f,1}^\eta}{\rho_\psi^\eta}
+
2\frac{\ell_{\psi,1}^\eta\,\ell_{\psi,2}^\eta\,\ell_{f,0}^\eta}
{(\rho_\psi^\eta)^2}
\right)
\ell_{*,0}^\eta .
\end{equation}
\end{proof}

\subsection{Local Proxy-gradient Bias}

\begin{proof}
By definition,
\[
C_{\lambda,\mu}^\star(x)
=
\min_{y\in\operatorname{int}(Y)}
\Bigl(f(x,y)+\lambda\psi_\mu(x,y)\Bigr)
-
\lambda\psi_\mu^\star(x),
\qquad
\psi_\mu^\star(x):=\min_{y\in\operatorname{int}(Y)}\psi_\mu(x,y).
\]
By the envelope theorem for the first minimization, since
\(\nabla_y(f+\lambda\psi_\mu)(x,y_{\lambda,\mu}^\star(x))=0\),
\[
\nabla_x
\min_{y\in\operatorname{int}(Y)}
\bigl(f(x,y)+\lambda\psi_\mu(x,y)\bigr)
=
\nabla_x f(x,y_{\lambda,\mu}^\star(x))
+
\lambda\nabla_x\psi_\mu(x,y_{\lambda,\mu}^\star(x)).
\]
For the second term, since \(y_\mu^\star(x)\) minimizes \(\psi_\mu(x,\cdot)\), the envelope theorem gives
\[
\nabla_x\psi_\mu^\star(x)=\nabla_x\psi_\mu(x,y_\mu^\star(x)).
\]
Subtracting yields
\begin{equation}
\label{eq:proxy_envelope_identity}
\nabla C_{\lambda,\mu}^\star(x)
=
\nabla_x f(x,y_{\lambda,\mu}^\star(x))
+
\lambda\Bigl(
\nabla_x\psi_\mu(x,y_{\lambda,\mu}^\star(x))
-
\nabla_x\psi_\mu(x,y_\mu^\star(x))
\Bigr).
\end{equation}

\smallskip
We then apply Lemma~\ref{lem:proxy_direction_linearization} with \(y=y_{\lambda,\mu}^\star(x)\). Because
\(y_{\lambda,\mu}^\star(x)\) minimizes \(L_{\lambda,\mu}(x,\cdot)\), we have
\[
\nabla_y L_{\lambda,\mu}(x,y_{\lambda,\mu}^\star(x))=0,
\]
so the correction term in Lemma~\ref{lem:proxy_direction_linearization} vanishes. Therefore,
\begin{equation}
\label{eq:proxy_bias_preliminary}
\begin{aligned}
&\|\nabla F_\mu(x)-\nabla_x L_{\lambda,\mu}(x,y_{\lambda,\mu}^\star(x))\|_2 \\
&\qquad\le
2\Bigl(\frac{\ell_{\psi,1}^\eta}{\rho_\psi^\eta}\Bigr)
\|y_{\lambda,\mu}^\star(x)-y_\mu^\star(x)\|_{y_\mu^\star(x)}
\Bigl(
\ell_{f,1}^\eta
+
\lambda\min\{
2\ell_{\psi,1}^\eta,\,
\ell_{\psi,2}^\eta
\|y_{\lambda,\mu}^\star(x)-y_\mu^\star(x)\|_{y_\mu^\star(x)}
\}
\Bigr).
\end{aligned}
\end{equation}
By \eqref{eq:proxy_envelope_identity},
\[
\nabla C_{\lambda,\mu}^\star(x)
=
\nabla_x L_{\lambda,\mu}(x,y_{\lambda,\mu}^\star(x)),
\]
so \eqref{eq:proxy_bias_preliminary} is a bound on
\(\|\nabla F_\mu(x)-\nabla C_{\lambda,\mu}^\star(x)\|_2\).

\smallskip
\noindent We then bound
\(\|y_{\lambda,\mu}^\star(x)-y_\mu^\star(x)\|_{y_\mu^\star(x)}\) by \(O(1/\lambda)\).
The optimality conditions are
\[
\nabla_y\psi_\mu(x,y_\mu^\star(x))=0,
\qquad
\nabla_y f(x,y_{\lambda,\mu}^\star(x))
+
\lambda\nabla_y\psi_\mu(x,y_{\lambda,\mu}^\star(x))=0.
\]
Using strong monotonicity of \(\nabla_y\psi_\mu(x,\cdot)\) on the neighborhood in the anchored Dikin geometry
(Proposition~\ref{prop:derived_dikin_regularity}), we have
\[
\rho_\psi^\eta
\|y_{\lambda,\mu}^\star(x)-y_\mu^\star(x)\|_{y_\mu^\star(x)}^2
\le
\big\langle
\nabla_y\psi_\mu(x,y_{\lambda,\mu}^\star(x))
-
\nabla_y\psi_\mu(x,y_\mu^\star(x)),
\,y_{\lambda,\mu}^\star(x)-y_\mu^\star(x)
\big\rangle.
\]
Since \(\nabla_y\psi_\mu(x,y_\mu^\star(x))=0\), this becomes
\[
\rho_\psi^\eta
\|y_{\lambda,\mu}^\star(x)-y_\mu^\star(x)\|_{y_\mu^\star(x)}^2
\le
\big\langle
\nabla_y\psi_\mu(x,y_{\lambda,\mu}^\star(x)),
\,y_{\lambda,\mu}^\star(x)-y_\mu^\star(x)
\big\rangle.
\]
Substitute
\[
\nabla_y\psi_\mu(x,y_{\lambda,\mu}^\star(x))
=
-\frac{1}{\lambda}\nabla_y f(x,y_{\lambda,\mu}^\star(x))
\]
and apply Cauchy--Schwarz in the anchored Dikin primal/dual pair:
\[
\rho_\psi^\eta
\|y_{\lambda,\mu}^\star(x)-y_\mu^\star(x)\|_{y_\mu^\star(x)}^2
\le
\frac{1}{\lambda}
\|\nabla_y f(x,y_{\lambda,\mu}^\star(x))\|_{y_\mu^\star(x),*}
\|y_{\lambda,\mu}^\star(x)-y_\mu^\star(x)\|_{y_\mu^\star(x)}.
\]
Canceling one factor gives
\[
\|y_{\lambda,\mu}^\star(x)-y_\mu^\star(x)\|_{y_\mu^\star(x)}
\le
\frac{1}{\lambda\rho_\psi^\eta}
\|\nabla_y f(x,y_{\lambda,\mu}^\star(x))\|_{y_\mu^\star(x),*}.
\]
By Assumption~\ref{ass:euclidean}, \(\|\nabla_y f(x,y)\|_2\le \ell_{f,0}\), and by the
Euclidean--Dikin norm conversion, after absorbing the geometric conversion factor into
\(\ell_{f,0}^\eta\),
\begin{equation}
\label{eq:proxy_center_bias_rate}
\|y_{\lambda,\mu}^\star(x)-y_\mu^\star(x)\|_{y_\mu^\star(x)}
\le
\frac{\ell_{f,0}^\eta}{\lambda\rho_\psi^\eta}.
\end{equation}

\smallskip
\noindent We then remove the \(\lambda\)-dependence from the bracket in
\eqref{eq:proxy_bias_preliminary}.
Using \(\min\{a,b\}\le b\) and \eqref{eq:proxy_center_bias_rate},
\[
\lambda
\min\Bigl\{
2\ell_{\psi,1}^\eta,\,
\ell_{\psi,2}^\eta
\|y_{\lambda,\mu}^\star(x)-y_\mu^\star(x)\|_{y_\mu^\star(x)}
\Bigr\}
\le
\lambda\ell_{\psi,2}^\eta
\|y_{\lambda,\mu}^\star(x)-y_\mu^\star(x)\|_{y_\mu^\star(x)}
\le
\frac{\ell_{\psi,2}^\eta\ell_{f,0}^\eta}{\rho_\psi^\eta}.
\]
Thus the bracket in \eqref{eq:proxy_bias_preliminary} is bounded by the local Dikin constant
\[
\ell_{f,1}^\eta
+
\frac{\ell_{\psi,2}^\eta\ell_{f,0}^\eta}{\rho_\psi^\eta}.
\]
Plugging this bound and \eqref{eq:proxy_center_bias_rate} into
\eqref{eq:proxy_bias_preliminary} yields
\[
\|\nabla F_\mu(x)-\nabla C_{\lambda,\mu}^\star(x)\|_2
\le
2\Bigl(\frac{\ell_{\psi,1}^\eta}{\rho_\psi^\eta}\Bigr)
\cdot
\frac{\ell_{f,0}^\eta}{\rho_\psi^\eta\lambda}
\cdot
\Bigl(
\ell_{f,1}^\eta
+
\frac{\ell_{\psi,2}^\eta\ell_{f,0}^\eta}{\rho_\psi^\eta}
\Bigr)
=
\frac{c_x^\eta}{\lambda},
\]
where an explicit admissible constant is
\begin{equation}
\label{eq:cx_eta_explicit}
c_x^\eta
:=
2
\Bigl(
\frac{\ell_{\psi,1}^\eta\ell_{f,0}^\eta}{(\rho_\psi^\eta)^2}
\Bigr)
\Bigl(
\ell_{f,1}^\eta
+
\frac{\ell_{\psi,2}^\eta\ell_{f,0}^\eta}{\rho_\psi^\eta}
\Bigr).
\end{equation}
This proves Proposition~\ref{prop:local_consequences}\textup{(iv)}.
\end{proof}

\section{Tracker Contraction}
\label{app:tracker}

We now analyze the two inner trackers under the local Dikin geometry. The tracker argument has two
layers. First, with \(x_k\) and \(\lambda_k\) fixed, the exact and proxy lower objectives are locally
well conditioned on their target neighborhoods, so repeated inner steps contract toward the corresponding
fixed centers. Second, after the outer update, the centers themselves move:
\(y_\mu^\star(x_k)\) changes with \(x_k\), while
\(y_{\lambda_k,\mu}^\star(x_k)\) changes with both \(x_k\) and \(\lambda_k\). The moving-center
recursions transfer the fixed-center contraction from the old anchors to the new anchors and produce
the drift terms controlled by the \textit{barrier-aware} schedule.

There is one technical point before the tracker proofs. Algorithm~\ref{alg:dikin_sbo} uses the
implementable frozen barrier-metrics \(\nabla^2\phi(z_k), \nabla^2\phi(y_k)\)
during the \(T\) inner steps of outer iteration \(k\). The analysis, however, measures tracker errors in
the unknown center-anchor metrics
\[
H_\mu^\star(x_k)=\nabla^2\phi(y_\mu^\star(x_k)),
\qquad
H_{\lambda_k,\mu}^\star(x_k)
=
\nabla^2\phi(y_{\lambda_k,\mu}^\star(x_k)).
\]
The next subsection justifies this reduction. The point is that, once the warm starts lie in the target
neighborhoods, the frozen metrics are spectrally equivalent to the corresponding center anchors by
self-concordant metric change. Moreover, a frozen-metric step has the same contraction form in the
center-anchor norm, up to fixed \(\eta\)-dependent constants. After this reduction, all later tracker
proofs can be written directly in the clean center-anchor Dikin norms.

\subsection{Frozen Metrics and Center-anchor Reduction}
\label{app:frozen_metric_reduction}

\begin{proposition}[Frozen Dikin metrics are center-anchor equivalent]
\label{prop:frozen_metric_equivalence}
Fix an outer iteration \(k\). Suppose
\(z_k\in T_{\eta,\mu}(x_k)\) and
\(y_k\in T_{\eta,\mu}^{\lambda_k}(x_k)\),
then
\[
(1-\eta)^2H_\mu^\star(x_k)
\preceq
\nabla^2\phi(z_k)
\preceq
(1-\eta)^{-2}H_\mu^\star(x_k),
\]
and
\[
(1-\eta)^2H_{\lambda_k,\mu}^\star(x_k)
\preceq
\nabla^2\phi(y_k)
\preceq
(1-\eta)^{-2}H_{\lambda_k,\mu}^\star(x_k).
\]
Consequently, after replacing the local Dikin constants by fixed \(\eta\)-dependent multiples, the
exact and proxy curvature, smoothness, and Lipschitz-Hessian bounds of
Proposition~\ref{prop:derived_dikin_regularity} also hold with respect to the frozen metrics
\(\nabla^2\phi(z_k)\) and \(\nabla^2\phi(y_k)\). In particular, on the exact buffer neighborhood,
\[
\bar\rho_\psi^\eta \nabla^2\phi(z_k)
\preceq
\nabla_{yy}^2\psi_\mu(x_k,y)
\preceq
\bar\ell_{\psi,1}^\eta \nabla^2\phi(z_k),
\]
and on the proxy buffer neighborhood,
\[
\frac{\lambda_k\bar\rho_\psi^\eta}{2}\nabla^2\phi(y_k)
\preceq
\nabla_{yy}^2L_{\lambda_k,\mu}(x_k,y)
\preceq
\bigl(\lambda_k\bar\ell_{\psi,1}^\eta+\bar\ell_{f,1}^\eta\bigr)\nabla^2\phi(y_k),
\]
where one may take
\[
\bar\rho_\psi^\eta:=(1-\eta)^2\rho_\psi^\eta,
\qquad
\bar\ell_{\psi,1}^\eta:=(1-\eta)^{-2}\ell_{\psi,1}^\eta,
\qquad
\bar\ell_{f,1}^\eta:=(1-\eta)^{-2}\ell_{f,1}^\eta.
\]
\end{proposition}

\begin{proof}
We first prove the exact-tracker metric comparison. Since
\(z_k\in T_{\eta,\mu}(x_k)\), by definition
\[
\|z_k-y_\mu^\star(x_k)\|_{y_\mu^\star(x_k)}\le \eta < 1.
\]
Applying the self-concordant metric-change lemma with
\(y_1=y_\mu^\star(x_k)\) and \(y_2=z_k\) gives
\[
(1-\eta)^2\nabla^2\phi(y_\mu^\star(x_k))
\preceq
\nabla^2\phi(z_k)
\preceq
(1-\eta)^{-2}\nabla^2\phi(y_\mu^\star(x_k)).
\]
Since \(H_\mu^\star(x_k)=\nabla^2\phi(y_\mu^\star(x_k))\)
, this proves
\[
(1-\eta)^2H_\mu^\star(x_k)
\preceq
\nabla^2\phi(z_k)
\preceq
(1-\eta)^{-2}H_\mu^\star(x_k).
\]

The proxy-tracker comparison is identical. Since
\(y_k\in T_{\eta,\mu}^{\lambda_k}(x_k)\),
\[
\|y_k-y_{\lambda_k,\mu}^\star(x_k)\|_{y_{\lambda_k,\mu}^\star(x_k)}
\le \eta .
\]
Applying the same metric-change lemma with
\(y_1=y_{\lambda_k,\mu}^\star(x_k)\) and \(y_2=y_k\) gives
\[
(1-\eta)^2H_{\lambda_k,\mu}^\star(x_k)
\preceq
\nabla^2\phi(y_k)
\preceq
(1-\eta)^{-2}H_{\lambda_k,\mu}^\star(x_k).
\]

We now transfer the curvature bounds. On the exact buffer neighborhood,
Proposition~\ref{prop:derived_dikin_regularity} gives
\[
\rho_\psi^\eta H_\mu^\star(x_k)
\preceq
\nabla_{yy}^2\psi_\mu(x_k,y)
\preceq
\ell_{\psi,1}^\eta H_\mu^\star(x_k).
\]
From the metric comparison above,
\[
H_\mu^\star(x_k)
\preceq
(1-\eta)^{-2}\nabla^2\phi(z_k),
\qquad
H_\mu^\star(x_k)
\succeq
(1-\eta)^2\nabla^2\phi(z_k).
\]
Therefore
\[
(1-\eta)^2\rho_\psi^\eta \nabla^2\phi(z_k)
\preceq
\nabla_{yy}^2\psi_\mu(x_k,y)
\preceq
(1-\eta)^{-2}\ell_{\psi,1}^\eta \nabla^2\phi(z_k).
\]
Thus the exact lower objective is strongly convex and smooth in the frozen metric \(\nabla^2\phi(z_k)\), with
constants \(\bar\rho_\psi^\eta=(1-\eta)^2\rho_\psi^\eta\) and
\(\bar\ell_{\psi,1}^\eta=(1-\eta)^{-2}\ell_{\psi,1}^\eta\).

For the proxy tracker, Lemma~\ref{lem:proxy_objective_curvature} gives, on
\(T_{2\eta,\mu}^{\lambda_k}(x_k)\),
\[
\frac{\lambda_k\rho_\psi^\eta}{2}H_{\lambda_k,\mu}^\star(x_k)
\preceq
\nabla_{yy}^2L_{\lambda_k,\mu}(x_k,y)
\preceq
\bigl(\lambda_k\ell_{\psi,1}^\eta+\ell_{f,1}^\eta\bigr)
H_{\lambda_k,\mu}^\star(x_k).
\]
Using
\[
H_{\lambda_k,\mu}^\star(x_k)
\succeq
(1-\eta)^2\nabla^2\phi(y_k),
\qquad
H_{\lambda_k,\mu}^\star(x_k)
\preceq
(1-\eta)^{-2}\nabla^2\phi(y_k),
\]
we obtain
\[
\frac{\lambda_k(1-\eta)^2\rho_\psi^\eta}{2}\nabla^2\phi(y_k)
\preceq
\nabla_{yy}^2L_{\lambda_k,\mu}(x_k,y)
\preceq
(1-\eta)^{-2}
\bigl(\lambda_k\ell_{\psi,1}^\eta+\ell_{f,1}^\eta\bigr)\nabla^2\phi(y_k).
\]
Equivalently,
\[
\frac{\lambda_k\bar\rho_\psi^\eta}{2}\nabla^2\phi(y_k)
\preceq
\nabla_{yy}^2L_{\lambda_k,\mu}(x_k,y)
\preceq
\bigl(\lambda_k\bar\ell_{\psi,1}^\eta+\bar\ell_{f,1}^\eta\bigr)\nabla^2\phi(y_k),
\]
with
\(\bar\ell_{f,1}^\eta=(1-\eta)^{-2}\ell_{f,1}^\eta\).

Finally, the same metric comparison transfers the corresponding operator-norm bounds. Indeed, if
\(H\) denotes either center anchor and \(B\) the corresponding frozen metric, then
\[
(1-\eta)^2H\preceq B\preceq (1-\eta)^{-2}H
\]
implies
\[
\|u\|_B\le (1-\eta)^{-1}\|u\|_H,
\qquad
\|u\|_H\le (1-\eta)^{-1}\|u\|_B,
\]
and
\[
\|w\|_{B,*}\le (1-\eta)^{-1}\|w\|_{H,*},
\qquad
\|w\|_{H,*}\le (1-\eta)^{-1}\|w\|_{B,*}.
\]
Hence any bound of the form
\[
\|M\|_{2\to H,*}\le C
\]
transfers to
\[
\|M\|_{2\to B,*}\le (1-\eta)^{-1}C,
\]
and any bound of the form
\[
\|M\|_{H\to H,*}\le C
\]
transfers to
\[
\|M\|_{B\to B,*}\le (1-\eta)^{-2}C.
\]
Similarly, Lipschitz-Hessian bounds transfer after multiplying by a fixed power of
\((1-\eta)^{-1}\). Since \(\eta\in(0,1/2)\) is fixed, these factors are finite and can be absorbed into
the local Dikin constants. This proves the claimed frozen-metric regularity bounds.
\end{proof}

\begin{lemma}[Frozen-metric steps in anchored normal form]
\label{lem:frozen_metric_anchored_contraction}
Let \(c\in\operatorname{int}(Y)\), let \(H_c:=\nabla^2\phi(c)\), and let
\(\bar y\in\operatorname{int}(Y)\) satisfy
\(\|\bar y-c\|_{c}\le \eta\), where \(\eta\in(0,1/2)\). Define the frozen barrier-metric
\(B:=\nabla^2\phi(\bar y)\). Let \(h\) be twice differentiable on the Dikin buffer
\(T_{2\eta}(c):=\{y\in\operatorname{int}(Y):\|y-c\|_c\le 2\eta\}\), and suppose
\(\nabla h(c)=0\) and \( m H_c\preceq \nabla^2 h(y)\preceq L H_c,
y\in T_{2\eta}(c), 
\)
 for some \(0<m\le L<\infty\).
 
Let \(
\theta_\eta:=(1-\eta)^{-2}-1.\) Choose $\eta$ so that \(\theta_\eta L\le m/4\). Then, for every \(y\in T_{\eta}(c)\), the frozen-metric step
\(
y^+ := y-\alpha B^{-1}\nabla h(y)\)
satisfies
\[
\|y^+-c\|_c^2
\le
(1-\alpha m)\|y-c\|_c^2
\]
whenever \(
\alpha(1+\theta_\eta)^2L\le \frac12 .\)
In particular, after taking \(\eta\) sufficiently small and absorbing fixed \(\eta\)-dependent constants
into \(m\) and \(L\), the frozen-metric step has the same anchored contraction form as the ideal
center-metric step.
\end{lemma}

\begin{proof}
By the self-concordant metric-change inequality and the assumption
\(\|\bar y-c\|_c\le \eta\), we have
\[
(1-\eta)^2H_c\preceq B\preceq (1-\eta)^{-2}H_c.
\]
Define
\[
S:=H_c^{1/2}B^{-1}H_c^{1/2}.
\]
Then the eigenvalues of \(S\) lie in
\([(1-\eta)^2,(1-\eta)^{-2}]\). Hence
\[
\|S\|_2\le (1-\eta)^{-2}=1+\theta_\eta,
\qquad
\|S-I\|_2\le \theta_\eta.
\]

Let
\[
e:=y-c,\qquad p:=H_c^{1/2}e,\qquad r:=H_c^{-1/2}\nabla h(y).
\]
The frozen-metric update can be written in the \(H_c\)-coordinates as
\[
p^+ := H_c^{1/2}(y^+-c)
=
p-\alpha S r .
\]
Therefore
\[
\|p^+\|_2^2
=
\|p\|_2^2
-2\alpha\langle Sr,p\rangle
+\alpha^2\|Sr\|_2^2.
\]

Since \(\nabla h(c)=0\) and \(mH_c\preceq\nabla^2h\preceq LH_c\) on the buffer neighborhood, \(h\) is
\(m\)-strongly convex and \(L\)-smooth in the anchored norm. Thus
\[
\langle r,p\rangle
=
\langle \nabla h(y),y-c\rangle
\ge
m\|p\|_2^2,
\]
and
\[
\|r\|_2
=
\|\nabla h(y)\|_{c,*}
\le
L\|y-c\|_c
=
L\|p\|_2.
\]
Consequently,
\[
\begin{aligned}
\langle Sr,p\rangle
&=
\langle r,p\rangle+\langle (S-I)r,p\rangle  \\
&\ge
\langle r,p\rangle-\theta_\eta\|r\|_2\|p\|_2 \\
&\ge
\langle r,p\rangle-\theta_\eta L\|p\|_2^2 \\
&\ge
\left(1-\frac{\theta_\eta L}{m}\right)\langle r,p\rangle
\ge
\frac34\langle r,p\rangle,
\end{aligned}
\]
where the last step uses \(\theta_\eta L\le m/4\).

Moreover, the \(L\)-smoothness and convexity of \(h\) in the anchored metric imply the
Baillon--Haddad inequality
\[
\|r\|_2^2
=
\|\nabla h(y)\|_{c,*}^2
\le
L\langle \nabla h(y),y-c\rangle
=
L\langle r,p\rangle.
\]
Hence
\[
\|Sr\|_2^2
\le
(1+\theta_\eta)^2\|r\|_2^2
\le
(1+\theta_\eta)^2L\langle r,p\rangle.
\]
Using \(\alpha(1+\theta_\eta)^2L\le 1/2\), we obtain
\[
\alpha^2\|Sr\|_2^2
\le
\frac{\alpha}{2}\langle r,p\rangle.
\]
Combining the preceding estimates gives
\[
\begin{aligned}
\|p^+\|_2^2
&\le
\|p\|_2^2
-2\alpha\cdot \frac34\langle r,p\rangle
+\frac{\alpha}{2}\langle r,p\rangle  \\
&=
\|p\|_2^2-\alpha\langle r,p\rangle \\
&\le
(1-\alpha m)\|p\|_2^2.
\end{aligned}
\]
Since \(\|p\|_2=\|y-c\|_c\) and \(\|p^+\|_2=\|y^+-c\|_c\), this proves
\[
\|y^+-c\|_c^2\le (1-\alpha m)\|y-c\|_c^2.
\]
\end{proof}

Throughout the algorithmic analysis, we choose \(\eta\) small enough so that the frozen-metric
normal-form condition in Lemma~\ref{lem:frozen_metric_anchored_contraction} holds for both the
exact and proxy tracker curvature pairs; equivalently, the fixed \(\eta\)-dependent metric-comparison
factors are absorbed into the local Dikin constants used in Definition~\ref{def:barrier_aware}.

\paragraph{Convention for the remaining tracker analysis.}
Proposition~\ref{prop:frozen_metric_equivalence} and
Lemma~\ref{lem:frozen_metric_anchored_contraction} allow us to analyze the implementable frozen
barrier-metric updates as center-anchor Dikin contractions. Thus, in the fixed-anchor contraction
lemmas below, we write the estimates in the norms anchored at
\(y_\mu^\star(x_k)\) and \(y_{\lambda_k,\mu}^\star(x_k)\), even though Algorithm~\ref{alg:dikin_sbo}
uses the frozen metrics \(\nabla^2\phi(z_k)\) and \(\nabla^2\phi(y_k)\). All constants below are
understood after absorbing fixed \(\eta\)-dependent metric-comparison factors. We therefore relabel
the enlarged constants as
\(\rho_\psi^\eta,\ell_{\psi,1}^\eta,\ell_{f,1}^\eta\), etc., without changing the schedule scaling or the
convergence rate.

\subsection{Contraction for the \(y\)-tracker}
\label{app:y_fixed_anchor}

\begin{lemma}[Contraction of the proxy tracker]
\label{lem:y_fixed_anchor_contraction}
Fix an outer iteration \(k\). Suppose the proxy inner iterates remain in
\(T_{\eta,\mu}^{\lambda_k}(x_k)\), and suppose
\begin{equation}\label{eq:step_proxy_contraction}
\alpha_k\ell_{f,1}^{\eta}\le \frac18,
\qquad
\beta_k\ell_{\psi,1}^{\eta}\le \frac18,
\qquad
\beta_k:=\alpha_k\lambda_k .
\end{equation}
Then
\begin{equation}\label{eq:ytracker_contraction}
\|y_{k+1}-y_{\lambda_k,\mu}^\star(x_k)\|_{y_{\lambda_k,\mu}^\star(x_k)}^2
\le
\left(1-\frac{7\rho_\psi^\eta}{8}\beta_k\right)^T
\|y_k-y_{\lambda_k,\mu}^\star(x_k)\|_{y_{\lambda_k,\mu}^\star(x_k)}^2 .
\end{equation}
\end{lemma}

\begin{proof}
The implemented proxy-tracker update in Algorithm~\ref{alg:dikin_sbo} uses the frozen barrier-metric
\(\nabla^2\phi(y_k)\) during the \(T\) inner steps. By
Lemma~\ref{lem:frozen_metric_anchored_contraction}, on the maintained proxy neighborhood this frozen-metric
step admits the same anchored contraction normal form as a center-metric step, after absorbing fixed
\(\eta\)-dependent comparison factors into
\(\rho_\psi^\eta,\ell_{\psi,1}^\eta,\ell_{f,1}^\eta\). Thus it suffices to prove the contraction in the
anchored Dikin norm at \(y_{\lambda_k,\mu}^\star(x_k)\).

Apply Lemma~\ref{lem:frozen_metric_anchored_contraction} with
\[
h(y)=L_{\lambda_k,\mu}(x_k,y),
\qquad
c=y_{\lambda_k,\mu}^\star(x_k),
\qquad
\bar y=y_k,
\qquad
B=\nabla^2\phi(y_k).
\]
By Proposition~\ref{prop:derived_dikin_regularity} and
Lemma~\ref{lem:proxy_objective_curvature}, for every
\(y\in T_{\eta,\mu}^{\lambda_k}(x_k)\),
\begin{equation}
\label{eq:proxy_tracker_curvature}
\frac{\lambda_k\rho_\psi^\eta}{2}H_{\lambda_k,\mu}^\star(x_k)
\preceq
\nabla_{yy}^2L_{\lambda_k,\mu}(x_k,y)
\preceq
\bigl(\lambda_k\ell_{\psi,1}^\eta+\ell_{f,1}^\eta\bigr)
H_{\lambda_k,\mu}^\star(x_k).
\end{equation}
Therefore, \(y\mapsto L_{\lambda_k,\mu}(x_k,y)\) is
\(\lambda_k\rho_\psi^\eta/2\)-strongly convex and
\(\lambda_k\ell_{\psi,1}^\eta+\ell_{f,1}^\eta\)-smooth with respect to the anchored norm
\(\|\cdot\|_{y_{\lambda_k,\mu}^\star(x_k)}\) on the maintained proxy neighborhood.

Thus, for every inner iterate
\(y_k^{(t)}\in T_{\eta,\mu}^{\lambda_k}(x_k)\), strong convexity gives
\begin{align}
&\left\langle
\nabla_yL_{\lambda_k,\mu}(x_k,y_k^{(t)})
-
\nabla_yL_{\lambda_k,\mu}(x_k,y_{\lambda_k,\mu}^\star(x_k)),
\,y_k^{(t)}-y_{\lambda_k,\mu}^\star(x_k)
\right\rangle
\nonumber\\
&\qquad\ge
\frac{\lambda_k\rho_\psi^\eta}{2}
\|y_k^{(t)}-y_{\lambda_k,\mu}^\star(x_k)\|_{y_{\lambda_k,\mu}^\star(x_k)}^2 .
\label{eq:proxy_tracker_sc}
\end{align}
Since \(y_{\lambda_k,\mu}^\star(x_k)\) minimizes \(L_{\lambda_k,\mu}(x_k,\cdot)\),
\[
\nabla_yL_{\lambda_k,\mu}(x_k,y_{\lambda_k,\mu}^\star(x_k))=0,
\]
and hence \eqref{eq:proxy_tracker_sc} becomes
\begin{equation}
\label{eq:proxy_tracker_sc_simplified}
\left\langle
\nabla_yL_{\lambda_k,\mu}(x_k,y_k^{(t)}),
\,y_k^{(t)}-y_{\lambda_k,\mu}^\star(x_k)
\right\rangle
\ge
\frac{\lambda_k\rho_\psi^\eta}{2}
\|y_k^{(t)}-y_{\lambda_k,\mu}^\star(x_k)\|_{y_{\lambda_k,\mu}^\star(x_k)}^2 .
\end{equation}

Moreover, convexity and smoothness imply the Baillon--Haddad co-coercivity inequality in the
anchored primal/dual pair:
\begin{align}
&\|\nabla_yL_{\lambda_k,\mu}(x_k,y_k^{(t)})
-
\nabla_yL_{\lambda_k,\mu}(x_k,y_{\lambda_k,\mu}^\star(x_k))
\|_{y_{\lambda_k,\mu}^\star(x_k),*}^2
\nonumber\\
&\qquad\le
\bigl(\lambda_k\ell_{\psi,1}^\eta+\ell_{f,1}^\eta\bigr)
\left\langle
\nabla_yL_{\lambda_k,\mu}(x_k,y_k^{(t)})
-
\nabla_yL_{\lambda_k,\mu}(x_k,y_{\lambda_k,\mu}^\star(x_k)),
\,y_k^{(t)}-y_{\lambda_k,\mu}^\star(x_k)
\right\rangle .
\label{eq:proxy_tracker_cocoercive}
\end{align}
Using again
\(\nabla_yL_{\lambda_k,\mu}(x_k,y_{\lambda_k,\mu}^\star(x_k))=0\), this gives
\begin{equation}
\label{eq:proxy_tracker_grad_sq}
\|\nabla_yL_{\lambda_k,\mu}(x_k,y_k^{(t)})\|_{y_{\lambda_k,\mu}^\star(x_k),*}^2
\le
\bigl(\lambda_k\ell_{\psi,1}^\eta+\ell_{f,1}^\eta\bigr)
\left\langle
\nabla_yL_{\lambda_k,\mu}(x_k,y_k^{(t)}),
\,y_k^{(t)}-y_{\lambda_k,\mu}^\star(x_k)
\right\rangle .
\end{equation}

By the anchored normal form supplied by
Lemma~\ref{lem:frozen_metric_anchored_contraction}, the one-step proxy update satisfies the
center-anchor expansion
\begin{align*}
&\|y_k^{(t+1)}-y_{\lambda_k,\mu}^\star(x_k)\|_{y_{\lambda_k,\mu}^\star(x_k)}^2 \\
&\quad =
\left\|
y_k^{(t)}-y_{\lambda_k,\mu}^\star(x_k)
-
\alpha_k\bigl(H_{\lambda_k,\mu}^\star(x_k)\bigr)^{-1}
\nabla_yL_{\lambda_k,\mu}(x_k,y_k^{(t)})
\right\|_{y_{\lambda_k,\mu}^\star(x_k)}^2 \\
&\quad =
\|y_k^{(t)}-y_{\lambda_k,\mu}^\star(x_k)\|_{y_{\lambda_k,\mu}^\star(x_k)}^2 \\
&\qquad
-2\alpha_k
\left\langle
\nabla_yL_{\lambda_k,\mu}(x_k,y_k^{(t)}),
\,y_k^{(t)}-y_{\lambda_k,\mu}^\star(x_k)
\right\rangle \\
&\qquad
+\alpha_k^2
\|\nabla_yL_{\lambda_k,\mu}(x_k,y_k^{(t)})\|_{y_{\lambda_k,\mu}^\star(x_k),*}^2 .
\end{align*}
Here the expansion is written in the center-anchor norm for analysis; the implemented frozen-metric
step is reduced to this normal form by Lemma~\ref{lem:frozen_metric_anchored_contraction}.

Applying \eqref{eq:proxy_tracker_grad_sq} to the last term gives
\[
\alpha_k^2
\|\nabla_yL_{\lambda_k,\mu}(x_k,y_k^{(t)})\|_{y_{\lambda_k,\mu}^\star(x_k),*}^2
\le
\alpha_k^2
\bigl(\lambda_k\ell_{\psi,1}^\eta+\ell_{f,1}^\eta\bigr)
\left\langle
\nabla_yL_{\lambda_k,\mu}(x_k,y_k^{(t)}),
\,y_k^{(t)}-y_{\lambda_k,\mu}^\star(x_k)
\right\rangle .
\]
Therefore,
\begin{align*}
&\|y_k^{(t+1)}-y_{\lambda_k,\mu}^\star(x_k)\|_{y_{\lambda_k,\mu}^\star(x_k)}^2 \\
&\quad\le
\|y_k^{(t)}-y_{\lambda_k,\mu}^\star(x_k)\|_{y_{\lambda_k,\mu}^\star(x_k)}^2 \\
&\qquad
-
\Bigl(
2\alpha_k
-
\alpha_k^2(\lambda_k\ell_{\psi,1}^\eta+\ell_{f,1}^\eta)
\Bigr)
\left\langle
\nabla_yL_{\lambda_k,\mu}(x_k,y_k^{(t)}),
\,y_k^{(t)}-y_{\lambda_k,\mu}^\star(x_k)
\right\rangle .
\end{align*}

Under the step-size rule \eqref{eq:step_proxy_contraction},
\[
\alpha_k(\lambda_k\ell_{\psi,1}^{\eta}+\ell_{f,1}^{\eta})
=
\beta_k\ell_{\psi,1}^{\eta}
+
\alpha_k\ell_{f,1}^{\eta}
\le
\frac14.
\]
Thus
\[
2\alpha_k-\alpha_k^2(\lambda_k\ell_{\psi,1}^{\eta}+\ell_{f,1}^{\eta})
=
\alpha_k
\Bigl(
2-\alpha_k(\lambda_k\ell_{\psi,1}^{\eta}+\ell_{f,1}^{\eta})
\Bigr)
\ge
\frac{7}{4}\alpha_k .
\]
Combining this with \eqref{eq:proxy_tracker_sc_simplified}, we get
\begin{align*}
&\|y_k^{(t+1)}-y_{\lambda_k,\mu}^\star(x_k)\|_{y_{\lambda_k,\mu}^\star(x_k)}^2 \\
&\quad\le
\|y_k^{(t)}-y_{\lambda_k,\mu}^\star(x_k)\|_{y_{\lambda_k,\mu}^\star(x_k)}^2
-
\frac{7}{4}\alpha_k
\left\langle
\nabla_yL_{\lambda_k,\mu}(x_k,y_k^{(t)}),
\,y_k^{(t)}-y_{\lambda_k,\mu}^\star(x_k)
\right\rangle \\
&\quad\le
\left(1-\frac{7}{8}\alpha_k\lambda_k\rho_\psi^\eta\right)
\|y_k^{(t)}-y_{\lambda_k,\mu}^\star(x_k)\|_{y_{\lambda_k,\mu}^\star(x_k)}^2 \\
&\quad=
\left(1-\frac{7}{8}\rho_\psi^\eta\beta_k\right)
\|y_k^{(t)}-y_{\lambda_k,\mu}^\star(x_k)\|_{y_{\lambda_k,\mu}^\star(x_k)}^2 .
\end{align*}
Therefore,
\begin{equation}
\label{eq:proxy_tracker_one_step}
\|y_k^{(t+1)}-y_{\lambda_k,\mu}^\star(x_k)\|_{y_{\lambda_k,\mu}^\star(x_k)}^2
\le
\left(1-\frac{7\rho_\psi^\eta}{8}\beta_k\right)
\|y_k^{(t)}-y_{\lambda_k,\mu}^\star(x_k)\|_{y_{\lambda_k,\mu}^\star(x_k)}^2 .
\end{equation}
Repeating \eqref{eq:proxy_tracker_one_step} for \(t=0,\ldots,T-1\), and using
\(y_k^{(0)}=y_k\) and \(y_{k+1}=y_k^{(T)}\), gives
\[
\|y_{k+1}-y_{\lambda_k,\mu}^\star(x_k)\|_{y_{\lambda_k,\mu}^\star(x_k)}^2
\le
\left(1-\frac{7\rho_\psi^\eta}{8}\beta_k\right)^T
\|y_k-y_{\lambda_k,\mu}^\star(x_k)\|_{y_{\lambda_k,\mu}^\star(x_k)}^2 .
\]
This proves \eqref{eq:ytracker_contraction}.
\end{proof}

\begin{corollary}[Scheduled fixed-anchor contraction of the proxy tracker]
\label{cor:y_fixed_anchor_scheduled}
Under the step size requirement of Lemma~\ref{lem:y_fixed_anchor_contraction}, if in addition
\(T\rho_\psi^\eta\beta_k\le 1/4\), then
\[
\|y_{k+1}-y_{\lambda_k,\mu}^\star(x_k)\|_{y_{\lambda_k,\mu}^\star(x_k)}^2
\le
\left(1-\frac{3T\rho_\psi^\eta}{4}\beta_k\right)
\|y_k-y_{\lambda_k,\mu}^\star(x_k)\|_{y_{\lambda_k,\mu}^\star(x_k)}^2 .
\]
\end{corollary}

\begin{proof}
The proof of Lemma~\ref{lem:y_fixed_anchor_contraction} gives the one-step bound
\[
\|y_k^{(t+1)}-y_{\lambda_k,\mu}^\star(x_k)\|_{y_{\lambda_k,\mu}^\star(x_k)}^2
\le
\left(1-\frac{7\rho_\psi^\eta}{8}\beta_k\right)
\|y_k^{(t)}-y_{\lambda_k,\mu}^\star(x_k)\|_{y_{\lambda_k,\mu}^\star(x_k)}^2,
\]
because \(\alpha_k\ell_{f,1}^{\eta}+\beta_k\ell_{\psi,1}^{\eta}\le 1/4\). Repeating this \(T\) times
and using \(T\rho_\psi^\eta\beta_k\le 1/4\) yields
\[
\left(1-\frac{7\rho_\psi^\eta}{8}\beta_k\right)^T
\le
1-\frac{3T\rho_\psi^\eta}{4}\beta_k,
\]
which proves the claim.
\end{proof}

\subsection{Dikin Recursion for the \(y\)-tracker}

\begin{lemma}[Moving-anchor recursion for the proxy tracker]
\label{lem:y_moving_anchor_recursion}
Fix an outer iteration \(k\) and define
\(\delta_k:=\lambda_{k+1}-\lambda_k\ge 0\). Let
\[
I_{k+1}
:=
\big\|y_{k+1}-y_{\lambda_{k+1},\mu}^\star(x_{k+1})\big\|_{y_{\lambda_{k+1},\mu}^\star(x_{k+1})}^2 .
\]
Then, for any \(j_k>0\),
\begin{align}
\lambda_{k+1} I_{k+1}
\le\;&
\lambda_k
\Bigg[
\kappa_{y,k}^2
\frac{\lambda_{k+1}}{\lambda_k}
\Big(
1+\frac{2\delta_k}{\lambda_k}
+2\xi\alpha_k j_k
+(\ell_{\lambda,1}^\eta)^2\xi^2
\bigl(\ell_{f,0}^2\alpha_k^2+4\ell_{g,0}^2\beta_k^2\bigr)
\Big)
\Bigg]
\nonumber\\
&\qquad\qquad\cdot
\big\|y_{k+1}-y_{\lambda_k,\mu}^\star(x_k)\big\|_{y_{\lambda_k,\mu}^\star(x_k)}^2
\nonumber\\
&\quad
+\lambda_{k+1}\kappa_{y,k}^2
\Bigg[
2(\ell_{\lambda,0}^\eta)^2
+
\frac{(\ell_{\lambda,0}^\eta)^2}{2\xi\alpha_k j_k}
+
\frac12
\Bigg]
\xi^2\alpha_k^2\|q_k^x\|_2^2
\nonumber\\
&\quad
+\kappa_{y,k}^2
\Bigg[
\frac{8(\ell_{f,0}^\eta)^2}{(\rho_\psi^\eta)^2}
\frac{\delta_k}{\lambda_{k+1}}
+
2\frac{\kappa_{y,k}^2(\ell_{f,0}^\eta)^2}{(\rho_\psi^\eta)^2}
\frac{\lambda_{k+1}}{\lambda_k}
\Bigg]
\frac{\delta_k}{\lambda_k^2}.
\label{eq:proxy_moving_anchor_raw}
\end{align}
\end{lemma}

\begin{proof}
\proofstep{Anchor change.}
Define
\[
\kappa_{y,k}
:=
\left(
1-
\big\|y_{\lambda_{k+1},\mu}^\star(x_{k+1})
-y_{\lambda_k,\mu}^\star(x_k)\big\|_{y_{\lambda_k,\mu}^\star(x_k)}
\right)^{-1}.
\]
Whenever the center displacement is less than one in the old anchor norm, Lemma~\ref{lem:anchor_switch}
applied to the two proxy anchors gives
\[
\|u\|_{y_{\lambda_{k+1},\mu}^\star(x_{k+1})}^2
\le
\kappa_{y,k}^2
\|u\|_{y_{\lambda_k,\mu}^\star(x_k)}^2
\qquad \forall u.
\]
Therefore,
\begin{equation}
\label{eq:proxy_anchor_switch}
I_{k+1}
\le
\kappa_{y,k}^2
\big\|y_{k+1}-y_{\lambda_{k+1},\mu}^\star(x_{k+1})
\big\|_{y_{\lambda_k,\mu}^\star(x_k)}^2 .
\end{equation}
Expanding the squared anchored Dikin norm at the old proxy anchor gives
\begin{align}
&\big\|y_{k+1}-y_{\lambda_{k+1},\mu}^\star(x_{k+1})
\big\|_{y_{\lambda_k,\mu}^\star(x_k)}^2
\nonumber\\
&\quad =
\big\|y_{k+1}-y_{\lambda_k,\mu}^\star(x_k)
\big\|_{y_{\lambda_k,\mu}^\star(x_k)}^2
+
\big\|y_{\lambda_{k+1},\mu}^\star(x_{k+1})
-y_{\lambda_k,\mu}^\star(x_k)
\big\|_{y_{\lambda_k,\mu}^\star(x_k)}^2
\nonumber\\
&\qquad
-2
\Big\langle
y_{k+1}-y_{\lambda_k,\mu}^\star(x_k),\,
y_{\lambda_{k+1},\mu}^\star(x_{k+1})
-y_{\lambda_k,\mu}^\star(x_k)
\Big\rangle_{y_{\lambda_k,\mu}^\star(x_k)} .
\label{eq:proxy_three_term_expansion}
\end{align}

\proofstep{Center drift.}
Apply Proposition~\ref{prop:local_consequences}\textup{(i)} with
\((\lambda_1,x_1)=(\lambda_k,x_k)\) and
\((\lambda_2,x_2)=(\lambda_{k+1},x_{k+1})\). This gives
\begin{equation}
\label{eq:proxy_center_drift}
\big\|y_{\lambda_{k+1},\mu}^\star(x_{k+1})
-y_{\lambda_k,\mu}^\star(x_k)
\big\|_{y_{\lambda_k,\mu}^\star(x_k)}
\le
\frac{2\delta_k\ell_{f,0}^\eta}{\lambda_k\lambda_{k+1}\rho_\psi^\eta}
+
\ell_{\lambda,0}^\eta\|x_{k+1}-x_k\|_2 .
\end{equation}
Squaring and using \((a+b)^2\le 2a^2+2b^2\) yields
\begin{align}
\label{eq:proxy_center_drift_square}
\big\|y_{\lambda_{k+1},\mu}^\star(x_{k+1})
-y_{\lambda_k,\mu}^\star(x_k)
\big\|_{y_{\lambda_k,\mu}^\star(x_k)}^2
\le\;&
\frac{8(\ell_{f,0}^\eta)^2}{(\rho_\psi^\eta)^2}
\frac{\delta_k^2}{\lambda_k^2\lambda_{k+1}^2}
+
2(\ell_{\lambda,0}^\eta)^2\|x_{k+1}-x_k\|_2^2 .
\end{align}
Since \(x_{k+1}-x_k=-\xi\alpha_k q_k^x\), we have
\[
\|x_{k+1}-x_k\|_2^2
=
\xi^2\alpha_k^2\|q_k^x\|_2^2 .
\]

\proofstep{Cross term.}
Consider the cross term in \eqref{eq:proxy_three_term_expansion}. Because
Lemma~\ref{lem:proxy_center_cross_bound} is stated in the anchored Dikin geometry at
\(y_{\lambda_k,\mu}^\star(x_k)\), and its right-hand side depends on the test vector only through
anchored Dikin norms, the same bound applies after replacing the test vector by its negative. Hence,
for any \(j_k>0\), applying Lemma~\ref{lem:proxy_center_cross_bound} with
\[
v_k:=y_{k+1}-y_{\lambda_k,\mu}^\star(x_k)
\]
gives
\begin{align}
&-2
\Big\langle
y_{k+1}-y_{\lambda_k,\mu}^\star(x_k),\,
y_{\lambda_{k+1},\mu}^\star(x_{k+1})
-y_{\lambda_k,\mu}^\star(x_k)
\Big\rangle_{y_{\lambda_k,\mu}^\star(x_k)}
\nonumber\\
&\quad\le
\Bigg(
\frac{2\delta_k}{\lambda_k}
+
2\xi\alpha_k j_k
+
(\ell_{\lambda,1}^\eta)^2\xi^2
\bigl(\ell_{f,0}^2\alpha_k^2+4\ell_{g,0}^2\beta_k^2\bigr)
\Bigg)
\big\|y_{k+1}-y_{\lambda_k,\mu}^\star(x_k)
\big\|_{y_{\lambda_k,\mu}^\star(x_k)}^2
\nonumber\\
&\qquad
+
\left(
\frac{\xi\alpha_k}{2j_k}(\ell_{\lambda,0}^\eta)^2
+
\frac{\xi^2}{2}\alpha_k^2
\right)
\|q_k^x\|_2^2
+
2\frac{\kappa_{y,k}^2(\ell_{f,0}^\eta)^2}{(\rho_\psi^\eta)^2}
\frac{\delta_k}{\lambda_k^3}.
\label{eq:proxy_cross_term_bound}
\end{align}

\proofstep{Collecting terms.}
Combining \eqref{eq:proxy_anchor_switch}, \eqref{eq:proxy_three_term_expansion},
\eqref{eq:proxy_center_drift_square}, and \eqref{eq:proxy_cross_term_bound}, we obtain
\begin{align}
I_{k+1}
\le\;&
\kappa_{y,k}^2
\Bigg(
1+\frac{2\delta_k}{\lambda_k}
+2\xi\alpha_k j_k
+(\ell_{\lambda,1}^\eta)^2\xi^2
\bigl(\ell_{f,0}^2\alpha_k^2+4\ell_{g,0}^2\beta_k^2\bigr)
\Bigg)
\nonumber\\
&\qquad\qquad\cdot
\big\|y_{k+1}-y_{\lambda_k,\mu}^\star(x_k)
\big\|_{y_{\lambda_k,\mu}^\star(x_k)}^2
\nonumber\\
&\quad
+\kappa_{y,k}^2
\Bigg[
2(\ell_{\lambda,0}^\eta)^2\xi^2\alpha_k^2
+
\frac{(\ell_{\lambda,0}^\eta)^2}{2}\frac{\xi\alpha_k}{j_k}
+
\frac{\xi^2}{2}\alpha_k^2
\Bigg]\|q_k^x\|_2^2
\nonumber\\
&\quad
+\kappa_{y,k}^2
\Bigg[
\frac{8(\ell_{f,0}^\eta)^2}{(\rho_\psi^\eta)^2}
\frac{\delta_k^2}{\lambda_k^2\lambda_{k+1}^2}
+
2\frac{\kappa_{y,k}^2(\ell_{f,0}^\eta)^2}{(\rho_\psi^\eta)^2}
\frac{\delta_k}{\lambda_k^3}
\Bigg].
\label{eq:proxy_moving_anchor_preweight}
\end{align}
Multiplying \eqref{eq:proxy_moving_anchor_preweight} by \(\lambda_{k+1}\) and rearranging the
last line gives \eqref{eq:proxy_moving_anchor_raw}.
\end{proof}

\begin{corollary}[Scheduled moving-anchor recursion for the proxy tracker]
\label{cor:y_moving_anchor_scheduled}
Adopting the convention from
Appendix~\ref{app:supporting_lemmas} that \(\ell_{*,1}^\eta\) is enlarged so that
\(\ell_{*,1}^\eta\ge\max\{1,\ell_{\lambda,1}^\eta\}\), if
\begin{equation}
\label{eq:proxy_recursion_schedule}
\frac{\delta_k}{\lambda_k}\le \frac{T\rho_\psi^\eta}{16}\beta_k,
\qquad
2\xi^2\max\{\ell_{g,0}^2,\ell_{f,0}^2\}(\ell_{*,1}^{\eta})^2\beta_k^2
\le
\frac{T\rho_\psi^\eta}{16}\beta_k,
\end{equation}
then
\begin{equation}
\label{eq:proxy_recursion_scheduled}
I_{k+1}
\le
\kappa_{y,k}^2
\left(1+\frac{3T\rho_\psi^\eta}{8}\beta_k+\frac{\delta_k}{\lambda_k}\right)
\big\|y_{k+1}-y_{\lambda_k,\mu}^\star(x_k)
\big\|_{y_{\lambda_k,\mu}^\star(x_k)}^2
+
O\!\left(
\frac{\xi^2(\ell_{\lambda,0}^{\eta})^2\alpha_k^2}{\rho_\psi^\eta T\beta_k}
\right)\|q_k^x\|_2^2
+
O\!\left(
\frac{(\ell_{f,0}^{\eta})^2}{(\rho_\psi^\eta)^2}
\frac{\delta_k}{\lambda_k^3}
\right).
\end{equation}
\end{corollary}

\begin{proof}
Choose
\[
j_k=\frac{T\rho_\psi^\eta\lambda_k}{16\xi}.
\]
Then
\[
2\xi\alpha_k j_k
=
\frac{T\rho_\psi^\eta}{8}\beta_k .
\]
Substituting this choice into \eqref{eq:proxy_moving_anchor_preweight} gives
\begin{align}
I_{k+1}
\le\;&
\kappa_{y,k}^2
\Bigg(
1+\frac{2\delta_k}{\lambda_k}
+\frac{T\rho_\psi^\eta}{8}\beta_k
+(\ell_{\lambda,1}^\eta)^2\xi^2
\bigl(\ell_{f,0}^2\alpha_k^2+4\ell_{g,0}^2\beta_k^2\bigr)
\Bigg)
\nonumber\\
&\qquad\qquad\cdot
\big\|y_{k+1}-y_{\lambda_k,\mu}^\star(x_k)
\big\|_{y_{\lambda_k,\mu}^\star(x_k)}^2
\nonumber\\
&\quad
+\kappa_{y,k}^2
\Bigg[
2(\ell_{\lambda,0}^\eta)^2\xi^2\alpha_k^2
+
\frac{8(\ell_{\lambda,0}^\eta)^2\xi^2\alpha_k^2}{\rho_\psi^\eta T\beta_k}
+
\frac{\xi^2}{2}\alpha_k^2
\Bigg]\|q_k^x\|_2^2
\nonumber\\
&\quad
+\kappa_{y,k}^2
\Bigg[
\frac{8(\ell_{f,0}^\eta)^2}{(\rho_\psi^\eta)^2}
\frac{\delta_k^2}{\lambda_k^2\lambda_{k+1}^2}
+
2\frac{\kappa_{y,k}^2(\ell_{f,0}^\eta)^2}{(\rho_\psi^\eta)^2}
\frac{\delta_k}{\lambda_k^3}
\Bigg].
\label{eq:proxy_recursion_after_j}
\end{align}
We now bound the parenthesis in the first line of \eqref{eq:proxy_recursion_after_j}.
Split \(2\delta_k/\lambda_k=\delta_k/\lambda_k+\delta_k/\lambda_k\) and bound the second
copy by the first inequality in \eqref{eq:proxy_recursion_schedule}, giving
\(\delta_k/\lambda_k\le T\rho_\psi^\eta\beta_k/16\). The contraction-budget term is
\(T\rho_\psi^\eta\beta_k/8\). For the anchor-switch term, the convention
\(\ell_{\lambda,1}^\eta\le\ell_{*,1}^\eta\) and \(\alpha_k\le\beta_k\) (using
\(\lambda_k\ge 1\)) give
\[
(\ell_{\lambda,1}^\eta)^2\xi^2\bigl(\ell_{f,0}^2\alpha_k^2+4\ell_{g,0}^2\beta_k^2\bigr)
\le
5\xi^2(\ell_{*,1}^\eta)^2\max\{\ell_{g,0}^2,\ell_{f,0}^2\}\beta_k^2
\le
\frac{5T\rho_\psi^\eta}{32}\beta_k,
\]
where the last step uses the second inequality in \eqref{eq:proxy_recursion_schedule}.
Summing the four pieces yields
\(\delta_k/\lambda_k + T\rho_\psi^\eta\beta_k/16 + T\rho_\psi^\eta\beta_k/8 +
5T\rho_\psi^\eta\beta_k/32 = \delta_k/\lambda_k + 11T\rho_\psi^\eta\beta_k/32
\le \delta_k/\lambda_k + 3T\rho_\psi^\eta\beta_k/8\), i.e.
\[
\frac{\delta_k}{\lambda_k}+\frac{3T\rho_\psi^\eta}{8}\beta_k .
\]
Thus the first line of \eqref{eq:proxy_recursion_after_j} is bounded by
\[
\kappa_{y,k}^2
\left(1+\frac{3T\rho_\psi^\eta}{8}\beta_k+\frac{\delta_k}{\lambda_k}\right)
\big\|y_{k+1}-y_{\lambda_k,\mu}^\star(x_k)
\big\|_{y_{\lambda_k,\mu}^\star(x_k)}^2 .
\]
The second line is absorbed into
\[
O\!\left(
\frac{\xi^2(\ell_{\lambda,0}^{\eta})^2\alpha_k^2}{\rho_\psi^\eta T\beta_k}
\right)\|q_k^x\|_2^2,
\]
after increasing the absolute constant if necessary. The last line is bounded by
\[
O\!\left(
\frac{(\ell_{f,0}^{\eta})^2}{(\rho_\psi^\eta)^2}
\frac{\delta_k}{\lambda_k^3}
\right),
\]
using \(\lambda_{k+1}\ge \lambda_k\) and the schedule condition
\(\delta_k/\lambda_k=O(\beta_k)\). Combining these bounds proves
\eqref{eq:proxy_recursion_scheduled}.
\end{proof}

\subsection{Contraction for the \(z\)-tracker}

We next analyze the inner tracker \(z_k\), whose role is to follow the moving exact lower
minimizer \(y_{\mu}^*(x_k)\). As on the proxy side, the analysis separates into a fixed-anchor
contraction and a moving-anchor recursion.

\begin{lemma}[Contraction of the exact tracker]
\label{lem:z_fixed_anchor_contraction}
Fix an outer iteration \(k\). Suppose the exact inner iterates remain in
\(T_{\eta,\mu}(x_k)\), and suppose
\begin{equation}
\label{eq:step_exact_contraction}
\ell_{\psi,1}^\eta\,\gamma_k \le \frac{1}{4},
\qquad
T\,\rho_\psi^\eta\,\gamma_k \le \frac{1}{4}.
\end{equation}
Then
\begin{equation}
\label{eq:ztracker_contraction}
\|z_{k+1}-y_\mu^\star(x_k)\|_{y_\mu^\star(x_k)}^2
\le
\bigl(1-\gamma_k\rho_\psi^\eta\bigr)^T
\|z_k-y_\mu^\star(x_k)\|_{y_\mu^\star(x_k)}^2 .
\end{equation}
\end{lemma}

\begin{proof}
The implemented exact-tracker update in Algorithm~\ref{alg:dikin_sbo} uses the frozen barrier
metric \(\nabla^2\phi(z_k)\) during the \(T\) inner steps. By
Lemma~\ref{lem:frozen_metric_anchored_contraction}, on the maintained exact neighborhood this
frozen-metric step admits the same anchored contraction normal form as a center-metric step, after
absorbing fixed \(\eta\)-dependent comparison factors into
\(\rho_\psi^\eta\) and \(\ell_{\psi,1}^\eta\). Thus it suffices to carry out the contraction argument in
the anchored Dikin norm at \(y_\mu^\star(x_k)\).

Apply Lemma~\ref{lem:frozen_metric_anchored_contraction} with
\[
h(y)=\psi_\mu(x_k,y),
\qquad
c=y_\mu^\star(x_k),
\qquad
\bar y=z_k,
\qquad
B=\nabla^2\phi(z_k).
\]
By Proposition~\ref{prop:derived_dikin_regularity}, for every
\(y\in T_{\eta,\mu}(x_k)\),
\begin{equation}
\label{eq:exact_tracker_curvature}
\rho_\psi^\eta H_\mu^\star(x_k)
\preceq
\nabla_{yy}^2\psi_\mu(x_k,y)
\preceq
\ell_{\psi,1}^\eta H_\mu^\star(x_k).
\end{equation}
Therefore, \(y\mapsto \psi_\mu(x_k,y)\) is \(\rho_\psi^\eta\)-strongly convex and
\(\ell_{\psi,1}^\eta\)-smooth with respect to the anchored norm
\(\|\cdot\|_{y_\mu^\star(x_k)}\) on the maintained exact neighborhood.

Thus, for every inner iterate \(z_k^{(t)}\in T_{\eta,\mu}(x_k)\), strong convexity gives
\begin{align}
&\left\langle
\nabla_y\psi_\mu(x_k,z_k^{(t)})
-
\nabla_y\psi_\mu(x_k,y_\mu^\star(x_k)),
\,z_k^{(t)}-y_\mu^\star(x_k)
\right\rangle
\nonumber\\
&\qquad\ge
\rho_\psi^\eta
\|z_k^{(t)}-y_\mu^\star(x_k)\|_{y_\mu^\star(x_k)}^2 .
\label{eq:exact_tracker_sc}
\end{align}
Since \(y_\mu^\star(x_k)\) minimizes \(\psi_\mu(x_k,\cdot)\),
\[
\nabla_y\psi_\mu(x_k,y_\mu^\star(x_k))=0,
\]
and hence \eqref{eq:exact_tracker_sc} becomes
\begin{equation}
\label{eq:exact_tracker_sc_simplified}
\left\langle
\nabla_y\psi_\mu(x_k,z_k^{(t)}),
\,z_k^{(t)}-y_\mu^\star(x_k)
\right\rangle
\ge
\rho_\psi^\eta
\|z_k^{(t)}-y_\mu^\star(x_k)\|_{y_\mu^\star(x_k)}^2 .
\end{equation}

Moreover, convexity and smoothness imply the Baillon--Haddad co-coercivity inequality in the
anchored primal/dual pair:
\begin{align}
&\|\nabla_y\psi_\mu(x_k,z_k^{(t)})
-
\nabla_y\psi_\mu(x_k,y_\mu^\star(x_k))
\|_{y_\mu^\star(x_k),*}^2
\nonumber\\
&\qquad\le
\ell_{\psi,1}^\eta
\left\langle
\nabla_y\psi_\mu(x_k,z_k^{(t)})
-
\nabla_y\psi_\mu(x_k,y_\mu^\star(x_k)),
\,z_k^{(t)}-y_\mu^\star(x_k)
\right\rangle .
\label{eq:exact_tracker_cocoercive}
\end{align}
Using again \(\nabla_y\psi_\mu(x_k,y_\mu^\star(x_k))=0\), this gives
\begin{equation}
\label{eq:exact_tracker_grad_sq}
\|\nabla_y\psi_\mu(x_k,z_k^{(t)})\|_{y_\mu^\star(x_k),*}^2
\le
\ell_{\psi,1}^\eta
\left\langle
\nabla_y\psi_\mu(x_k,z_k^{(t)}),
\,z_k^{(t)}-y_\mu^\star(x_k)
\right\rangle .
\end{equation}

By the anchored normal form supplied by
Lemma~\ref{lem:frozen_metric_anchored_contraction}, the one-step exact-tracker update satisfies the
center-anchor expansion
\begin{align*}
&\|z_k^{(t+1)}-y_\mu^\star(x_k)\|_{y_\mu^\star(x_k)}^2 \\
&\quad =
\left\|
z_k^{(t)}-y_\mu^\star(x_k)
-
\gamma_k\bigl(H_\mu^\star(x_k)\bigr)^{-1}
\nabla_y\psi_\mu(x_k,z_k^{(t)})
\right\|_{y_\mu^\star(x_k)}^2 \\
&\quad =
\|z_k^{(t)}-y_\mu^\star(x_k)\|_{y_\mu^\star(x_k)}^2 \\
&\qquad
-2\gamma_k
\left\langle
\nabla_y\psi_\mu(x_k,z_k^{(t)}),
\,z_k^{(t)}-y_\mu^\star(x_k)
\right\rangle \\
&\qquad
+\gamma_k^2
\|\nabla_y\psi_\mu(x_k,z_k^{(t)})\|_{y_\mu^\star(x_k),*}^2 .
\end{align*}
Here the expansion is written in the exact center-anchor norm for analysis; the implemented frozen
barrier-metric step is reduced to this normal form by
Lemma~\ref{lem:frozen_metric_anchored_contraction}.

Applying \eqref{eq:exact_tracker_grad_sq} to the last term yields
\[
\gamma_k^2
\|\nabla_y\psi_\mu(x_k,z_k^{(t)})\|_{y_\mu^\star(x_k),*}^2
\le
\gamma_k^2\ell_{\psi,1}^\eta
\left\langle
\nabla_y\psi_\mu(x_k,z_k^{(t)}),
\,z_k^{(t)}-y_\mu^\star(x_k)
\right\rangle .
\]
Therefore,
\begin{align*}
&\|z_k^{(t+1)}-y_\mu^\star(x_k)\|_{y_\mu^\star(x_k)}^2 \\
&\quad\le
\|z_k^{(t)}-y_\mu^\star(x_k)\|_{y_\mu^\star(x_k)}^2
-
\bigl(2\gamma_k-\gamma_k^2\ell_{\psi,1}^\eta\bigr)
\left\langle
\nabla_y\psi_\mu(x_k,z_k^{(t)}),
\,z_k^{(t)}-y_\mu^\star(x_k)
\right\rangle .
\end{align*}
Under the step-size condition \eqref{eq:step_exact_contraction},
\[
2\gamma_k-\gamma_k^2\ell_{\psi,1}^\eta
=
\gamma_k(2-\gamma_k\ell_{\psi,1}^\eta)
\ge
\gamma_k .
\]
Combining this with \eqref{eq:exact_tracker_sc_simplified}, we obtain the one-step contraction
\begin{align}
\|z_k^{(t+1)}-y_\mu^\star(x_k)\|_{y_\mu^\star(x_k)}^2
&\le
\|z_k^{(t)}-y_\mu^\star(x_k)\|_{y_\mu^\star(x_k)}^2
-
\gamma_k\rho_\psi^\eta
\|z_k^{(t)}-y_\mu^\star(x_k)\|_{y_\mu^\star(x_k)}^2
\nonumber\\
&=
\bigl(1-\gamma_k\rho_\psi^\eta\bigr)
\|z_k^{(t)}-y_\mu^\star(x_k)\|_{y_\mu^\star(x_k)}^2 .
\label{eq:exact_tracker_one_step}
\end{align}
Repeating \eqref{eq:exact_tracker_one_step} for \(t=0,\ldots,T-1\), and using
\(z_k^{(0)}=z_k\) and \(z_{k+1}=z_k^{(T)}\), gives
\[
\|z_{k+1}-y_\mu^\star(x_k)\|_{y_\mu^\star(x_k)}^2
\le
\bigl(1-\gamma_k\rho_\psi^\eta\bigr)^T
\|z_k-y_\mu^\star(x_k)\|_{y_\mu^\star(x_k)}^2 .
\]
This proves \eqref{eq:ztracker_contraction}.
\end{proof}

\begin{corollary}[Scheduled fixed-anchor contraction of the \(z\)-tracker]
\label{cor:z_fixed_anchor_scheduled}
Under the assumptions of Lemma~\ref{lem:z_fixed_anchor_contraction}, if in addition
\(T\rho_\psi^\eta\gamma_k\le 1/4\), then
\[
\|z_{k+1}-y_\mu^\star(x_k)\|_{y_\mu^\star(x_k)}^2
\le
\left(1-\frac{3T\rho_\psi^\eta}{4}\gamma_k\right)
\|z_k-y_\mu^\star(x_k)\|_{y_\mu^\star(x_k)}^2 .
\]
\end{corollary}

\begin{proof}
Lemma~\ref{lem:z_fixed_anchor_contraction} gives
\[
\|z_{k+1}-y_\mu^\star(x_k)\|_{y_\mu^\star(x_k)}^2
\le
(1-\rho_\psi^\eta\gamma_k)^T
\|z_k-y_\mu^\star(x_k)\|_{y_\mu^\star(x_k)}^2 .
\]
Since \(T\rho_\psi^\eta\gamma_k\le 1/4\), the elementary bound
\((1-a)^T\le 1-\frac34Ta\) for \(Ta\le 1/4\), with \(a=\rho_\psi^\eta\gamma_k\), yields the claim.
\end{proof}

\subsection{Dikin Recursion for the \(z\)-tracker}

\begin{lemma}[Moving-anchor recursion for the exact tracker]
\label{lem:z_moving_anchor_recursion}
Fix an outer iteration \(k\). Define
\[
J_{k+1}
:=
\big\|z_{k+1}-y_\mu^\star(x_{k+1})\big\|_{y_\mu^\star(x_{k+1})}^2 .
\]
Let
\[
\kappa_{z,k}
:=
\left(
1-
\big\|y_\mu^\star(x_{k+1})-y_\mu^\star(x_k)\big\|_{y_\mu^\star(x_k)}
\right)^{-1}.
\]
Then, for any \(j_k>0\),
\begin{align}
J_{k+1}
\le\;&
\kappa_{z,k}^2
\Bigl(
1
+
2\xi\alpha_k j_k
+
(\ell_{*,1}^\eta)^2\xi^2
\bigl(\ell_{f,0}^2\alpha_k^2+4\ell_{g,0}^2\beta_k^2\bigr)
\Bigr)
\big\|z_{k+1}-y_\mu^\star(x_k)\big\|_{y_\mu^\star(x_k)}^2
\nonumber\\
&\quad
+
\kappa_{z,k}^2
\Bigl(
(\ell_{*,0}^\eta)^2\xi^2\alpha_k^2
+
\frac{\xi\alpha_k}{2j_k}(\ell_{*,0}^\eta)^2
+
\frac{\xi^2}{2}\alpha_k^2
\Bigr)
\|q_k^x\|_2^2 .
\label{eq:exact_moving_anchor_raw}
\end{align}
\end{lemma}

\begin{proof}
\proofstep{Anchor change.}
Whenever
\[
\big\|y_\mu^\star(x_{k+1})-y_\mu^\star(x_k)\big\|_{y_\mu^\star(x_k)}<1,
\]
Lemma~\ref{lem:anchor_switch}, applied to the two exact anchors, gives
\[
\|u\|_{y_\mu^\star(x_{k+1})}^2
\le
\kappa_{z,k}^2
\|u\|_{y_\mu^\star(x_k)}^2
\qquad \forall u.
\]
Hence
\begin{equation}
\label{eq:exact_anchor_switch}
J_{k+1}
=
\|z_{k+1}-y_\mu^\star(x_{k+1})\|_{y_\mu^\star(x_{k+1})}^2
\le
\kappa_{z,k}^2
\|z_{k+1}-y_\mu^\star(x_{k+1})\|_{y_\mu^\star(x_k)}^2 .
\end{equation}
Expanding the squared norm at the old anchor \(y_\mu^\star(x_k)\), we obtain
\begin{align}
&\|z_{k+1}-y_\mu^\star(x_{k+1})\|_{y_\mu^\star(x_k)}^2
\nonumber\\
&\quad =
\|z_{k+1}-y_\mu^\star(x_k)\|_{y_\mu^\star(x_k)}^2
+
\|y_\mu^\star(x_{k+1})-y_\mu^\star(x_k)\|_{y_\mu^\star(x_k)}^2
\nonumber\\
&\qquad
-
2
\Big\langle
z_{k+1}-y_\mu^\star(x_k),\,
y_\mu^\star(x_{k+1})-y_\mu^\star(x_k)
\Big\rangle_{y_\mu^\star(x_k)} .
\label{eq:exact_three_term_expansion}
\end{align}

\paragraph{Center drift.}
By Proposition~\ref{prop:local_consequences}\textup{(ii)},
\[
\|y_\mu^\star(x_{k+1})-y_\mu^\star(x_k)\|_{y_\mu^\star(x_k)}
\le
\ell_{*,0}^\eta \|x_{k+1}-x_k\|_2 .
\]
Using the outer update \(x_{k+1}-x_k=-\xi\alpha_k q_k^x\), this gives
\begin{equation}
\label{eq:exact_center_drift_square}
\|y_\mu^\star(x_{k+1})-y_\mu^\star(x_k)\|_{y_\mu^\star(x_k)}^2
\le
(\ell_{*,0}^\eta)^2\xi^2\alpha_k^2\|q_k^x\|_2^2 .
\end{equation}

\proofstep{Cross term.}
Consider the cross term in \eqref{eq:exact_three_term_expansion}. Lemma~\ref{lem:exact_center_cross_bound} is
stated in the anchored Dikin geometry at \(y_\mu^\star(x_k)\), and its right-hand side depends on
the test vector only through anchored Dikin norms. Therefore the same bound applies after replacing
the test vector by its negative. Taking
\[
v_k:=z_{k+1}-y_\mu^\star(x_k),
\]
we obtain, for any \(j_k>0\),
\begin{align}
&-2
\Big\langle
z_{k+1}-y_\mu^\star(x_k),\,
y_\mu^\star(x_{k+1})-y_\mu^\star(x_k)
\Big\rangle_{y_\mu^\star(x_k)}
\nonumber\\
&\quad\le
\Bigl(
2\xi\alpha_k j_k
+
(\ell_{*,1}^\eta)^2\xi^2
\bigl(\ell_{f,0}^2\alpha_k^2+4\ell_{g,0}^2\beta_k^2\bigr)
\Bigr)
\big\|z_{k+1}-y_\mu^\star(x_k)\big\|_{y_\mu^\star(x_k)}^2
\nonumber\\
&\qquad
+
\left(
\frac{\xi\alpha_k}{2j_k}(\ell_{*,0}^\eta)^2
+
\frac{\xi^2}{2}\alpha_k^2
\right)
\|q_k^x\|_2^2 .
\label{eq:exact_cross_term_bound}
\end{align}

\proofstep{Collecting terms.}
Combining \eqref{eq:exact_anchor_switch}, \eqref{eq:exact_three_term_expansion},
\eqref{eq:exact_center_drift_square}, and \eqref{eq:exact_cross_term_bound}, we get
\begin{align*}
J_{k+1}
\le\;&
\kappa_{z,k}^2
\Bigl(
1
+
2\xi\alpha_k j_k
+
(\ell_{*,1}^\eta)^2\xi^2
\bigl(\ell_{f,0}^2\alpha_k^2+4\ell_{g,0}^2\beta_k^2\bigr)
\Bigr)
\big\|z_{k+1}-y_\mu^\star(x_k)\big\|_{y_\mu^\star(x_k)}^2
\\
&\quad
+
\kappa_{z,k}^2
\Bigl(
(\ell_{*,0}^\eta)^2\xi^2\alpha_k^2
+
\frac{\xi\alpha_k}{2j_k}(\ell_{*,0}^\eta)^2
+
\frac{\xi^2}{2}\alpha_k^2
\Bigr)
\|q_k^x\|_2^2 .
\end{align*}
This proves \eqref{eq:exact_moving_anchor_raw}.
\end{proof}

\begin{corollary}[Scheduled moving-anchor recursion for the exact tracker]
\label{cor:z_moving_anchor_scheduled}
If
\begin{equation}
\label{eq:exact_recursion_schedule}
(\ell_{*,1}^\eta)^2\xi^2
\bigl(\ell_{f,0}^2\alpha_k^2+4\ell_{g,0}^2\beta_k^2\bigr)
\le
\frac{T\rho_\psi^\eta}{16}\gamma_k,
\end{equation}
then
\begin{equation}
\label{eq:exact_recursion_scheduled}
J_{k+1}
\le
\kappa_{z,k}^2
\left(1+\frac{3T\rho_\psi^\eta}{8}\gamma_k\right)
\big\|z_{k+1}-y_\mu^\star(x_k)\big\|_{y_\mu^\star(x_k)}^2
+
O\!\left(
\frac{\xi^2(\ell_{*,0}^\eta)^2\alpha_k^2}
{T\rho_\psi^\eta\gamma_k}
\right)
\|q_k^x\|_2^2 .
\end{equation}
\end{corollary}

\begin{proof}
Choose
\[
j_k:=\frac{T\rho_\psi^\eta\gamma_k}{16\xi\alpha_k}.
\]
Then
\[
2\xi\alpha_k j_k
=
\frac{T\rho_\psi^\eta}{8}\gamma_k .
\]
Substituting this choice into \eqref{eq:exact_moving_anchor_raw}, and using
\eqref{eq:exact_recursion_schedule}, gives
\[
1
+
2\xi\alpha_k j_k
+
(\ell_{*,1}^\eta)^2\xi^2
\bigl(\ell_{f,0}^2\alpha_k^2+4\ell_{g,0}^2\beta_k^2\bigr)
\le
1+\frac{3T\rho_\psi^\eta}{16}\gamma_k
\le
1+\frac{3T\rho_\psi^\eta}{8}\gamma_k .
\]
Moreover,
\[
\frac{\xi\alpha_k}{2j_k}(\ell_{*,0}^\eta)^2
=
\frac{8\xi^2(\ell_{*,0}^\eta)^2\alpha_k^2}
{T\rho_\psi^\eta\gamma_k}.
\]
The remaining terms in the coefficient of \(\|q_k^x\|_2^2\) are absorbed into the same order after
increasing the absolute constant. Therefore \eqref{eq:exact_recursion_scheduled} follows.
\end{proof}

\section{Tube Invariance}
\label{app:tube_invariance}

All local estimates above are valid on the enlarged buffer neighborhood. We now show that the
\textit{barrier-aware} schedule keeps the algorithmic iterates inside the smaller target neighborhoods. Define the
\(T\)-step contraction gaps
\begin{equation}
\label{eq:tube_contraction_gaps}
\Gamma_{\psi,k}
:=
1-\bigl(1-\rho_\psi^\eta\gamma_k\bigr)^T,
\qquad
\Gamma_{L,k}
:=
1-\left(1-\frac{7\rho_\psi^\eta}{8}\beta_k\right)^T,
\qquad
\beta_k:=\alpha_k\lambda_k .
\end{equation}
The tube-maintenance mechanism is that fixed-center contraction creates the gaps
\(\Gamma_{\psi,k}\) and \(\Gamma_{L,k}\), while the center drifts
\(\|y_\mu^\star(x_{k+1})-y_\mu^\star(x_k)\|_{y_\mu^\star(x_k)}\) and \(\|y_{\lambda_{k+1},\mu}^\star(x_{k+1})
-y_{\lambda_k,\mu}^\star(x_k)\|_{y_{\lambda_k,\mu}^\star(x_k)}\) consume part of these gaps.

\begin{lemma}[Explicit drift bounds]
\label{lem:tube_drift_explicit}
Suppose Assumption~\ref{ass:euclidean} holds, the local consequences in
Proposition~\ref{prop:local_consequences} hold, and the multiplier schedule satisfies
\[
\frac{\delta_k}{\lambda_k}
\le
\frac{T\rho_\psi^\eta}{16}\beta_k,
\qquad
\delta_k:=\lambda_{k+1}-\lambda_k\ge0,
\qquad
\lambda_k\ge\lambda_0 .
\]
Then
\begin{equation}
\label{eq:exact_drift_explicit}
\|y_\mu^\star(x_{k+1})-y_\mu^\star(x_k)\|_{y_\mu^\star(x_k)}
\le
\xi\,\ell_{*,0}^{\eta}
\left(\frac{\ell_{f,0}}{\lambda_0}+2\ell_{g,0}\right)\beta_k ,
\end{equation}
and
\begin{equation}
\label{eq:proxy_drift_explicit}
\|y_{\lambda_{k+1},\mu}^\star(x_{k+1})
-y_{\lambda_k,\mu}^\star(x_k)\|_{y_{\lambda_k,\mu}^\star(x_k)}
\le
\left(
\frac{T\ell_{f,0}^{\eta}}{8\lambda_0}
+
\xi\,\ell_{\lambda,0}^{\eta}
\left(\frac{\ell_{f,0}}{\lambda_0}+2\ell_{g,0}\right)
\right)\beta_k .
\end{equation}
\end{lemma}

\begin{proof}
First, by the definition of \(q_k^x\), the identity \(\nabla_x\psi_\mu=\nabla_x g\), and the coarse
gradient bounds,
\[
\|q_k^x\|_2
\le
\ell_{f,0}+2\lambda_k\ell_{g,0}
\le
\lambda_k
\left(\frac{\ell_{f,0}}{\lambda_0}+2\ell_{g,0}\right).
\]
Since \(x_{k+1}=x_k-\xi\alpha_k q_k^x\), this gives
\[
\|x_{k+1}-x_k\|_2
\le
\xi\alpha_k\lambda_k
\left(\frac{\ell_{f,0}}{\lambda_0}+2\ell_{g,0}\right)
=
\xi\beta_k
\left(\frac{\ell_{f,0}}{\lambda_0}+2\ell_{g,0}\right).
\]
The exact-center bound follows from Proposition~\ref{prop:local_consequences}(ii):
\[
\|y_\mu^\star(x_{k+1})-y_\mu^\star(x_k)\|_{y_\mu^\star(x_k)}
\le
\ell_{*,0}^{\eta}\|x_{k+1}-x_k\|_2,
\]
which proves \eqref{eq:exact_drift_explicit}.

For the proxy center, Proposition~\ref{prop:local_consequences}(i), applied with
\((\lambda_1,x_1)=(\lambda_k,x_k)\) and
\((\lambda_2,x_2)=(\lambda_{k+1},x_{k+1})\), gives
\[
\|y_{\lambda_{k+1},\mu}^\star(x_{k+1})
-y_{\lambda_k,\mu}^\star(x_k)\|_{y_{\lambda_k,\mu}^\star(x_k)}
\le
\frac{2\ell_{f,0}^{\eta}}{\rho_\psi^\eta}
\frac{\delta_k}{\lambda_k\lambda_{k+1}}
+
\ell_{\lambda,0}^{\eta}\|x_{k+1}-x_k\|_2 .
\]
Using \(\lambda_{k+1}\ge\lambda_0\) and
\(\delta_k/\lambda_k\le T\rho_\psi^\eta\beta_k/16\), the first term is bounded by
\[
\frac{2\ell_{f,0}^{\eta}}{\rho_\psi^\eta}
\frac{\delta_k}{\lambda_k\lambda_{k+1}}
\le
\frac{2\ell_{f,0}^{\eta}}{\rho_\psi^\eta}
\frac{T\rho_\psi^\eta\beta_k}{16\lambda_0}
=
\frac{T\ell_{f,0}^{\eta}}{8\lambda_0}\beta_k .
\]
Combining this with the bound on \(\|x_{k+1}-x_k\|_2\) proves
\eqref{eq:proxy_drift_explicit}.
\end{proof}

\subsection{Closed-form Sufficient Conditions for the Drift Bounds}

\begin{corollary}[\textit{barrier-aware} schedules imply drift closure]
\label{cor:tube_drift_closure}
Suppose the schedule satisfies \(\beta_k\le\gamma_k\),
\(T\rho_\psi^\eta\gamma_k\le1\), \(T\rho_\psi^\eta\beta_k\le1\), and the multiplier-growth
condition in Definition~\ref{def:barrier_aware}. If
\begin{equation}
\label{eq:tube_closure_xi_conditions}
\frac{\xi}{T}
\le
\frac{\eta\rho_\psi^\eta}
{16\ell_{*,0}^{\eta}
\left(\ell_{f,0}/\lambda_0+2\ell_{g,0}\right)},
\qquad
\frac{\xi}{T}
\le
\frac{\eta\rho_\psi^\eta}
{64\ell_{\lambda,0}^{\eta}
\left(\ell_{f,0}/\lambda_0+2\ell_{g,0}\right)},
\end{equation}
and
\begin{equation}
\label{eq:tube_closure_lambda0_condition}
\lambda_0
\ge
\frac{8\ell_{f,0}^{\eta}}{\eta\rho_\psi^\eta},
\end{equation}
then the center drifts satisfy
\begin{equation}
\label{eq:tube_drift_dominance}
\|y_\mu^\star(x_{k+1})-y_\mu^\star(x_k)\|_{y_\mu^\star(x_k)}
\le
\frac{\eta}{8}\Gamma_{\psi,k},
\qquad
\|y_{\lambda_{k+1},\mu}^\star(x_{k+1})
-y_{\lambda_k,\mu}^\star(x_k)\|_{y_{\lambda_k,\mu}^\star(x_k)}
\le
\frac{\eta}{8}\Gamma_{L,k}.
\end{equation}
In particular, the first two terms in the tube-closure condition of
Definition~\ref{def:barrier_aware} \eqref{eq:main_sched_s3}, together with the lower bound on \(\lambda_0\), imply the
abstract drift conditions required for tube maintenance.
\end{corollary}

\begin{proof}
Since \(T\rho_\psi^\eta\gamma_k\le1\), the elementary bound
\(1-(1-a)^T\ge Ta/2\) for \(Ta\le1\) gives
\[
\Gamma_{\psi,k}
=
1-(1-\rho_\psi^\eta\gamma_k)^T
\ge
\frac{T\rho_\psi^\eta}{2}\gamma_k .
\]
Similarly, since \(T\rho_\psi^\eta\beta_k\le1\),
\[
\Gamma_{L,k}
=
1-\left(1-\frac{7\rho_\psi^\eta}{8}\beta_k\right)^T
\ge
\frac{7T\rho_\psi^\eta}{16}\beta_k
\ge
\frac{T\rho_\psi^\eta}{4}\beta_k .
\]

For the exact center, Lemma~\ref{lem:tube_drift_explicit} and \(\beta_k\le\gamma_k\) give
\[
\|y_\mu^\star(x_{k+1})-y_\mu^\star(x_k)\|_{y_\mu^\star(x_k)}
\le
\xi\ell_{*,0}^{\eta}
\left(\frac{\ell_{f,0}}{\lambda_0}+2\ell_{g,0}\right)\gamma_k .
\]
The first inequality in \eqref{eq:tube_closure_xi_conditions} implies
\[
\|y_\mu^\star(x_{k+1})-y_\mu^\star(x_k)\|_{y_\mu^\star(x_k)}
\le
\frac{\eta T\rho_\psi^\eta}{16}\gamma_k
\le
\frac{\eta}{8}\Gamma_{\psi,k}.
\]

For the proxy center, Lemma~\ref{lem:tube_drift_explicit} gives
\[
\|y_{\lambda_{k+1},\mu}^\star(x_{k+1})
-y_{\lambda_k,\mu}^\star(x_k)\|_{y_{\lambda_k,\mu}^\star(x_k)}
\le
\left(
\frac{T\ell_{f,0}^{\eta}}{8\lambda_0}
+
\xi\ell_{\lambda,0}^{\eta}
\left(\frac{\ell_{f,0}}{\lambda_0}+2\ell_{g,0}\right)
\right)\beta_k .
\]
The lower bound \eqref{eq:tube_closure_lambda0_condition} implies
\[
\frac{T\ell_{f,0}^{\eta}}{8\lambda_0}
\le
\frac{\eta T\rho_\psi^\eta}{64},
\]
and the second inequality in \eqref{eq:tube_closure_xi_conditions} implies
\[
\xi\ell_{\lambda,0}^{\eta}
\left(\frac{\ell_{f,0}}{\lambda_0}+2\ell_{g,0}\right)
\le
\frac{\eta T\rho_\psi^\eta}{64}.
\]
Therefore
\[
\|y_{\lambda_{k+1},\mu}^\star(x_{k+1})
-y_{\lambda_k,\mu}^\star(x_k)\|_{y_{\lambda_k,\mu}^\star(x_k)}
\le
\frac{\eta T\rho_\psi^\eta}{32}\beta_k
\le
\frac{\eta}{8}\Gamma_{L,k}.
\]
This proves \eqref{eq:tube_drift_dominance}.
\end{proof}

\begin{remark}
Corollary~\ref{cor:tube_drift_closure} is the explicit closure calculation behind the first two terms
of the tube-closure condition in Definition~\ref{def:barrier_aware} \eqref{eq:main_sched_s3}. The remaining term in
Definition~\ref{def:barrier_aware} \eqref{eq:main_sched_s3} controls the higher-order anchor-switch factors appearing in the
moving-anchor recursions of Appendix~\ref{app:tracker}. Thus the full \textit{barrier-aware} schedule
contains exactly the ingredients needed for both tube maintenance and the later Lyapunov descent.
\end{remark}

\subsection{Proof of Theorem~\ref{thm:tube_invariance}}
\label{app:proof_tube_invariance}

We first record a deterministic neighborhood-transition lemma.

\begin{lemma}[Contraction plus drift implies neighborhood transition]
\label{lem:contraction_plus_drift_tube}
Let \(\bar y,\bar y^+\in\operatorname{int}(Y)\), let \(\eta\in(0,1/2)\), and let
\(\Gamma\in[0,1]\). Suppose
\[
\|\bar y^+-\bar y\|_{\bar y}\le \frac{\eta}{8}\Gamma.
\]
If \(u\in\operatorname{int}(Y)\) satisfies
\[
\|u-\bar y\|_{\bar y}\le \eta\sqrt{1-\Gamma},
\]
then
\[
\|u-\bar y^+\|_{\bar y^+}\le \eta.
\]
\end{lemma}

\begin{proof}
Let \(d:=\|\bar y^+-\bar y\|_{\bar y}\). Since \(d\le \eta\Gamma/8<1\), the anchor-switch bound gives
\[
\|v\|_{\bar y^+}\le (1-d)^{-1}\|v\|_{\bar y}
\qquad \forall v.
\]
Therefore
\[
\|u-\bar y^+\|_{\bar y^+}
\le
(1-d)^{-1}\bigl(\|u-\bar y\|_{\bar y}+\|\bar y^+-\bar y\|_{\bar y}\bigr).
\]
Using the assumptions,
\[
\|u-\bar y^+\|_{\bar y^+}
\le
\eta\frac{\sqrt{1-\Gamma}+\Gamma/8}{1-\eta\Gamma/8}.
\]
Since \(\sqrt{1-\Gamma}\le 1-\Gamma/2\) for \(\Gamma\in[0,1]\) and \(\eta<1/2\),
\[
\sqrt{1-\Gamma}+\frac{\Gamma}{8}
\le
1-\frac{3\Gamma}{8}
\le
1-\frac{\eta\Gamma}{8}.
\]
Thus \(\|u-\bar y^+\|_{\bar y^+}\le\eta\).
\end{proof}

\paragraph{Proof of Theorem~\ref{thm:tube_invariance}}
\begin{proof}
The proof is by induction on the outer iteration \(k\). The base case follows from the assumed
initialization
\[
z_0\in T_{\eta,\mu}(x_0),
\qquad
y_0\in T_{\eta,\mu}^{\lambda_0}(x_0).
\]

Assume that, at the beginning of outer iteration \(k\),
\[
z_k\in T_{\eta,\mu}(x_k),
\qquad
y_k\in T_{\eta,\mu}^{\lambda_k}(x_k).
\]
We prove that all inner iterates remain in the corresponding target neighborhoods and that the next warm
starts lie in the target neighborhoods at \(x_{k+1}\) and \(\lambda_{k+1}\).

First consider the exact tracker. Since \(z_k^{(0)}=z_k\in T_{\eta,\mu}(x_k)\), and
\(T_{\eta,\mu}(x_k)\subseteq T_{2\eta,\mu}(x_k)\), the local Dikin regularity estimates are valid at
\(z_k^{(0)}\). The one-step contraction estimate from Lemma~\ref{lem:z_fixed_anchor_contraction} gives, whenever
\(z_k^{(t)}\in T_{\eta,\mu}(x_k)\),
\[
\|z_k^{(t+1)}-y_\mu^\star(x_k)\|_{y_\mu^\star(x_k)}^2
\le
(1-\rho_\psi^\eta\gamma_k)
\|z_k^{(t)}-y_\mu^\star(x_k)\|_{y_\mu^\star(x_k)}^2 .
\]
Thus, by induction over \(t=0,\ldots,T-1\), every exact inner iterate remains in
\(T_{\eta,\mu}(x_k)\). After \(T\) inner steps,
\[
\|z_{k+1}-y_\mu^\star(x_k)\|_{y_\mu^\star(x_k)}^2
\le
(1-\Gamma_{\psi,k})
\|z_k-y_\mu^\star(x_k)\|_{y_\mu^\star(x_k)}^2
\le
(1-\Gamma_{\psi,k})\eta^2 .
\]
By Corollary~\ref{cor:tube_drift_closure},
\[
\|y_\mu^\star(x_{k+1})-y_\mu^\star(x_k)\|_{y_\mu^\star(x_k)}
\le
\frac{\eta}{8}\Gamma_{\psi,k}.
\]
Applying Lemma~\ref{lem:contraction_plus_drift_tube} with
\[
\bar y=y_\mu^\star(x_k),
\qquad
\bar y^+=y_\mu^\star(x_{k+1}),
\qquad
u=z_{k+1},
\qquad
\Gamma=\Gamma_{\psi,k},
\]
we obtain
\[
z_{k+1}\in T_{\eta,\mu}(x_{k+1}).
\]

The proxy tracker is handled in the same way. Since
\(y_k^{(0)}=y_k\in T_{\eta,\mu}^{\lambda_k}(x_k)\), the local proxy neighborhood estimates are valid at the
initial proxy inner iterate. The one-step contraction estimate from Lemma~\ref{lem:y_fixed_anchor_contraction}
implies that whenever \(y_k^{(t)}\in T_{\eta,\mu}^{\lambda_k}(x_k)\),
\[
\|y_k^{(t+1)}-y_{\lambda_k,\mu}^\star(x_k)\|_{y_{\lambda_k,\mu}^\star(x_k)}^2
\le
\left(1-\frac{7\rho_\psi^\eta}{8}\beta_k\right)
\|y_k^{(t)}-y_{\lambda_k,\mu}^\star(x_k)\|_{y_{\lambda_k,\mu}^\star(x_k)}^2 .
\]
Hence all proxy inner iterates remain in \(T_{\eta,\mu}^{\lambda_k}(x_k)\), and after \(T\) steps,
\[
\|y_{k+1}-y_{\lambda_k,\mu}^\star(x_k)\|_{y_{\lambda_k,\mu}^\star(x_k)}^2
\le
(1-\Gamma_{L,k})
\|y_k-y_{\lambda_k,\mu}^\star(x_k)\|_{y_{\lambda_k,\mu}^\star(x_k)}^2
\le
(1-\Gamma_{L,k})\eta^2 .
\]
By Corollary~\ref{cor:tube_drift_closure},
\[
\|y_{\lambda_{k+1},\mu}^\star(x_{k+1})
-y_{\lambda_k,\mu}^\star(x_k)\|_{y_{\lambda_k,\mu}^\star(x_k)}
\le
\frac{\eta}{8}\Gamma_{L,k}.
\]
Applying Lemma~\ref{lem:contraction_plus_drift_tube} with
\[
\bar y=y_{\lambda_k,\mu}^\star(x_k),
\qquad
\bar y^+=y_{\lambda_{k+1},\mu}^\star(x_{k+1}),
\qquad
u=y_{k+1},
\qquad
\Gamma=\Gamma_{L,k},
\]
we obtain
\[
y_{k+1}\in T_{\eta,\mu}^{\lambda_{k+1}}(x_{k+1}).
\]

Therefore the induction hypothesis propagates from \(k\) to \(k+1\). By induction, for every outer
iteration \(k\), all exact inner iterates remain in \(T_{\eta,\mu}(x_k)\), all proxy inner iterates
remain in \(T_{\eta,\mu}^{\lambda_k}(x_k)\), and the next warm starts satisfy
\[
z_{k+1}\in T_{\eta,\mu}(x_{k+1}),
\qquad
y_{k+1}\in T_{\eta,\mu}^{\lambda_{k+1}}(x_{k+1}).
\]
This proves forward invariance of the exact and proxy target neighborhoods.
\end{proof}

\section{Proof of Theorem~\ref{thm:deterministic_poly}}
\label{app:deterministic}

\subsection{One-step Descent Decomposition}
\label{app:one_step_descent}

We now combine the two ingredients established in the previous appendices.
Appendix~\ref{app:local_consequences} shows that, inside the maintained Dikin neighborhoods, the barrier-smoothed
outer objective \(F_\mu\) is locally smooth and the proxy-gradient bias is \(O(1/\lambda)\).
Appendix~\ref{app:tracker} then quantifies how the two inner trackers evolve in anchored Dikin
norms, producing fixed-center contraction and moving-anchor recursions for the proxy and exact
tracking errors.

The purpose of this subsection is to merge these facts into a one-step descent estimate for the outer
objective, and then for a Lyapunov function coupling \(F_\mu\) with the two tracker errors. The first
result expresses the one-step decrease of \(F_\mu(x_k)\) in terms of the current-center proxy and
exact tracking errors, together with the \(O(1/\lambda_k)\) proxy bias. These current-center terms are
then handled by the tracker recursions from Appendix~\ref{app:tracker}.

\begin{proposition}[Outer descent under current-center tracker errors]
\label{prop:outer_descent}
Fix an outer iterate \(x_k\in X\) and suppose
\begin{equation}
\label{eq:outer_smooth_stepsize}
\xi L_F^\eta\alpha_k\le \frac12 .
\end{equation}
Then, for each \(k\),
\begin{align}
F_\mu(x_{k+1})-F_\mu(x_k)
&\le
-\frac{\xi\alpha_k}{4}
\Bigl(2\|\nabla F_\mu(x_k)\|_2^2+\|q_k^x\|_2^2\Bigr)
\notag\\
&\quad
+\frac{3\xi\alpha_k}{2}
\Bigl[
(\ell_{f,1}^\eta+\lambda_k\ell_{\psi,1}^\eta)^2
\big\|y_{k+1}-y_{\lambda_k,\mu}^\star(x_k)\big\|_{y_{\lambda_k,\mu}^\star(x_k)}^2
\notag\\
&\qquad\qquad
+(\lambda_k\ell_{\psi,1}^\eta)^2
\big\|z_{k+1}-y_\mu^\star(x_k)\big\|_{y_\mu^\star(x_k)}^2
+\frac{(c_x^\eta)^2}{\lambda_k^2}
\Bigr].
\label{eq:outer_descent_current_center}
\end{align}
Here \(c_x^\eta\) is the local Dikin constant from the proxy-gradient bias estimate in
Proposition~\ref{prop:local_consequences}\textup{(iv)}.

Moreover, if
\begin{equation}
\label{eq:outer_balance_stepsize}
\lambda_k\ge \lambda_0\ge \frac{2\ell_{f,1}^\eta}{\rho_\psi^\eta},
\qquad
\frac{\xi}{T}
\le
c_\xi\frac{\rho_\psi^\eta}{(\ell_{\psi,1}^\eta)^2},
\end{equation}
for a sufficiently small absolute constant \(c_\xi>0\), then
\begin{align}
F_\mu(x_{k+1})-F_\mu(x_k)
&\le
-\frac{\xi\alpha_k}{4}
\Bigl(2\|\nabla F_\mu(x_k)\|_2^2+\|q_k^x\|_2^2\Bigr)
\notag\\
&\quad
+
\frac{T\rho_\psi^\eta\alpha_k\lambda_k^2}{32}
\Bigl(
2\big\|y_{k+1}-y_{\lambda_k,\mu}^\star(x_k)\big\|_{y_{\lambda_k,\mu}^\star(x_k)}^2
+
\big\|z_{k+1}-y_\mu^\star(x_k)\big\|_{y_\mu^\star(x_k)}^2
\Bigr)
\notag\\
&\quad
+
\frac{3\xi\alpha_k}{2}\frac{(c_x^\eta)^2}{\lambda_k^2}.
\label{eq:outer_descent_balanced}
\end{align}
\end{proposition}

\begin{proof}
By \(L_F^\eta\)-smoothness of \(F_\mu\) on \(X\),
\[
F_\mu(x_{k+1})
\le
F_\mu(x_k)
+
\langle \nabla F_\mu(x_k),x_{k+1}-x_k\rangle
+
\frac{L_F^\eta}{2}\|x_{k+1}-x_k\|_2^2 .
\]
Substituting the outer update \(x_{k+1}-x_k=-\xi\alpha_k q_k^x\) and using
\eqref{eq:outer_smooth_stepsize}, we get
\[
F_\mu(x_{k+1})-F_\mu(x_k)
\le
-\xi\alpha_k\langle \nabla F_\mu(x_k),q_k^x\rangle
+
\frac{\xi\alpha_k}{4}\|q_k^x\|_2^2 .
\]
Using
\[
\langle a,b\rangle
=
\frac12\bigl(\|a\|_2^2+\|b\|_2^2-\|a-b\|_2^2\bigr),
\]
we obtain
\begin{equation}
\label{eq:outer_basic_descent}
F_\mu(x_{k+1})-F_\mu(x_k)
\le
-\frac{\xi\alpha_k}{2}\|\nabla F_\mu(x_k)\|_2^2
-\frac{\xi\alpha_k}{4}\|q_k^x\|_2^2
+
\frac{\xi\alpha_k}{2}\|\nabla F_\mu(x_k)-q_k^x\|_2^2 .
\end{equation}

We now bound the proxy-direction error. Add and subtract
\(\nabla C_{\lambda_k,\mu}^\star(x_k)\):
\[
\|\nabla F_\mu(x_k)-q_k^x\|_2
\le
\|q_k^x-\nabla C_{\lambda_k,\mu}^\star(x_k)\|_2
+
\|\nabla C_{\lambda_k,\mu}^\star(x_k)-\nabla F_\mu(x_k)\|_2 .
\]
By Proposition~\ref{prop:local_consequences}\textup{(iv)},
\begin{equation}
\label{eq:outer_proxy_bias}
\|\nabla C_{\lambda_k,\mu}^\star(x_k)-\nabla F_\mu(x_k)\|_2
\le
\frac{c_x^\eta}{\lambda_k}.
\end{equation}
For the first term, the envelope identity for \(C_{\lambda_k,\mu}^\star\) gives
\[
\nabla C_{\lambda_k,\mu}^\star(x_k)
=
\nabla_x f(x_k,y_{\lambda_k,\mu}^\star(x_k))
+
\lambda_k
\Bigl(
\nabla_x\psi_\mu(x_k,y_{\lambda_k,\mu}^\star(x_k))
-
\nabla_x\psi_\mu(x_k,y_\mu^\star(x_k))
\Bigr).
\]
Therefore,
\begin{align*}
q_k^x-\nabla C_{\lambda_k,\mu}^\star(x_k)
&=
\nabla_x f(x_k,y_{k+1})
-
\nabla_x f(x_k,y_{\lambda_k,\mu}^\star(x_k))
\\
&\quad
+
\lambda_k
\Bigl(
\nabla_x\psi_\mu(x_k,y_{k+1})
-
\nabla_x\psi_\mu(x_k,y_{\lambda_k,\mu}^\star(x_k))
\Bigr)
\\
&\quad
-
\lambda_k
\Bigl(
\nabla_x\psi_\mu(x_k,z_{k+1})
-
\nabla_x\psi_\mu(x_k,y_\mu^\star(x_k))
\Bigr).
\end{align*}
Using the local cross-derivative bounds for \(f\) and \(\psi_\mu\) on the maintained proxy and
exact neighborhoods,
\[
\|\nabla_x f(x_k,y_{k+1})
-
\nabla_x f(x_k,y_{\lambda_k,\mu}^\star(x_k))\|_2
\le
\ell_{f,1}^\eta
\big\|y_{k+1}-y_{\lambda_k,\mu}^\star(x_k)\big\|_{y_{\lambda_k,\mu}^\star(x_k)},
\]
\[
\|\nabla_x\psi_\mu(x_k,y_{k+1})
-
\nabla_x\psi_\mu(x_k,y_{\lambda_k,\mu}^\star(x_k))\|_2
\le
\ell_{\psi,1}^\eta
\big\|y_{k+1}-y_{\lambda_k,\mu}^\star(x_k)\big\|_{y_{\lambda_k,\mu}^\star(x_k)},
\]
and
\[
\|\nabla_x\psi_\mu(x_k,z_{k+1})
-
\nabla_x\psi_\mu(x_k,y_\mu^\star(x_k))\|_2
\le
\ell_{\psi,1}^\eta
\big\|z_{k+1}-y_\mu^\star(x_k)\big\|_{y_\mu^\star(x_k)}.
\]
Combining these bounds with \eqref{eq:outer_proxy_bias} gives
\begin{align*}
\|\nabla F_\mu(x_k)-q_k^x\|_2
&\le
(\ell_{f,1}^\eta+\lambda_k\ell_{\psi,1}^\eta)
\big\|y_{k+1}-y_{\lambda_k,\mu}^\star(x_k)\big\|_{y_{\lambda_k,\mu}^\star(x_k)}
\\
&\quad
+
\lambda_k\ell_{\psi,1}^\eta
\big\|z_{k+1}-y_\mu^\star(x_k)\big\|_{y_\mu^\star(x_k)}
+
\frac{c_x^\eta}{\lambda_k}.
\end{align*}
Squaring and using \((a+b+c)^2\le 3(a^2+b^2+c^2)\), we obtain
\begin{align}
\|\nabla F_\mu(x_k)-q_k^x\|_2^2
&\le
3\Bigl[
(\ell_{f,1}^\eta+\lambda_k\ell_{\psi,1}^\eta)^2
\big\|y_{k+1}-y_{\lambda_k,\mu}^\star(x_k)\big\|_{y_{\lambda_k,\mu}^\star(x_k)}^2
\notag\\
&\qquad
+
(\lambda_k\ell_{\psi,1}^\eta)^2
\big\|z_{k+1}-y_\mu^\star(x_k)\big\|_{y_\mu^\star(x_k)}^2
+
\frac{(c_x^\eta)^2}{\lambda_k^2}
\Bigr].
\label{eq:outer_error_square}
\end{align}
Plugging \eqref{eq:outer_error_square} into \eqref{eq:outer_basic_descent} proves
\eqref{eq:outer_descent_current_center}.

It remains to derive the simplified current-center form. Since
\(\lambda_k\ge \lambda_0\ge 2\ell_{f,1}^\eta/\rho_\psi^\eta\) and
\(\ell_{\psi,1}^\eta\ge \rho_\psi^\eta\), we have
\[
\ell_{f,1}^\eta
\le
\frac12\lambda_k\rho_\psi^\eta
\le
\frac12\lambda_k\ell_{\psi,1}^\eta,
\]
and hence
\[
\ell_{f,1}^\eta+\lambda_k\ell_{\psi,1}^\eta
\le
\frac32\lambda_k\ell_{\psi,1}^\eta .
\]
By choosing \(c_\xi\) in \eqref{eq:outer_balance_stepsize} sufficiently small, the two tracker
error coefficients in \eqref{eq:outer_descent_current_center} are bounded by
\[
\frac{T\rho_\psi^\eta\alpha_k\lambda_k^2}{16}
\quad\text{for the proxy error,}
\qquad
\frac{T\rho_\psi^\eta\alpha_k\lambda_k^2}{32}
\quad\text{for the exact error.}
\]
Therefore \eqref{eq:outer_descent_current_center} reduces to
\eqref{eq:outer_descent_balanced}.
\end{proof}

\subsection{Lyapunov Descent}
\label{app:lyapunov}

The outer descent estimate in Proposition~\ref{prop:outer_descent} contains positive terms involving
the current-center proxy and exact tracking errors. The moving-anchor recursions in
Appendix~\ref{app:tracker} show that these terms can be absorbed by contraction of the warm-start
errors \(I_k\) and \(J_k\), provided the schedule controls center drift and anchor-switch effects. We
therefore introduce a Lyapunov function coupling \(F_\mu(x_k)\) with the two anchored tracker
errors. The resulting inequality is the main deterministic descent step: it leaves a negative gradient
term, negative tracker-error terms, and only summable \(O(\alpha_k/\lambda_k^2)\) and
\(O(\delta_k/\lambda_k^2)\) remainders.

\begin{proposition}[Lyapunov descent]
\label{prop:lyapunov_descent}
Fix a neighborhood radius \(\eta\in(0,1/2)\) and an inner-loop length \(T\in\mathbb N\). Suppose the
\textit{barrier-aware} conditions in Definition~\ref{def:barrier_aware} hold, and the iterates remain in the
target neighborhoods guaranteed by Theorem~\ref{thm:tube_invariance}. Define the tracker errors
\begin{equation}
\label{eq:lyap_tracker_errors}
I_k
:=
\big\|y_k-y_{\lambda_k,\mu}^\star(x_k)\big\|_{y_{\lambda_k,\mu}^\star(x_k)}^2,
\qquad
J_k
:=
\big\|z_k-y_\mu^\star(x_k)\big\|_{y_\mu^\star(x_k)}^2,
\end{equation}
and the Lyapunov function
\begin{equation}
\label{eq:lyap_function}
V_k
:=
F_\mu(x_k)
+
\lambda_k\ell_{\psi,1}^\eta I_k
+
\frac{\lambda_k\ell_{\psi,1}^\eta}{2}J_k .
\end{equation}
Then, for each \(k\),
\begin{align}
V_{k+1}-V_k
&\le
-\frac{\xi\alpha_k}{2}\|\nabla F_\mu(x_k)\|_2^2
-\frac{\lambda_k\ell_{\psi,1}^\eta T\rho_\psi^\eta\beta_k}{8}\,
\big\|y_k-y_{\lambda_k,\mu}^\star(x_k)\big\|_{y_{\lambda_k,\mu}^\star(x_k)}^2
\nonumber\\
&\quad
-\frac{\lambda_k\ell_{\psi,1}^\eta T\rho_\psi^\eta\gamma_k}{16}\,
\big\|z_k-y_\mu^\star(x_k)\big\|_{y_\mu^\star(x_k)}^2
\nonumber\\
&\quad
+
O\!\left(\xi(c_x^\eta)^2\right)\frac{\alpha_k}{\lambda_k^2}
+
O\!\left(
\frac{\ell_{\psi,1}^\eta(\ell_{f,0}^\eta)^2}{(\rho_\psi^\eta)^3}
\right)
\frac{\delta_k}{\lambda_k^2}.
\label{eq:lyap_descent}
\end{align}
\end{proposition}

\begin{proof}
From the definition of \(V_k\),
\begin{align}
V_{k+1}-V_k
&=
F_\mu(x_{k+1})-F_\mu(x_k)
+
\ell_{\psi,1}^\eta
\bigl(\lambda_{k+1}I_{k+1}-\lambda_k I_k\bigr)
\nonumber\\
&\quad
+
\frac{\ell_{\psi,1}^\eta}{2}
\bigl(\lambda_{k+1}J_{k+1}-\lambda_k J_k\bigr).
\label{eq:lyap_difference_identity}
\end{align}
We use Proposition~\ref{prop:outer_descent} in its balanced current-center form. Namely, after the
outer descent estimate and the \textit{barrier-aware} step-size bounds, the current-center tracking errors enter
with budgets
\[
\frac{\lambda_k\ell_{\psi,1}^\eta T\rho_\psi^\eta\beta_k}{16}
\big\|y_{k+1}-y_{\lambda_k,\mu}^\star(x_k)\big\|_{y_{\lambda_k,\mu}^\star(x_k)}^2
\]
for the proxy tracker, and
\[
\frac{\lambda_k\ell_{\psi,1}^\eta T\rho_\psi^\eta\gamma_k}{64}
\big\|z_{k+1}-y_\mu^\star(x_k)\big\|_{y_\mu^\star(x_k)}^2
\]
for the exact tracker. The latter uses \(\beta_k\le\gamma_k\). Substituting this into
\eqref{eq:lyap_difference_identity} gives
\begin{align}
V_{k+1}-V_k
&\le
-\frac{\xi\alpha_k}{2}\|\nabla F_\mu(x_k)\|_2^2
-\frac{\xi\alpha_k}{4}\|q_k^x\|_2^2
+
\frac{3\xi\alpha_k}{2}\frac{(c_x^\eta)^2}{\lambda_k^2}
\nonumber\\
&\quad
+\ell_{\psi,1}^\eta
\underbrace{
\Bigl(
\lambda_{k+1}I_{k+1}
+
\frac{\lambda_kT\rho_\psi^\eta\beta_k}{16}
\big\|y_{k+1}-y_{\lambda_k,\mu}^\star(x_k)\big\|_{y_{\lambda_k,\mu}^\star(x_k)}^2
-\lambda_k I_k
\Bigr)
}_{(i)}
\nonumber\\
&\quad
+\frac{\ell_{\psi,1}^\eta}{2}
\underbrace{
\Bigl(
\lambda_{k+1}J_{k+1}
+
\frac{\lambda_kT\rho_\psi^\eta\gamma_k}{32}
\big\|z_{k+1}-y_\mu^\star(x_k)\big\|_{y_\mu^\star(x_k)}^2
-\lambda_k J_k
\Bigr)
}_{(ii)} .
\label{eq:lyap_decomposition}
\end{align}

We first control the proxy block \((i)\). By
Corollary~\ref{cor:y_moving_anchor_scheduled},
\begin{align}
(i)
&\le
\lambda_k\kappa_{y,k}^2
\left(
1+\frac{3T\rho_\psi^\eta\beta_k}{8}
+\frac{\delta_k}{\lambda_k}
\right)
\big\|y_{k+1}-y_{\lambda_k,\mu}^\star(x_k)\big\|_{y_{\lambda_k,\mu}^\star(x_k)}^2
-\lambda_k I_k
+
(iii),
\label{eq:lyap_proxy_block_raw}
\end{align}
where
\begin{equation}
\label{eq:lyap_proxy_remainder}
(iii)
:=
O\!\left(\xi^2(\ell_{*,0}^\eta)^2\right)
\frac{\lambda_k\alpha_k^2}{\rho_\psi^\eta T\beta_k}\|q_k^x\|_2^2
+
O\!\left(\frac{(\ell_{f,0}^\eta)^2}{(\rho_\psi^\eta)^3}\right)
\frac{\delta_k}{\lambda_k^2}.
\end{equation}
Here, as usual, \(\ell_{*,0}^\eta\) may be enlarged to dominate the proxy center-stability constant
\(\ell_{\lambda,0}^\eta\).

The multiplier-growth condition \(\delta_k/\lambda_k\le T\rho_\psi^\eta\beta_k/16\) implies
\[
1+\frac{3T\rho_\psi^\eta\beta_k}{8}+\frac{\delta_k}{\lambda_k}
\le
1+\frac{T\rho_\psi^\eta\beta_k}{2}.
\]
Therefore
\begin{align}
(i)
&\le
\lambda_k\kappa_{y,k}^2
\left(1+\frac{T\rho_\psi^\eta\beta_k}{2}\right)
\big\|y_{k+1}-y_{\lambda_k,\mu}^\star(x_k)\big\|_{y_{\lambda_k,\mu}^\star(x_k)}^2
-\lambda_k I_k
+
(iii).
\label{eq:lyap_proxy_block_simplified}
\end{align}
By Corollary~\ref{cor:y_fixed_anchor_scheduled},
\[
\big\|y_{k+1}-y_{\lambda_k,\mu}^\star(x_k)\big\|_{y_{\lambda_k,\mu}^\star(x_k)}^2
\le
\left(1-\frac{3T\rho_\psi^\eta\beta_k}{4}\right)I_k .
\]
Thus
\begin{align}
(i)
&\le
\lambda_k
\left[
\kappa_{y,k}^2
\left(1+\frac{T\rho_\psi^\eta\beta_k}{2}\right)
\left(1-\frac{3T\rho_\psi^\eta\beta_k}{4}\right)
-1
\right]I_k
+
(iii).
\label{eq:lyap_proxy_block_after_contraction}
\end{align}
The tube-closure bound gives
\[
\big\|y_{\lambda_{k+1},\mu}^\star(x_{k+1})
-
y_{\lambda_k,\mu}^\star(x_k)\big\|_{y_{\lambda_k,\mu}^\star(x_k)}
\le
\frac{T\rho_\psi^\eta\beta_k}{32},
\]
so
\[
\kappa_{y,k}^2
=
\left(
1-
\big\|y_{\lambda_{k+1},\mu}^\star(x_{k+1})
-
y_{\lambda_k,\mu}^\star(x_k)\big\|_{y_{\lambda_k,\mu}^\star(x_k)}
\right)^{-2}.
\]
Using \(T\rho_\psi^\eta\beta_k\le 1/4\), the elementary scalar inequality
\[
\frac{1}{(1-s)^2}\left(1+\frac{a}{2}\right)\left(1-\frac{3a}{4}\right)
\le
1-\frac{a}{8},
\qquad
0\le s\le \frac{a}{32},\quad 0\le a\le \frac14,
\]
with \(a=T\rho_\psi^\eta\beta_k\), yields
\begin{equation}
\label{eq:lyap_proxy_scalar_absorption}
\kappa_{y,k}^2
\left(1+\frac{T\rho_\psi^\eta\beta_k}{2}\right)
\left(1-\frac{3T\rho_\psi^\eta\beta_k}{4}\right)
\le
1-\frac{T\rho_\psi^\eta\beta_k}{8}.
\end{equation}
This replaces the long anchor-switch display and keeps the expression within the page margins.
Combining \eqref{eq:lyap_proxy_block_after_contraction} and
\eqref{eq:lyap_proxy_scalar_absorption}, we obtain
\begin{equation}
\label{eq:lyap_proxy_block_final}
(i)
\le
-\frac{\lambda_kT\rho_\psi^\eta\beta_k}{8}I_k
+
O\!\left(\xi^2(\ell_{*,0}^\eta)^2\right)
\frac{\alpha_k}{\rho_\psi^\eta T}\|q_k^x\|_2^2
+
O\!\left(\frac{(\ell_{f,0}^\eta)^2}{(\rho_\psi^\eta)^3}\right)
\frac{\delta_k}{\lambda_k^2}.
\end{equation}

We next control the exact block \((ii)\). By
Corollary~\ref{cor:z_moving_anchor_scheduled},
\begin{align}
(ii)
&\le
\lambda_k\kappa_{z,k}^2
\left(
1+\frac{\delta_k}{\lambda_k}
+\frac{3T\rho_\psi^\eta\gamma_k}{8}
+\frac{T\rho_\psi^\eta\beta_k}{32}
\right)
\big\|z_{k+1}-y_\mu^\star(x_k)\big\|_{y_\mu^\star(x_k)}^2
-\lambda_k J_k
+
(iv),
\label{eq:lyap_exact_block_raw}
\end{align}
where
\begin{equation}
\label{eq:lyap_exact_remainder}
(iv)
:=
O\!\left(\xi^2(\ell_{*,0}^\eta)^2\right)
\frac{\lambda_{k+1}\alpha_k^2}{T\rho_\psi^\eta\gamma_k}\|q_k^x\|_2^2 .
\end{equation}
Using \(\delta_k/\lambda_k\le T\rho_\psi^\eta\beta_k/16\) and \(\beta_k\le\gamma_k\), the parenthesis
in \eqref{eq:lyap_exact_block_raw} is bounded by \(1+T\rho_\psi^\eta\gamma_k/2\). Hence
\begin{align}
(ii)
&\le
\lambda_k\kappa_{z,k}^2
\left(1+\frac{T\rho_\psi^\eta\gamma_k}{2}\right)
\big\|z_{k+1}-y_\mu^\star(x_k)\big\|_{y_\mu^\star(x_k)}^2
-\lambda_kJ_k
+
(iv).
\label{eq:lyap_exact_block_simplified}
\end{align}
By Corollary~\ref{cor:z_fixed_anchor_scheduled},
\[
\big\|z_{k+1}-y_\mu^\star(x_k)\big\|_{y_\mu^\star(x_k)}^2
\le
\left(1-\frac{3T\rho_\psi^\eta\gamma_k}{4}\right)J_k .
\]
Therefore
\begin{align}
(ii)
&\le
\lambda_k
\left[
\kappa_{z,k}^2
\left(1+\frac{T\rho_\psi^\eta\gamma_k}{2}\right)
\left(1-\frac{3T\rho_\psi^\eta\gamma_k}{4}\right)
-1
\right]J_k
+
(iv).
\label{eq:lyap_exact_block_after_contraction}
\end{align}
The tube-closure bound gives
\[
\big\|y_\mu^\star(x_{k+1})-y_\mu^\star(x_k)\big\|_{y_\mu^\star(x_k)}
\le
\frac{T\rho_\psi^\eta\gamma_k}{32}.
\]
Using the same scalar inequality as above with \(a=T\rho_\psi^\eta\gamma_k\), we get
\begin{equation}
\label{eq:lyap_exact_scalar_absorption}
\kappa_{z,k}^2
\left(1+\frac{T\rho_\psi^\eta\gamma_k}{2}\right)
\left(1-\frac{3T\rho_\psi^\eta\gamma_k}{4}\right)
\le
1-\frac{T\rho_\psi^\eta\gamma_k}{8}.
\end{equation}
This is the compact replacement for the long expression that was previously overflowing the page.
Thus
\begin{equation}
\label{eq:lyap_exact_block_final}
(ii)
\le
-\frac{\lambda_kT\rho_\psi^\eta\gamma_k}{8}J_k
+
O\!\left(\xi^2(\ell_{*,0}^\eta)^2\right)
\frac{\alpha_k\beta_k}{T\rho_\psi^\eta\gamma_k}\|q_k^x\|_2^2,
\end{equation}
where we used \(\lambda_{k+1}\alpha_k^2=O(\alpha_k\beta_k)\), which follows from
\(\lambda_{k+1}=\lambda_k+\delta_k\), \(\delta_k/\lambda_k=O(\beta_k)\), and
\(\beta_k=\alpha_k\lambda_k\).

Substituting \eqref{eq:lyap_proxy_block_final} and \eqref{eq:lyap_exact_block_final} into
\eqref{eq:lyap_decomposition} gives
\begin{align}
V_{k+1}-V_k
&\le
-\frac{\xi\alpha_k}{2}\|\nabla F_\mu(x_k)\|_2^2
+
O\!\left(\xi(c_x^\eta)^2\right)\frac{\alpha_k}{\lambda_k^2}
\nonumber\\
&\quad
-\frac{\xi\alpha_k}{4}
\left[
1
-
O\!\left(
\frac{\xi\ell_{\psi,1}^\eta(\ell_{*,0}^\eta)^2}{\rho_\psi^\eta T}
\right)
-
O\!\left(
\frac{\xi\ell_{\psi,1}^\eta(\ell_{*,0}^\eta)^2\beta_k}
{\rho_\psi^\eta T\gamma_k}
\right)
\right]
\|q_k^x\|_2^2
\nonumber\\
&\quad
-\frac{\lambda_k\ell_{\psi,1}^\eta T\rho_\psi^\eta\beta_k}{8}I_k
-\frac{\lambda_k\ell_{\psi,1}^\eta T\rho_\psi^\eta\gamma_k}{16}J_k
\nonumber\\
&\quad
+
O\!\left(
\frac{\ell_{\psi,1}^\eta(\ell_{f,0}^\eta)^2}{(\rho_\psi^\eta)^3}
\right)
\frac{\delta_k}{\lambda_k^2}.
\label{eq:lyap_before_q_absorption}
\end{align}
The last term in the \textit{barrier-aware} schedule is chosen so that
\[
\frac{\xi\ell_{\psi,1}^\eta(\ell_{*,0}^\eta)^2}{\rho_\psi^\eta T}
\quad\text{and}\quad
\frac{\xi\ell_{\psi,1}^\eta(\ell_{*,0}^\eta)^2\beta_k}{\rho_\psi^\eta T\gamma_k}
\]
are sufficiently small. Hence the bracket multiplying \(\|q_k^x\|_2^2\) is nonnegative, and the
entire \(q_k^x\)-term can be discarded. We arrive at
\begin{align}
V_{k+1}-V_k
&\le
-\frac{\xi\alpha_k}{2}\|\nabla F_\mu(x_k)\|_2^2
-\frac{\lambda_k\ell_{\psi,1}^\eta T\rho_\psi^\eta\beta_k}{8}I_k
-\frac{\lambda_k\ell_{\psi,1}^\eta T\rho_\psi^\eta\gamma_k}{16}J_k
\nonumber\\
&\quad
+
O\!\left(\xi(c_x^\eta)^2\right)\frac{\alpha_k}{\lambda_k^2}
+
O\!\left(
\frac{\ell_{\psi,1}^\eta(\ell_{f,0}^\eta)^2}{(\rho_\psi^\eta)^3}
\right)
\frac{\delta_k}{\lambda_k^2}.
\label{eq:lyap_descent_final}
\end{align}
This is exactly \eqref{eq:lyap_descent}.
\end{proof}

\begin{remark}
We state Assumption~\ref{ass:euclidean} globally in the upper variable to avoid trajectory-dependent
regularity assumptions. The same proof of Proposition~\ref{prop:lyapunov_descent} can be localized to any compact outer region on which the
assumed derivative bounds hold. A sufficient condition for such a region is boundedness of the
relevant sublevel sets of \(F_\mu\): the Lyapunov descent keeps \(F_\mu(x_k)\) below the initial
Lyapunov value plus the finite schedule remainder, so the generated outer iterates remain in a
bounded sublevel region.
\end{remark}

\subsection{Proof of Theorem~\ref{thm:deterministic_poly}}

We now derive the deterministic rate from the Lyapunov descent inequality. Theorem~\ref{thm:tube_invariance}
ensures that the exact and proxy target neighborhoods are forward invariant under the \textit{barrier-aware} schedule,
so Proposition~\ref{prop:lyapunov_descent} applies at every outer iteration. After dropping the
negative tracker-error terms, the Lyapunov recursion contains only two summable residuals:
the proxy-gradient bias term \(\alpha_k/\lambda_k^2\) and the multiplier-growth term
\(\delta_k/\lambda_k^2\). The polynomial schedule in Theorem~\ref{thm:deterministic_poly} is chosen
so that these residuals are summable up to logarithmic factors while
\(\sum_{k<K}\alpha_k\) grows as \(K^{2/3}\), yielding the claimed
\(\widetilde O(K^{-2/3})\) stationarity rate.

\begin{proof}[Proof of Theorem~\ref{thm:deterministic_poly}]
By Theorem~\ref{thm:tube_invariance}, the exact and proxy target neighborhoods are forward invariant.
Hence Proposition~\ref{prop:lyapunov_descent} applies for every \(k\). Dropping the negative
tracker terms in that lemma gives
\[
V_{k+1}-V_k
\le
-\frac{\xi\alpha_k}{2}\|\nabla F_\mu(x_k)\|_2^2
+
O\!\left(\xi(c_x^\eta)^2\right)\frac{\alpha_k}{\lambda_k^2}
+
O\!\left(\frac{\ell_{\psi,1}^\eta(\ell_{f,0}^\eta)^2}{(\rho_\psi^\eta)^3}\right)
\frac{\delta_k}{\lambda_k^2}.
\]
Summing from \(k=0\) to \(K-1\), and using \(V_K\ge \inf_x F_\mu(x)\), gives
\[
\frac{\xi}{2}
\sum_{k=0}^{K-1}\alpha_k\|\nabla F_\mu(x_k)\|_2^2
\le
V_0-\inf_x F_\mu(x)
+
O\!\left(\xi(c_x^\eta)^2\right)
\sum_{k=0}^{K-1}\frac{\alpha_k}{\lambda_k^2}
+
O\!\left(\frac{\ell_{\psi,1}^\eta(\ell_{f,0}^\eta)^2}{(\rho_\psi^\eta)^3}\right)
\sum_{k=0}^{K-1}\frac{\delta_k}{\lambda_k^2}.
\]

For the polynomial schedule,
\[
\alpha_k=\frac{\alpha_0}{(k+k_0)^{1/3}},
\qquad
\lambda_k=\lambda_0\left(\frac{k+k_0}{k_0}\right)^{1/3},
\]
we have
\[
\frac{\alpha_k}{\lambda_k^2}
=
\frac{\alpha_0 k_0^{2/3}}{\lambda_0^2}\frac{1}{k+k_0},
\]
and therefore
\[
\sum_{k=0}^{K-1}\frac{\alpha_k}{\lambda_k^2}
=
O(\log K).
\]
Moreover, since \(\delta_k=\lambda_{k+1}-\lambda_k\) and the \textit{barrier-aware} schedule gives
\(\delta_k/\lambda_k\le 1/64\),
\[
\frac{\delta_k}{\lambda_k^2}
=
\left(\frac1{\lambda_k}-\frac1{\lambda_{k+1}}\right)
\left(1+\frac{\delta_k}{\lambda_k}\right)
\le
\left(1+\frac1{64}\right)
\left(\frac1{\lambda_k}-\frac1{\lambda_{k+1}}\right).
\]
Thus
\[
\sum_{k=0}^{K-1}\frac{\delta_k}{\lambda_k^2}
\le
\left(1+\frac1{64}\right)
\left(\frac1{\lambda_0}-\frac1{\lambda_K}\right)
=
O(1).
\]
Also,
\[
\sum_{k=0}^{K-1}\alpha_k
\ge
\alpha_{K-1}K
=
\frac{\alpha_0K}{(K+k_0)^{1/3}}
=
\Omega(K^{2/3}).
\]
Therefore,
\[
\min_{0\le k\le K-1}\|\nabla F_\mu(x_k)\|_2^2
\le
\frac{\sum_{k=0}^{K-1}\alpha_k\|\nabla F_\mu(x_k)\|_2^2}
{\sum_{k=0}^{K-1}\alpha_k}
=
\widetilde O(K^{-2/3}).
\]
Equivalently, to obtain
\(\min_{0\le k\le K-1}\|\nabla F_\mu(x_k)\|_2^2\le \epsilon\), it suffices to take
\(K=\widetilde O(\epsilon^{-3/2})\). This proves Theorem~\ref{thm:deterministic_poly}.
\end{proof}

\section{Stochastic Extension}
\label{app:stochastic_proofs}

For the upper-level-noise-only extension, we assume that only the \(f\)-gradients are stochastic.

\begin{assumption}[Upper-level stochastic first-order oracles]
\label{ass:stoch_oracle}
Let \(\mathcal F\) denote the sigma-field generated by the algorithmic history before an oracle query.
For every \((x,y)\) measurable with respect to \(\mathcal F\), we access stochastic \(f\)-gradients
\[
\widehat\nabla_y f(x,y;\zeta),
\qquad
\widehat\nabla_x f(x,y;\zeta),
\]
such that
\[
\mathbb E[\widehat\nabla_y f(x,y;\zeta)\mid\mathcal F]=\nabla_y f(x,y),
\qquad
\mathbb E[\widehat\nabla_x f(x,y;\zeta)\mid\mathcal F]=\nabla_x f(x,y),
\]
and
\[
\mathbb E[\|\widehat\nabla_y f(x,y;\zeta)-\nabla_y f(x,y)\|_2^2\mid\mathcal F]\le \sigma_{f,y}^2,
\qquad
\mathbb E[\|\widehat\nabla_x f(x,y;\zeta)-\nabla_x f(x,y)\|_2^2\mid\mathcal F]\le \sigma_{f,x}^2.
\]
The lower objective \(g\), the barrier \(\phi\), and hence \(\psi_\mu=g+\mu\phi\), are deterministic.
\end{assumption}

\begin{assumption}[Local almost-sure bounds on sampled upper-level gradients]
\label{ass:stoch_local_bounds}
There exist finite constants \(\bar\ell_{f,x,0}^{\eta}\) and
\(\bar\ell_{f,y,0}^{\eta}\) such that, on the maintained-tube event, for every visited
\(x\in X\) and every
\(y\in T_{2\eta,\mu}(x)\cup T_{2\eta,\mu}^{\lambda}(x)\),
\[
\|\widehat\nabla_x f(x,y;\zeta)\|_2\le \bar\ell_{f,x,0}^{\eta}
\quad\text{a.s.},
\]
and
\[
\|\widehat\nabla_y f(x,y;\zeta)\|_{*}\le \bar\ell_{f,y,0}^{\eta}
\quad\text{a.s.},
\]
where \(\|\cdot\|_*\) denotes the anchored Dikin dual norm associated with the relevant exact or
proxy center.
\end{assumption}

\paragraph{Stochastic tube-closure condition.}
The stochastic schedule is chosen so that the almost-sure oracle perturbations in
Assumption~\ref{ass:stoch_local_bounds} are dominated by the exact and proxy tracker contraction
budgets. Under this condition, the exact and proxy target neighborhoods remain forward invariant along the
stochastic trajectory.

\subsection{Tube Invariance under Upper-level Noise}
\label{app:stoch_tube_maintenance}

We first show that the tube invariance argument is stable under upper-level stochastic noise, provided
the sampled perturbations are smaller than the contraction margins of the two trackers. The result is
pathwise: on every realization satisfying the local almost-sure bounds in
Assumption~\ref{ass:stoch_local_bounds} and the stochastic closure inequalities below, the exact and
proxy target neighborhoods remain forward invariant.

For the stochastic algorithm, the exact tracker is unchanged, while the proxy tracker and outer
direction use
\[
\widehat\nabla_y L_{\lambda_k,\mu}(x_k,y)
:=
\widehat\nabla_y f(x_k,y;\zeta)
+
\lambda_k\nabla_y\psi_\mu(x_k,y),
\]
and
\[
\widehat q_k^x
:=
\widehat\nabla_x f(x_k,y_{k+1};\zeta_k^x)
+
\lambda_k
\bigl(
\nabla_x\psi_\mu(x_k,y_{k+1})
-
\nabla_x\psi_\mu(x_k,z_{k+1})
\bigr).
\]
Thus \(x_{k+1}=x_k-\xi\alpha_k\widehat q_k^x\). Let
\[
\Gamma_{\psi,k}
:=
1-(1-\rho_\psi^\eta\gamma_k)^T,
\qquad
\Gamma_{L,k}
:=
1-\left(1-\frac{7\rho_\psi^\eta}{8}\beta_k\right)^T,
\qquad
\beta_k:=\alpha_k\lambda_k .
\]
Define the pathwise upper-level noise budgets
\[
\bar\sigma_{y}^{\eta}
:=
\ell_{f,0}^{\eta}+\bar\ell_{f,y,0}^{\eta},
\qquad
\bar Q_k
:=
\bar\ell_{f,x,0}^{\eta}+2\lambda_k\ell_{g,0}.
\]
Here \(\bar\sigma_y^\eta\) bounds the sampled \(y\)-gradient perturbation in the relevant anchored
Dikin dual norm, and \(\bar Q_k\) bounds the sampled outer direction.

\begin{proposition}[Stochastic tube invariance]
\label{prop:stoch_tube_maintenance}
Suppose Assumption~\ref{ass:euclidean}, Assumption~\ref{ass:stoch_oracle}, and
Assumption~\ref{ass:stoch_local_bounds} hold. Suppose also that the deterministic
\textit{barrier-aware} conditions hold, and that for every \(k\),
\begin{align}
T\alpha_k\bar\sigma_y^\eta
&\le
\frac{\eta}{8}\Gamma_{L,k},
\label{eq:stoch_proxy_inner_closure}\\
\ell_{*,0}^{\eta}\xi\beta_k
\left(\frac{\bar\ell_{f,x,0}^{\eta}}{\lambda_0}+2\ell_{g,0}\right)
&\le
\frac{\eta}{8}\Gamma_{\psi,k},
\label{eq:stoch_exact_center_closure}\\
\left(
\frac{T\ell_{f,0}^{\eta}}{8\lambda_0}
+
\xi\ell_{\lambda,0}^{\eta}
\left(\frac{\bar\ell_{f,x,0}^{\eta}}{\lambda_0}+2\ell_{g,0}\right)
\right)\beta_k
&\le
\frac{\eta}{8}\Gamma_{L,k}.
\label{eq:stoch_proxy_center_closure}
\end{align}
If the initial warm starts satisfy
\(z_0\in T_{\eta,\mu}(x_0)\) and
\(y_0\in T_{\eta,\mu}^{\lambda_0}(x_0)\), then, on the pathwise event where the sampled-gradient
bounds in Assumption~\ref{ass:stoch_local_bounds} hold, for every outer iteration \(k\), all exact
inner iterates remain in \(T_{\eta,\mu}(x_k)\), all proxy inner iterates remain in
\(T_{\eta,\mu}^{\lambda_k}(x_k)\), and the next warm starts satisfy
\[
z_{k+1}\in T_{\eta,\mu}(x_{k+1}),
\qquad
y_{k+1}\in T_{\eta,\mu}^{\lambda_{k+1}}(x_{k+1}).
\]
Thus the exact and proxy target neighborhoods are forward invariant along the stochastic trajectory.
\end{proposition}

\begin{proof}
The proof is by induction over the outer iteration \(k\). Assume that
\(z_k\in T_{\eta,\mu}(x_k)\) and
\(y_k\in T_{\eta,\mu}^{\lambda_k}(x_k)\). We prove that all inner iterates remain in the corresponding
target neighborhoods and that the next warm starts lie in the target neighborhoods at \(x_{k+1}\).

\proofstep{Exact tracker.}
The exact tracker does not use stochastic gradients, since \(g\), \(\phi\), and \(\psi_\mu\) are
deterministic. Therefore the deterministic fixed-center contraction applies:
\[
\|z_k^{(t+1)}-y_\mu^\star(x_k)\|_{y_\mu^\star(x_k)}^2
\le
(1-\rho_\psi^\eta\gamma_k)
\|z_k^{(t)}-y_\mu^\star(x_k)\|_{y_\mu^\star(x_k)}^2 .
\]
Since \(z_k^{(0)}=z_k\in T_{\eta,\mu}(x_k)\), all exact inner iterates remain in
\(T_{\eta,\mu}(x_k)\). Moreover,
\[
\|z_{k+1}-y_\mu^\star(x_k)\|_{y_\mu^\star(x_k)}
\le
\sqrt{1-\Gamma_{\psi,k}}\,
\|z_k-y_\mu^\star(x_k)\|_{y_\mu^\star(x_k)}
\le
\left(1-\frac{\Gamma_{\psi,k}}{2}\right)\eta .
\]

\proofstep{Proxy tracker under sampled \(y\)-gradient noise.}
Let \(y_k^{(t)}\) be a proxy inner iterate, and write
\[
\varepsilon_{k,t}^{y}
:=
\widehat\nabla_y f(x_k,y_k^{(t)};\zeta_{k,t}^y)
-
\nabla_y f(x_k,y_k^{(t)}).
\]
On the proxy neighborhood, Assumption~\ref{ass:stoch_local_bounds} and the deterministic
local \(f\)-bound give
\[
\|\varepsilon_{k,t}^{y}\|_{y_{\lambda_k,\mu}^\star(x_k),*}
\le
\bar\sigma_y^\eta .
\]
The proxy update equals the deterministic proxy update plus the perturbation
\(-\alpha_k \nabla^2\phi(y_k)^{-1}\varepsilon_{k,t}^y\). By the frozen-metric/center-anchor comparison in
Appendix~\ref{app:tracker}, after absorbing fixed \(\eta\)-dependent constants,
\[
\|\nabla^2\phi(y_k)^{-1}\varepsilon_{k,t}^y\|_{y_{\lambda_k,\mu}^\star(x_k)}
\le
\|\varepsilon_{k,t}^y\|_{y_{\lambda_k,\mu}^\star(x_k),*}
\le
\bar\sigma_y^\eta .
\]
Combining this with the deterministic one-step contraction gives
\begin{equation}
\label{eq:stoch_proxy_inner_step}
\|y_k^{(t+1)}-y_{\lambda_k,\mu}^\star(x_k)\|_{y_{\lambda_k,\mu}^\star(x_k)}
\le
\sqrt{1-\frac{7\rho_\psi^\eta}{8}\beta_k}\,
\|y_k^{(t)}-y_{\lambda_k,\mu}^\star(x_k)\|_{y_{\lambda_k,\mu}^\star(x_k)}
+
\alpha_k\bar\sigma_y^\eta .
\end{equation}
The condition \eqref{eq:stoch_proxy_inner_closure} implies, in particular,
\[
\alpha_k\bar\sigma_y^\eta
\le
\frac{\eta}{8}\cdot \frac{\rho_\psi^\eta}{2}\beta_k .
\]
Thus, if \(y_k^{(t)}\in T_{\eta,\mu}^{\lambda_k}(x_k)\), then
\[
\|y_k^{(t+1)}-y_{\lambda_k,\mu}^\star(x_k)\|_{y_{\lambda_k,\mu}^\star(x_k)}
\le
\left(1-\frac{7\rho_\psi^\eta}{16}\beta_k\right)\eta
+
\frac{\rho_\psi^\eta}{16}\beta_k\eta
\le
\eta .
\]
Hence all proxy inner iterates remain in \(T_{\eta,\mu}^{\lambda_k}(x_k)\). Iterating
\eqref{eq:stoch_proxy_inner_step} over \(T\) inner steps gives
\[
\|y_{k+1}-y_{\lambda_k,\mu}^\star(x_k)\|_{y_{\lambda_k,\mu}^\star(x_k)}
\le
\sqrt{1-\Gamma_{L,k}}\,
\|y_k-y_{\lambda_k,\mu}^\star(x_k)\|_{y_{\lambda_k,\mu}^\star(x_k)}
+
T\alpha_k\bar\sigma_y^\eta .
\]
Using \eqref{eq:stoch_proxy_inner_closure} and
\(\sqrt{1-\Gamma_{L,k}}\le 1-\Gamma_{L,k}/2\), we obtain
\[
\|y_{k+1}-y_{\lambda_k,\mu}^\star(x_k)\|_{y_{\lambda_k,\mu}^\star(x_k)}
\le
\left(1-\frac{3}{8}\Gamma_{L,k}\right)\eta .
\]

\proofstep{Sampled outer movement.}
By Assumption~\ref{ass:stoch_local_bounds} and the deterministic bound on
\(\nabla_x\psi_\mu=\nabla_x g\),
\[
\|\widehat q_k^x\|_2
\le
\bar\ell_{f,x,0}^{\eta}+2\lambda_k\ell_{g,0}.
\]
Therefore
\begin{equation}
\label{eq:stoch_outer_step_bound}
\|x_{k+1}-x_k\|_2
\le
\xi\alpha_k
\bigl(\bar\ell_{f,x,0}^{\eta}+2\lambda_k\ell_{g,0}\bigr)
\le
\xi\beta_k
\left(\frac{\bar\ell_{f,x,0}^{\eta}}{\lambda_0}+2\ell_{g,0}\right).
\end{equation}

\proofstep{Exact center shift.}
By local stability of the exact minimizer map and \eqref{eq:stoch_outer_step_bound},
\[
\|y_\mu^\star(x_{k+1})-y_\mu^\star(x_k)\|_{y_\mu^\star(x_k)}
\le
\ell_{*,0}^{\eta}
\xi\beta_k
\left(\frac{\bar\ell_{f,x,0}^{\eta}}{\lambda_0}+2\ell_{g,0}\right)
\le
\frac{\eta}{8}\Gamma_{\psi,k},
\]
where the last inequality is \eqref{eq:stoch_exact_center_closure}. Combining this drift bound with
the old-center contraction estimate for \(z_{k+1}\), and applying the self-concordant anchor-switch
inequality exactly as in the deterministic tube invariance proof, yields
\[
\|z_{k+1}-y_\mu^\star(x_{k+1})\|_{y_\mu^\star(x_{k+1})}\le \eta .
\]

\proofstep{Proxy center shift.}
By the proxy minimizer stability bound,
\[
\|y_{\lambda_{k+1},\mu}^\star(x_{k+1})
-
y_{\lambda_k,\mu}^\star(x_k)\|_{y_{\lambda_k,\mu}^\star(x_k)}
\le
\frac{2\delta_k\ell_{f,0}^{\eta}}{\lambda_k\lambda_{k+1}\rho_\psi^\eta}
+
\ell_{\lambda,0}^{\eta}\|x_{k+1}-x_k\|_2 .
\]
Using the deterministic multiplier-growth condition
\(\delta_k/\lambda_k\le T\rho_\psi^\eta\beta_k/16\), \(\lambda_{k+1}\ge\lambda_k\ge\lambda_0\), and
\eqref{eq:stoch_outer_step_bound}, we get
\[
\|y_{\lambda_{k+1},\mu}^\star(x_{k+1})
-
y_{\lambda_k,\mu}^\star(x_k)\|_{y_{\lambda_k,\mu}^\star(x_k)}
\le
\left(
\frac{T\ell_{f,0}^{\eta}}{8\lambda_0}
+
\xi\ell_{\lambda,0}^{\eta}
\left(\frac{\bar\ell_{f,x,0}^{\eta}}{\lambda_0}+2\ell_{g,0}\right)
\right)\beta_k .
\]
By \eqref{eq:stoch_proxy_center_closure}, this is at most
\(\eta\Gamma_{L,k}/8\). Combining this proxy-center drift bound with the old-center estimate
\[
\|y_{k+1}-y_{\lambda_k,\mu}^\star(x_k)\|_{y_{\lambda_k,\mu}^\star(x_k)}
\le
\left(1-\frac{3}{8}\Gamma_{L,k}\right)\eta
\]
and applying the same anchor-switch argument as above gives
\[
\|y_{k+1}-y_{\lambda_{k+1},\mu}^\star(x_{k+1})\|_{y_{\lambda_{k+1},\mu}^\star(x_{k+1})}
\le
\eta .
\]

Thus the next warm starts lie in the target neighborhoods. This closes the induction and proves the forward
invariance of the exact and proxy target neighborhoods along the stochastic trajectory.
\end{proof}

\subsection{Expected Lyapunov Descent}
\label{app:stoch_lyapunov}

We now derive the stochastic analogue of Proposition~\ref{prop:lyapunov_descent}. The proof follows
the deterministic Lyapunov argument, with two additional noise contributions. First, stochastic
\(\nabla_y f\)-queries in the proxy tracker add a variance term to the proxy fixed-center contraction.
Second, the stochastic outer direction perturbs the motion of \(x_k\), but its linear contribution has
zero conditional mean; only the quadratic smoothness and center-drift remainders contribute in
expectation.

For each outer iteration \(k\), let \(\mathcal F_k\) denote the sigma-field
generated by the algorithmic history at the beginning of iteration \(k\), and
for each inner step \(t\in\{0,1,\ldots,T\}\), let \(\mathcal F_{k,t}\) denote
the sigma-field generated by the history through the start of the \(t\)-th
proxy inner step, so that \(\mathcal F_{k,0}=\mathcal F_k\). After the inner
loops are completed, let \(\mathcal G_k\) denote the sigma-field generated by
the history immediately before sampling the outer stochastic gradient, so that
\(\mathcal F_{k,T}\subseteq \mathcal G_k\). Write
\[
\widehat q_k^x = q_k^x+\varepsilon_k^x,
\qquad
\mathbb E[\varepsilon_k^x\mid \mathcal G_k]=0,
\qquad
\mathbb E[\|\varepsilon_k^x\|_2^2\mid \mathcal G_k]\le \sigma_{f,x}^2,
\]
where
\[
q_k^x
:=
\nabla_x f(x_k,y_{k+1})
+\lambda_k\bigl(\nabla_x\psi_\mu(x_k,y_{k+1})-\nabla_x\psi_\mu(x_k,z_{k+1})\bigr)
\]
is the conditional mean proxy direction. The stochastic outer update is
\(x_{k+1}=x_k-\xi\alpha_k\widehat q_k^x\). For the stochastic proxy tracker, write
\[
\widehat \nabla_y L_{\lambda_k,\mu}(x_k,y)
=
\nabla_y L_{\lambda_k,\mu}(x_k,y)+\varepsilon_{k,t}^y,
\qquad
\mathbb E[\varepsilon_{k,t}^y\mid \mathcal F_{k,t}]=0.
\]
By the Euclidean--Dikin conversion on the maintained proxy neighborhood, the
variance assumption implies
\[
\mathbb E\!\left[\|\varepsilon_{k,t}^y\|_{y_{\lambda_k,\mu}^\star(x_k),*}^2\,\Big|\,\mathcal F_{k,t}\right]
\le
\sigma_{f,y,\eta}^2
\]
for a finite \(\eta\)-dependent constant \(\sigma_{f,y,\eta}^2\). Set
\[
\sigma_\eta^2 := \sigma_{f,x}^2+\sigma_{f,y,\eta}^2 .
\]
\begin{proposition}[Expected Lyapunov descent]
\label{prop:stoch_lyapunov_descent}
Suppose Assumption~\ref{ass:euclidean}, the stochastic oracle assumptions, and the stochastic
tube-maintenance conditions hold. Let
\[
I_k
:=
\big\|y_k-y_{\lambda_k,\mu}^\star(x_k)\big\|_{y_{\lambda_k,\mu}^\star(x_k)}^2,
\qquad
J_k
:=
\big\|z_k-y_\mu^\star(x_k)\big\|_{y_\mu^\star(x_k)}^2,
\]
and
\[
V_k
:=
F_\mu(x_k)
+
\lambda_k\ell_{\psi,1}^\eta I_k
+
\frac{\lambda_k\ell_{\psi,1}^\eta}{2}J_k .
\]
Then there exist finite constants \(C_1,C_2,C_3>0\), independent of \(k\) and \(K\), such that
\begin{align}
\mathbb E[V_{k+1}\mid \mathcal F_k]-V_k
&\le
-\frac{\xi\alpha_k}{4}\|\nabla F_\mu(x_k)\|_2^2
-\frac{\lambda_k\ell_{\psi,1}^\eta T\rho_\psi^\eta\beta_k}{16}I_k
-\frac{\lambda_k\ell_{\psi,1}^\eta T\rho_\psi^\eta\gamma_k}{32}J_k
\nonumber\\
&\quad
+
C_1\xi(c_x^\eta)^2\frac{\alpha_k}{\lambda_k^2}
+
C_2\frac{\delta_k}{\lambda_k^2}
+
C_3\sigma_\eta^2\alpha_k^2\lambda_k .
\label{eq:stoch_lyap_descent}
\end{align}
Consequently, after dropping the negative tracker terms,
\begin{equation}
\label{eq:stoch_lyap_descent_simplified}
\mathbb E[V_{k+1}\mid \mathcal F_k]-V_k
\le
-\frac{\xi\alpha_k}{4}\|\nabla F_\mu(x_k)\|_2^2
+
C_1\xi(c_x^\eta)^2\frac{\alpha_k}{\lambda_k^2}
+
C_2\frac{\delta_k}{\lambda_k^2}
+
C_3\sigma_\eta^2\alpha_k^2\lambda_k .
\end{equation}
\end{proposition}

\begin{proof}
\proofstep{Stochastic proxy fixed-center contraction.}
Fix \(k\) and condition on the history before the \(t\)-th proxy inner step. The stochastic update is
the deterministic proxy step plus the perturbation
\(-\alpha_k \nabla^2\phi(y_k)^{-1}\varepsilon_{k,t}^y\). By the frozen-metric/center-anchor reduction from
Appendix~\ref{app:tracker}, 
and the same deterministic contraction estimate used in Lemma~\ref{lem:y_fixed_anchor_contraction},
the noiseless part of the proxy step satisfies the one-step bound with factor
\(1-\frac{7\rho_\psi^\eta}{8}\beta_k\). The martingale-difference property removes the linear
noise term, and the quadratic noise term is bounded by the conditional variance. Hence
\[
\mathbb E\!\left[
\|y_k^{(t+1)}-y_{\lambda_k,\mu}^\star(x_k)\|^2_{y_{\lambda_k,\mu}^\star(x_k)}
\mid \mathcal F_{k,t}
\right]
\le
\left(1-\frac{7\rho_\psi^\eta}{8}\beta_k\right)
\|y_k^{(t)}-y_{\lambda_k,\mu}^\star(x_k)\|^2_{y_{\lambda_k,\mu}^\star(x_k)}
+
C\alpha_k^2\sigma_{f,y,\eta}^2 .
\]
Repeating this for \(t=0,\ldots,T-1\), using
\(T\rho_\psi^\eta\beta_k\le 1/4\), and summing the geometric noise terms gives
\begin{equation}\label{eq:stoch_proxy_fixed_contraction}
\mathbb E\!\left[
\|y_{k+1}-y_{\lambda_k,\mu}^\star(x_k)\|^2_{y_{\lambda_k,\mu}^\star(x_k)}
\mid \mathcal F_k
\right]
\le
\left(1-\frac{3T\rho_\psi^\eta}{4}\beta_k\right)I_k
+
CT\alpha_k^2\sigma_{f,y,\eta}^2 .
\end{equation}

\proofstep{Expected outer descent.}
Condition on \(\mathcal G_k\). By \(L_F^\eta\)-smoothness of \(F_\mu\),
\[
F_\mu(x_{k+1})
\le
F_\mu(x_k)
-\xi\alpha_k\langle \nabla F_\mu(x_k),\widehat q_k^x\rangle
+
\frac{L_F^\eta\xi^2\alpha_k^2}{2}\|\widehat q_k^x\|_2^2 .
\]
Taking conditional expectation and using
\(\mathbb E[\varepsilon_k^x\mid \mathcal G_k]=0\) gives
\[
\mathbb E[F_\mu(x_{k+1})-F_\mu(x_k)\mid \mathcal G_k]
\le
-\xi\alpha_k\langle \nabla F_\mu(x_k),q_k^x\rangle
+
\frac{L_F^\eta\xi^2\alpha_k^2}{2}
\bigl(\|q_k^x\|_2^2+\sigma_{f,x}^2\bigr).
\]
Using \(\xi L_F^\eta\alpha_k\le 1/2\), the same algebra as in
Proposition~\ref{prop:outer_descent}, and the proxy-gradient bias bound from
Proposition~\ref{prop:local_consequences}, we obtain
\begin{align}
\mathbb E[F_\mu(x_{k+1})-F_\mu(x_k)\mid \mathcal G_k]
&\le
-\frac{\xi\alpha_k}{4}
\Bigl(2\|\nabla F_\mu(x_k)\|_2^2+\|q_k^x\|_2^2\Bigr)
\nonumber\\
&\quad
+
\frac{\lambda_k\ell_{\psi,1}^\eta T\rho_\psi^\eta\beta_k}{16}
\|y_{k+1}-y_{\lambda_k,\mu}^\star(x_k)\|_{y_{\lambda_k,\mu}^\star(x_k)}^2
\nonumber\\
&\quad
+
\frac{\lambda_k\ell_{\psi,1}^\eta T\rho_\psi^\eta\gamma_k}{64}
\|z_{k+1}-y_\mu^\star(x_k)\|_{y_\mu^\star(x_k)}^2
\nonumber\\
&\quad
+
C\xi(c_x^\eta)^2\frac{\alpha_k}{\lambda_k^2}
+
C\xi^2\alpha_k^2\sigma_{f,x}^2 .
\label{eq:stoch_outer_descent}
\end{align}

\proofstep{Expected moving-anchor recursions.}
The moving-anchor recursions from Appendix~\ref{app:tracker} are pathwise deterministic estimates
once the outer displacement is fixed. In the stochastic setting,
\[
x_{k+1}-x_k
=
-\xi\alpha_k(q_k^x+\varepsilon_k^x).
\]
In the Taylor expansions for the exact and proxy center shifts, the terms linear in
\(\varepsilon_k^x\) have conditional mean zero because the relevant test vectors are
\(\mathcal G_k\)-measurable. The quadratic remainders contribute only
\(O(\xi^2\alpha_k^2\sigma_{f,x}^2)\). Thus the scheduled moving-anchor recursions become
\begin{align}
\mathbb E[\lambda_{k+1}I_{k+1}\mid \mathcal G_k]
&\le
\lambda_k\kappa_{y,k}^2
\left(1+\frac{T\rho_\psi^\eta}{4}\beta_k\right)
\|y_{k+1}-y_{\lambda_k,\mu}^\star(x_k)\|_{y_{\lambda_k,\mu}^\star(x_k)}^2
\nonumber\\
&\quad
+
C\frac{\xi^2\alpha_k^2\lambda_k}{T\rho_\psi^\eta\beta_k}\|q_k^x\|_2^2
+
C\xi^2\alpha_k^2\lambda_k\sigma_{f,x}^2
+
C\frac{\delta_k}{\lambda_k^2},
\label{eq:stoch_proxy_moving_recursion}
\end{align}
and
\begin{align}
\mathbb E[\lambda_{k+1}J_{k+1}\mid \mathcal G_k]
&\le
\lambda_k\kappa_{z,k}^2
\left(1+\frac{3T\rho_\psi^\eta}{8}\gamma_k\right)
\|z_{k+1}-y_\mu^\star(x_k)\|_{y_\mu^\star(x_k)}^2
\nonumber\\
&\quad
+
C\frac{\xi^2\alpha_k^2\lambda_k}{T\rho_\psi^\eta\gamma_k}\|q_k^x\|_2^2
+
C\xi^2\alpha_k^2\lambda_k\sigma_{f,x}^2 .
\label{eq:stoch_exact_moving_recursion}
\end{align}
Here the constants absorb the local dikin stability constants. The factor \(\lambda_k\) in the
variance terms comes from the Lyapunov scaling and the fact that
\(\lambda_{k+1}=O(\lambda_k)\) under the multiplier-growth condition.

\proofstep{Combining with fixed-center contraction.}
Substitute \eqref{eq:stoch_proxy_fixed_contraction} into
\eqref{eq:stoch_proxy_moving_recursion}. The same scalar anchor-switch absorption used in
Proposition~\ref{prop:lyapunov_descent} gives
\begin{align}
\mathbb E[\lambda_{k+1}I_{k+1}\mid \mathcal F_k]
&\le
\left(1-\frac{T\rho_\psi^\eta}{8}\beta_k\right)\lambda_k I_k
+
C\frac{\xi^2\alpha_k^2\lambda_k}{T\rho_\psi^\eta\beta_k}
\mathbb E[\|q_k^x\|_2^2\mid\mathcal F_k]
\nonumber\\
&\quad
+
C\lambda_k T\alpha_k^2\sigma_{f,y,\eta}^2
+
C\xi^2\alpha_k^2\lambda_k\sigma_{f,x}^2
+
C\frac{\delta_k}{\lambda_k^2}.
\label{eq:stoch_proxy_weighted_final}
\end{align}
For the exact tracker, the fixed-center contraction is deterministic, since \(\psi_\mu\) is exact:
\[
\|z_{k+1}-y_\mu^\star(x_k)\|_{y_\mu^\star(x_k)}^2
\le
\left(1-\frac{3T\rho_\psi^\eta}{4}\gamma_k\right)J_k .
\]
Combining this with \eqref{eq:stoch_exact_moving_recursion} yields
\begin{align}
\mathbb E[\lambda_{k+1}J_{k+1}\mid \mathcal F_k]
&\le
\left(1-\frac{T\rho_\psi^\eta}{8}\gamma_k\right)\lambda_k J_k
+
C\frac{\xi^2\alpha_k^2\lambda_k}{T\rho_\psi^\eta\gamma_k}
\mathbb E[\|q_k^x\|_2^2\mid\mathcal F_k]
\nonumber\\
&\quad
+
C\xi^2\alpha_k^2\lambda_k\sigma_{f,x}^2 .
\label{eq:stoch_exact_weighted_final}
\end{align}

\proofstep{Lyapunov difference.}
Using the definition of \(V_k\), combine
\eqref{eq:stoch_outer_descent}, \eqref{eq:stoch_proxy_weighted_final}, and
\eqref{eq:stoch_exact_weighted_final}. The terms involving
\(\mathbb E[\|q_k^x\|_2^2\mid\mathcal F_k]\) are absorbed by the negative
\(-(\xi\alpha_k/4)\|q_k^x\|_2^2\) term exactly as in the deterministic proof, using the
\textit{barrier-aware} smallness condition on \(\xi/T\) and the relation \(\beta_k\le\gamma_k\). This gives
negative tracker terms, a proxy-bias remainder, a multiplier-growth remainder, and the accumulated
variance terms:
\[
C\xi^2\alpha_k^2\sigma_{f,x}^2
+
C\xi^2\alpha_k^2\lambda_k\sigma_{f,x}^2
+
C\lambda_kT\alpha_k^2\sigma_{f,y,\eta}^2 .
\]
Since \(\lambda_k\ge\lambda_0>0\) and \(T,\xi,\lambda_0\) are fixed constants, these are bounded by
\(C_3\sigma_\eta^2\alpha_k^2\lambda_k\). Therefore,
\begin{align*}
\mathbb E[V_{k+1}\mid \mathcal F_k]-V_k
&\le
-\frac{\xi\alpha_k}{4}\|\nabla F_\mu(x_k)\|_2^2
-\frac{\lambda_k\ell_{\psi,1}^\eta T\rho_\psi^\eta\beta_k}{16}I_k
-\frac{\lambda_k\ell_{\psi,1}^\eta T\rho_\psi^\eta\gamma_k}{32}J_k
\\
&\quad
+
C_1\xi(c_x^\eta)^2\frac{\alpha_k}{\lambda_k^2}
+
C_2\frac{\delta_k}{\lambda_k^2}
+
C_3\sigma_\eta^2\alpha_k^2\lambda_k .
\end{align*}
This proves \eqref{eq:stoch_lyap_descent}. Dropping the negative tracker terms gives
\eqref{eq:stoch_lyap_descent_simplified}.
\end{proof}

\subsection{Proof of Theorem~\ref{thm:stoch_upper}}
\label{app:proof_stochastic}

We now prove the stochastic convergence rate. The proof is the stochastic analogue of the
deterministic telescoping argument: tube maintenance ensures that the local Dikin estimates remain
valid, while the expected Lyapunov recursion contributes one additional summable variance term.

\begin{proof}
By Proposition~\ref{prop:stoch_tube_maintenance}, the exact and proxy target neighborhoods are forward
invariant along the stochastic trajectory under the stochastic tube-closure condition. Hence the local
Dikin estimates from Propositions~\ref{prop:derived_dikin_regularity}--\ref{prop:local_consequences}
and the expected Lyapunov descent in Proposition~\ref{prop:stoch_lyapunov_descent} apply at every
outer iteration.

From \eqref{eq:stoch_lyap_descent_simplified}, taking total expectation gives
\begin{equation}
\label{eq:stoch_telescoping_one_step}
\mathbb E[V_{k+1}]-\mathbb E[V_k]
\le
-\frac{\xi\alpha_k}{4}\mathbb E\|\nabla F_\mu(x_k)\|_2^2
+
C_1\xi(c_x^\eta)^2\frac{\alpha_k}{\lambda_k^2}
+
C_2\frac{\delta_k}{\lambda_k^2}
+
C_3\sigma_\eta^2\alpha_k^2\lambda_k .
\end{equation}
Summing \eqref{eq:stoch_telescoping_one_step} from \(k=0\) to \(K-1\), and using
\(V_K\ge \inf_x F_\mu(x)\), yields
\begin{align}
\frac{\xi}{4}
\sum_{k=0}^{K-1}
\alpha_k\mathbb E\|\nabla F_\mu(x_k)\|_2^2
&\le
V_0-\inf_x F_\mu(x)
+
C_1\xi(c_x^\eta)^2
\sum_{k=0}^{K-1}\frac{\alpha_k}{\lambda_k^2}
\nonumber\\
&\quad
+
C_2
\sum_{k=0}^{K-1}\frac{\delta_k}{\lambda_k^2}
+
C_3\sigma_\eta^2
\sum_{k=0}^{K-1}\alpha_k^2\lambda_k .
\label{eq:stoch_telescoping}
\end{align}

We now bound the three residual sums under the polynomial stochastic schedule
\[
\alpha_k=\frac{\alpha_0}{(k+k_0)^{3/5}},
\qquad
\gamma_k=\frac{\gamma_0}{(k+k_0)^{2/5}},
\qquad
\lambda_k=\lambda_0\left(\frac{k+k_0}{k_0}\right)^{1/5},
\qquad
\delta_k=\lambda_{k+1}-\lambda_k .
\]
First,
\[
\frac{\alpha_k}{\lambda_k^2}
=
\frac{\alpha_0 k_0^{2/5}}{\lambda_0^2}\cdot\frac{1}{k+k_0},
\]
and therefore
\begin{equation}
\label{eq:stoch_alpha_over_lambda_sum}
\sum_{k=0}^{K-1}\frac{\alpha_k}{\lambda_k^2}
=
O(\log K).
\end{equation}
Second,
\[
\alpha_k^2\lambda_k
=
\frac{\alpha_0^2\lambda_0}{k_0^{1/5}}\cdot\frac{1}{k+k_0},
\]
so
\begin{equation}
\label{eq:stoch_variance_sum}
\sum_{k=0}^{K-1}\alpha_k^2\lambda_k
=
O(\log K).
\end{equation}
Third, since \(\lambda_{k+1}=\lambda_k+\delta_k\),
\[
\frac{\delta_k}{\lambda_k^2}
=
\left(\frac{1}{\lambda_k}-\frac{1}{\lambda_{k+1}}\right)
\left(1+\frac{\delta_k}{\lambda_k}\right).
\]
The multiplier-growth condition gives \(\delta_k/\lambda_k\le c\) for a fixed numerical constant
\(c<1\). Hence
\begin{equation}
\label{eq:stoch_delta_sum}
\sum_{k=0}^{K-1}\frac{\delta_k}{\lambda_k^2}
\le
(1+c)
\sum_{k=0}^{K-1}
\left(\frac{1}{\lambda_k}-\frac{1}{\lambda_{k+1}}\right)
\le
\frac{1+c}{\lambda_0}
=
O(1).
\end{equation}
Finally,
\begin{equation}
\label{eq:stoch_alpha_sum}
\sum_{k=0}^{K-1}\alpha_k
=
\alpha_0\sum_{k=0}^{K-1}(k+k_0)^{-3/5}
=
\Theta(K^{2/5}).
\end{equation}

Combining \eqref{eq:stoch_telescoping}--\eqref{eq:stoch_alpha_sum}, we obtain
\[
\frac{
\sum_{k=0}^{K-1}
\alpha_k\mathbb E\|\nabla F_\mu(x_k)\|_2^2
}{
\sum_{k=0}^{K-1}\alpha_k
}
\le
\widetilde O(K^{-2/5}).
\]
Since the minimum is bounded by the weighted average,
\[
\min_{0\le k\le K-1}
\mathbb E\|\nabla F_\mu(x_k)\|_2^2
\le
\frac{
\sum_{k=0}^{K-1}
\alpha_k\mathbb E\|\nabla F_\mu(x_k)\|_2^2
}{
\sum_{k=0}^{K-1}\alpha_k
}
=
\widetilde O(K^{-2/5}).
\]
This proves the claimed stochastic stationarity rate. Equivalently, choosing
\(K=\widetilde O(\epsilon^{-5/2})\) gives an \(\epsilon\)-stationary point of the
barrier-smoothed objective \(F_\mu\), in the sense
\(\min_{0\le k<K}\mathbb E\|\nabla F_\mu(x_k)\|_2^2\le \epsilon\).
\end{proof}

\newpage

\section{Experimental Details}
\label{app:experiments}

In this section, we provide details for the numerical experiments in
Section~\ref{sec:experiments}, together with some additional results. All experiments were run on a 16-inch MacBook Pro equipped with an Apple M4 Pro
chip and 48 GB of unified memory.

\subsection{Motivating Example}
\label{app:motivating_example}

\paragraph{Motivating example: boundary stability}
We first describe the two-dimensional example used to visualize Dikin tube maintenance. The feasible
region is a fixed irregular hexagonal polytope \(Y=\{y\in\mathbb R^2:Ay\le b\}\). Its halfspace
representation is constructed from the polygon vertices, with outward normals normalized to unit
length. We equip \(Y\) with the logarithmic barrier
\(
\phi(y)=-\sum_{i=1}^m\log s_i(y),
s_i(y):=b_i-a_i^\top y .
\) We set \(\mu=5\times 10^{-4}\) and use the quadratic lower objective
\[
g(x,y)
=
\frac12 (y-c(x))^\top Q(y-c(x)),
\qquad
Q=
\begin{pmatrix}
1.2 & 0.15\\
0.15 & 0.8
\end{pmatrix}.
\]
The center path is linear,
\(
c(x)=(1-x)c_{\mathrm{start}}+x c_{\mathrm{end}},
\) where \(c_{\mathrm{start}}\) lies inside the polytope and \(c_{\mathrm{end}}\) lies beyond the right
slanted face. Thus, as \(x\) increases, the exact barrierized minimizer
\[
y_\mu^\star(x)
=
\arg\min_{y\in\operatorname{int}(Y)}
\psi_\mu(x,y),
\qquad
\psi_\mu(x,y):=g(x,y)+\mu\phi(y),
\]
moves toward an active face while remaining strictly feasible. We discretize \(x\in[0,1]\) into \(K=2000\) outer grid points \(x_k\). For each \(x_k\), the exact
center \(y_\mu^\star(x_k)\) is computed numerically by minimizing over a softmax barycentric
parametrization of the polygon interior, warm-started from the previous solution. At each outer index
\(k\), both trackers are initialized from their previous state and run for \(T=30\) inner steps. The Euclidean baseline uses the plain exact-tracker update
\(
z_k^{(t+1)}
=
z_k^{(t)}
-
\gamma_{\mathrm{E}}\nabla_y\psi_\mu(x_k,z_k^{(t)}),
\)
with fixed step size
\(
\gamma_{\mathrm{E}}
=
0.7\,\gamma_{\mathrm{crit}}(0), \quad
\gamma_{\mathrm{crit}}(k)
=
\frac{2}{\lambda_{\max}(\nabla_{yy}^2\psi_\mu(x_k,y_\mu^\star(x_k)))} .
\)
The quantity \(\gamma_{\mathrm{crit}}(k)\) is used only as a reference local Euclidean stability scale.
The Euclidean step size is chosen from the initial, well-conditioned part of the path and then held
fixed as the exact center moves toward the boundary. The barrier-metric tracker uses the frozen Dikin preconditioner from Algorithm~\ref{alg:dikin_sbo}.
At outer index \(k\), update
\(
z_k^{(t+1)}
=
z_k^{(t)}
-
\gamma_k \nabla^2\phi(z_k)^{-1}\nabla_y\psi_\mu(x_k,z_k^{(t)}).
\)
The scalar \(\gamma_k\) is chosen conservatively, following the \textit{barrier-aware} schedule principle.
The exact centers \(y_\mu^\star(x_k)\) are not used to choose the implementable steps; they are
computed only for visualization and post-hoc verification.

After generating the trajectories, we evaluate the anchored Dikin error
\[
\|z_k-y_\mu^\star(x_k)\|_{y_\mu^\star(x_k)}
=
\Bigl(
(z_k-y_\mu^\star(x_k))^\top
\nabla^2\phi(y_\mu^\star(x_k))
(z_k-y_\mu^\star(x_k))
\Bigr)^{1/2}
\]
and check whether it remains below the neighborhood radius \(\eta=0.25\), namely whether
\[
z_k\in
T_{\eta,\mu}(x_k)
:=
\left\{
z\in\operatorname{int}(Y):
\|z-y_\mu^\star(x_k)\|_{y_\mu^\star(x_k)}\le\eta
\right\}.
\]
This neighborhood condition provides a post-hoc certificate of whether the generated trajectory remains in
the local Dikin region. The comparison shows that the fixed Euclidean tracker loses tube control near
the active face, whereas the barrier-metric tracker remains consistent with the moving local Dikin
geometry.

Figure~\ref{fig:motivating_example} reports four quantities. Panel~(a) plots the exact centers
\(y_\mu^\star(x_k)\), the Euclidean tracker, the barrier-metric tracker, and the first Euclidean tube
exit. Panel~(b) plots the auxiliary outer value \(f(x_k,z_k)\), together with the exact-center value
\(f(x_k,y_\mu^\star(x_k))\). Panel~(c) plots the anchored Dikin error and marks the tube radius
\(\eta\). Panel~(d) plots the exact lower objective gap
\(\psi_\mu(x_k,z_k)-\psi_\mu(x_k,y_\mu^\star(x_k))\).
The Euclidean tracker is shown until it leaves the maintained tube near the active face; the
barrier-metric tracker remains feasible and tube-controlled throughout the path.

\begin{figure}[h]
    \centering
    \includegraphics[width=\linewidth]{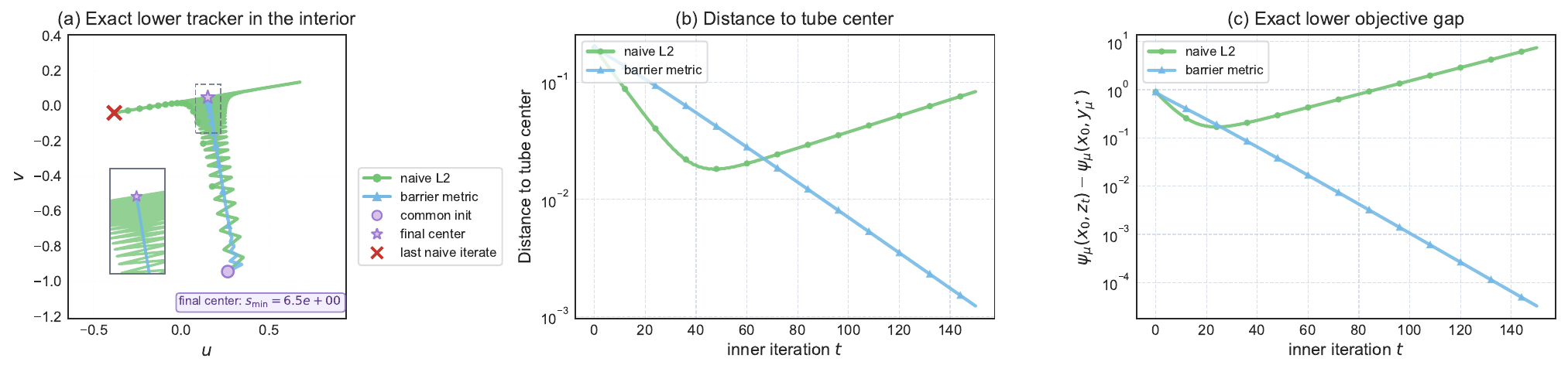}
    \caption{
Interior stability experiment for the exact lower tracker.
(a) The Euclidean tracker oscillates and drifts under a step size above the local Euclidean stability
threshold, while the barrier-metric tracker remains stable.
(b) The auxiliary objective decreases for the barrier-metric tracker but eventually increases for the
Euclidean tracker.
(c) The anchored Dikin error contracts under the barrier-metric and grows under the Euclidean step.
(d) The exact lower objective gap decays for the barrier-metric tracker, while the Euclidean gap
increases after a transient phase.
}
\label{fig:motivating_example_interior}
\end{figure}

\paragraph{Additional experiment: interior stability.}
We also include an auxiliary experiment in which the exact barrierized minimizer is far from the
boundary. In contrast to the active-face example above, here the outer parameter is fixed at
\(x_0=0.5\). This removes center drift and isolates the stability of the lower tracker under different
local geometries. We use the same two-dimensional polyhedral setup as in the main motivating example. Let
\(Y=\{y\in\mathbb R^2:Ay\le b\}\), with logarithmic barrier \(
\phi(y)=-\sum_{i=1}^m\log s_i(y),
s_i(y):=b_i-a_i^\top y .\)
For the fixed outer parameter \(x_0\), consider the exact barrierized lower objective
\[
\psi_\mu(x_0,y)
=
\frac12 (y-c)^\top Q(y-c)+\mu\phi(y),
\]
where \(Q\succ0\) is anisotropic and \(\mu>0\) is small. The center \(c\) and the polytope are chosen
so that
\[
y_\mu^\star(x_0)
=
\arg\min_{y\in\operatorname{int}(Y)}\psi_\mu(x_0,y)
\]
lies well inside \(Y\). We set the slack:
\(
\min_i s_i(y_\mu^\star(x_0))\approx 6.51,
\)
so the observed behavior is not caused by proximity to an active face. Starting from a common interior point \(z_0\) at Euclidean distance about \(1\) from
\(y_\mu^\star(x_0)\), we compare two exact lower trackers for the same objective
\(\psi_\mu(x_0,\cdot)\). The Euclidean tracker uses
\(
z_{t+1}
=
z_t-\gamma_{\mathrm{E}}\nabla_y\psi_\mu(x_0,z_t),
\)
where \(\gamma_{\mathrm{E}}\) is chosen slightly above the local Euclidean stability threshold
\(
\gamma_{\mathrm{crit}}
=
\frac{2}{\lambda_{\max}(\nabla_{yy}^2\psi_\mu(x_0,y_\mu^\star(x_0)))}.
\)
The barrier-metric tracker uses
\(
z_{t+1}
=
z_t-\gamma_{\mathrm{B}}\nabla^2\phi(z_t)^{-1}\nabla_y\psi_\mu(x_0,z_t).
\)

Figure~\ref{fig:motivating_example_interior} reports the trajectory, the auxiliary objective value
\(f(x_0,z_t)\), the anchored Dikin error
\[
\|z_t-y_\mu^\star(x_0)\|_{y_\mu^\star(x_0)}
=
\Big(
(z_t-y_\mu^\star(x_0))^\top
\nabla^2\phi(y_\mu^\star(x_0))
(z_t-y_\mu^\star(x_0))
\Big)^{1/2},
\]
and the exact lower objective gap \(
\psi_\mu(x_0,z_t)-\psi_\mu(x_0,y_\mu^\star(x_0)).
\) Although all iterates remain feasible and the minimizer is strictly interior, the Euclidean tracker
becomes unstable once its step size exceeds the local Euclidean stability threshold. In contrast, the
barrier-metric tracker steadily reduces both the Dikin error and the lower objective gap. This example
shows that the barrier-metric is useful not only for preventing boundary exits, but also for stabilizing
lower-level tracking under anisotropic local curvature.

\subsection{Congestion Toll Design}
\label{app:toll_design}

\paragraph{Main experiment.}
We provide the details of the synthetic congestion-toll benchmark used in Section~\ref{sec:experiments}.
For each pair \((n,\mathrm{seed})\), we generate an independent toll-design instance using
\texttt{numpy.random.default\_rng(seed)}. The number of shared bottlenecks is
\(
m_b=\max\{5,\lfloor n/10\rfloor\}.
\)
The incidence matrix \(C\in\{0,1\}^{m_b\times n}\) is generated by assigning each corridor to
between one and three uniformly sampled bottlenecks, followed by a correction step ensuring that
every bottleneck is used by at least one corridor. Corridor capacities are sampled as \(u_i\sim\mathrm{Unif}[0.8,1.2]\), and the total demand cap is set
to
\(
D=0.6\sum_i u_i .
\)
We initialize the lower variable at the strictly interior point \(y_{\mathrm{int}}=0.15u\), rescaling if
necessary so that \(\mathbf 1^\top y_{\mathrm{int}}<0.5D\). The bottleneck capacities are then chosen
to make this point strictly feasible while preserving random tightness. Specifically, with
\(r:=Cu\) and \(c:=Cy_{\mathrm{int}}\), we first sample
\(\widetilde d_r\sim\mathrm{Unif}[0.45,0.65]\,r_r\) and set
\(
d_r
=
\max\left\{
\widetilde d_r,\,
\frac{1.35\,c_r+10^{-3}}{\tau}
\right\}.
\)
This guarantees \(Cy_{\mathrm{int}}<\tau d\) with a fixed margin for every generated instance. The baseline travel costs are sampled as \(\ell_i\sim\mathrm{Unif}[0.5,2.0]\). The congestion matrix is
\(
Q=\operatorname{diag}(q)+VV^\top/n,
\)
where \(q_i\sim\mathrm{Unif}[0.1,0.5]\) and \(V\in\mathbb R^{n\times 3}\) has entries
\(V_{ij}\sim N(0,0.25^2)\). We use
\(
\kappa=1,
\beta=1,
\rho_{\mathrm{rev}}=10^{-2},
\rho_x=10^{-3}.
\)
All methods are initialized with
\(
x_0=(0.5,\ldots,0.5)^\top,
\)
and the tolls are constrained by \(0\le x\le 10\). The target revenue is set to
\(
R_{\mathrm{tar}}=0.25D\,\operatorname{mean}(x_0).
\)

In the main experiment, we use bottleneck tightness \(\tau=0.20\), barrier parameter
\(\mu=10^{-3}\), and dimensions
\(
n\in\{50,100,200,400,600,800,1000,1200\}.\) Each method receives a \(2\)-second optimization wall-clock budget per
instance. Exact objective evaluations used only for reporting are not counted against this optimization
budget. We compare five methods. BMFO is our proposed barrier-metric first-order method.
Exact-HG-Barrier solves the barrier-smoothed lower problem accurately and applies the exact implicit
hypergradient. CVXPY-HG solves the original lower quadratic program using CVXPY and
differentiates the active-set KKT system (\cite{agrawal2019differentiable}). Khanduri-DSIGD is a smoothed implicit-gradient baseline
inspired by \citet{khanduri2023linearly}. Kornowski-F2CBA is a first-order linearly constrained
bilevel baseline inspired by \citet{kornowski2024first}, implemented with a warm-started first-order
lower oracle.  For BMFO, the barrier-metric directions require solving linear systems with the
lower-level log-barrier Hessian. In the reported experiments, we use a dense
Cholesky solve for smaller instances and switch to a preconditioned
conjugate-gradient (PCG) solve for larger instances. For the normalized gap in Figure~\ref{fig:toll_design}(b), we distinguish the barrier-smoothed objective from the original
constrained bilevel objective. Given a toll vector \(x\), we define
\[
F_{\mathrm{orig}}(x) := f(x,y^\star(x)),\qquad
y^\star(x)\in\arg\min_{y\in Y(\tau)} g(x,y),
\]
where \(Y(\tau)\) is the original polyhedral feasible set without the log-barrier
term. We evaluate \(F_{\mathrm{orig}}\) by solving this lower-level convex QP to
high accuracy and then plugging the resulting response \(y^\star(x)\) into the
upper objective \(f\). Thus \(F_{\mathrm{orig}}\) measures performance on the
original constrained traffic-response problem. For each instance \((n,\mathrm{seed},\tau)\), the reference value
\(F_{\mathrm{ref}}\) is the best \(F_{\mathrm{orig}}\) value found by the offline
reference pool. Panel (b) reports
\[
\frac{F_{\mathrm{orig}}(x_{\mathrm{final}})-F_{\mathrm{ref}}}
{\max\{10^{-4},10^{-3}|F_{\mathrm{ref}}|\}}.
\]
Values below one meet the target tolerance, and values above \(1.2\) are clipped
only for visualization.

\paragraph{Additional results.}
Figures~\ref{fig:toll_design_tau03}--\ref{fig:toll_design_budget4} report additional congestion-toll
experiments varying the bottleneck tightness and wall-clock budget. Figure~\ref{fig:toll_design_tau03}
uses a slacker feasible region, \(\tau=0.30\), under the same 2-second optimization budget as the main
experiment. Figure~\ref{fig:toll_design_budget4} compares \(\tau=0.20\) and \(\tau=0.30\) under a
larger 4-second budget. Across these variants, BMFO retains low per-update cost as the lower
dimension grows and remains competitive in normalized original-objective gap, while exact or
implicit-gradient baselines become increasingly expensive at larger dimensions.

\begin{figure}[h]
    \centering
    \includegraphics[width=.75\linewidth]{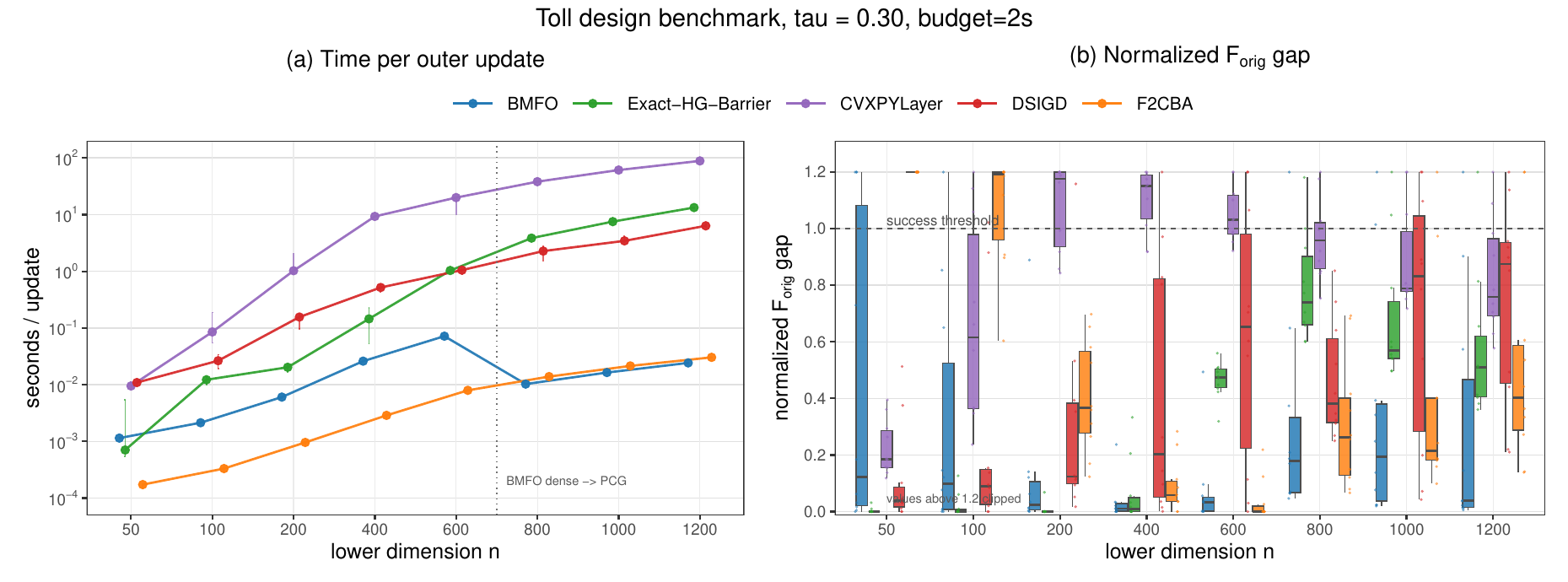}
\caption{
Congestion-toll benchmark with bottleneck tightness \(\tau=0.30\) and a \(2\)-second optimization
wall-clock budget per method.}
\label{fig:toll_design_tau03}
\end{figure}

\begin{figure}[h]
    \centering
    \includegraphics[width=.75\linewidth]{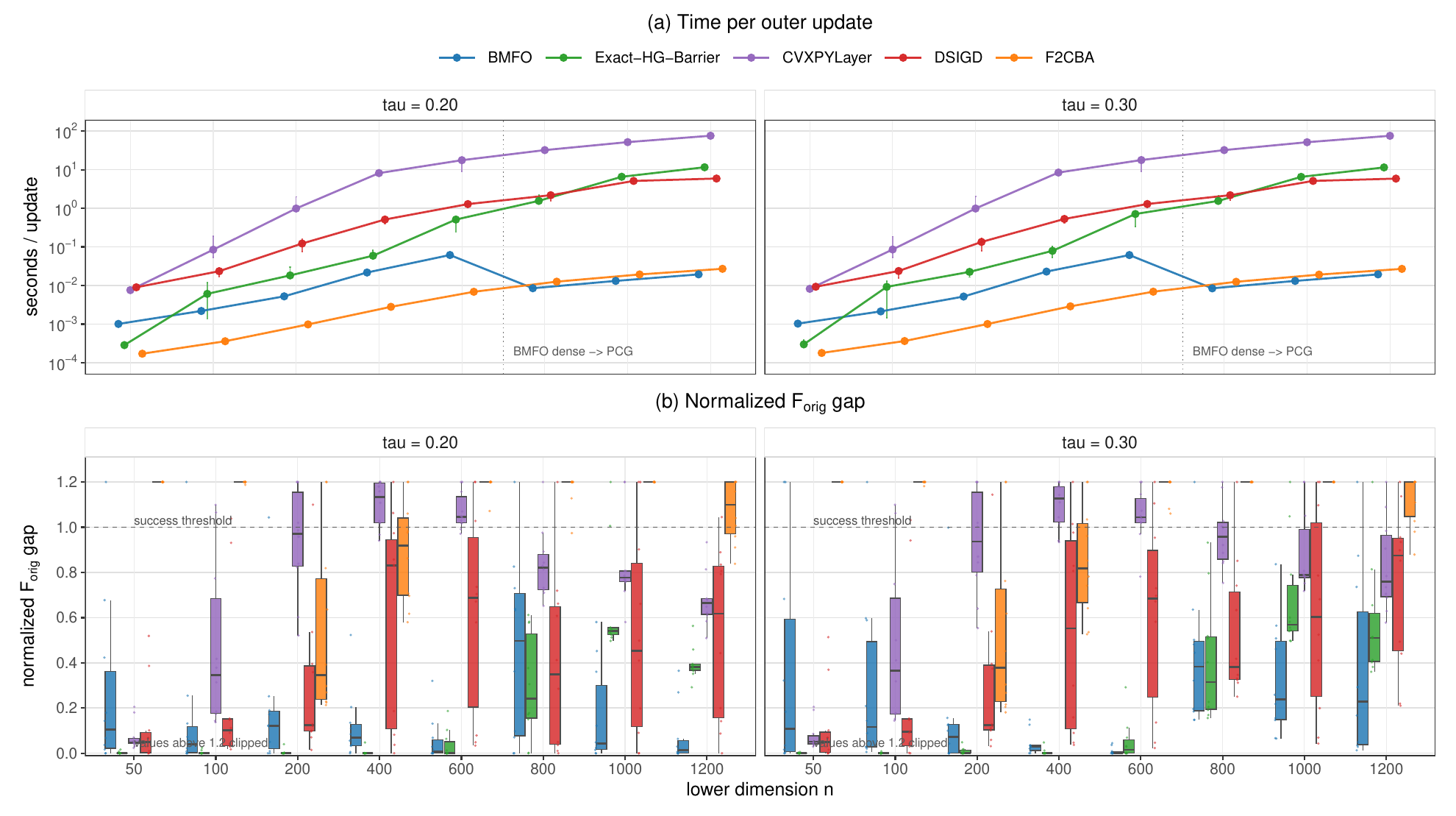}
\caption{
Congestion-toll benchmark with bottleneck tightness \(\tau=0.20, 0.30\) and a \(4\)-second optimization
wall-clock budget per method. 
}
\label{fig:toll_design_budget4}
\end{figure}

\newpage

\subsection{Constrained MNIST Hypercleaning}

\paragraph{Main experiment.}
We use a constrained MNIST hypercleaning benchmark. For each random seed
\(s\in\{0,1,2,3,4\}\), we load the standard MNIST training and test splits using
\texttt{torchvision}, normalize pixel intensities to \([0,1]\), and vectorize each image as
\(\xi\in\mathbb R^{784}\). We apply a seed-dependent random permutation to the \(60{,}000\)
training examples: the first \(n=19{,}000\) examples form the noisy training set, and the next
\(m=1{,}000\) examples form the clean validation set. We also keep \(5{,}000\) randomly permuted
test examples for evaluation. Training labels are corrupted independently with probability \(p\). If an example with label
\(y\in\{0,\ldots,9\}\) is selected for corruption, its label is replaced by
\((y+u)\bmod 10\), where \(u\sim{\rm Unif}\{1,\ldots,9\}\), so the corrupted label is always
different from the original one. Validation and test labels are kept clean. For each seed, BMFO and
F2CBA use the same train/validation/test split and the same corrupted training labels. Let
\( \mathcal D_{\rm train}=\{(\tilde\xi_i,\tilde y_i)\}_{i=1}^n,
\mathcal D_{\rm val}=\{(\xi_j,y_j)\}_{j=1}^m .
\) The upper variable \(\lambda\in\mathbb R^n\) assigns the training weight
\(\sigma(\lambda_i)\in(0,1)\) to example \(i\). The lower variable
\(w\in\mathbb R^{dK_{\rm cls}}\) is the vectorized parameter of a
\(K_{\rm cls}=10\)-class softmax-regression model, with \(d=784\). The lower problem is
\[
w^\star(\lambda)\in
\arg\min_{w\in\mathcal W_\tau}
\left\{
\frac1n\sum_{i=1}^n
\sigma(\lambda_i)\ell_{\rm CE}(\tilde\xi_i,\tilde y_i;w)
+
\frac{\rho}{2}\|w\|_2^2
\right\},
\]
and the upper objective is the validation cross-entropy
\[
F(\lambda)
=
\frac1m\sum_{j=1}^m
\ell_{\rm CE}(\xi_j,y_j;w^\star(\lambda)).
\]
The lower feasible set combines coordinatewise box constraints with patchwise signed-mass
constraints. We divide the \(28\times28\) image grid into \(4\times4\) non-overlapping
\(7\times7\) patches \(\{\mathcal P_r\}_r\), and impose
\[
\mathcal W_\tau
:=
\left\{
w:
|w_{q,c}|\le R\ \forall q,c,\qquad
\left|\sum_{q\in\mathcal P_r}w_{q,c}\right|
\le \tau b_{r,c}\ \forall r,c
\right\}.
\]
Here \(R>0\) is the coordinatewise box radius, \(b_{r,c}>0\) is the patchwise mass scale, and
\(\tau\in(0,1]\) controls constraint tightness. In our implementation, we use
\(
R=2, b_{r,c}=|\mathcal P_r|=7\cdot7=49,\tau=0.2,
\) so the patchwise signed-mass threshold is
\(
\tau b_{r,c}=0.2\cdot49=9.8 \)
for every patch \(r\) and class \(c\). Equivalently, the code uses
\texttt{patch\_size=7} and patch threshold \(9.8\). The initialization \(w=0\) is strictly feasible
for both the box and patchwise constraints. Figure~\ref{fig:mnist_hypercleaning} uses \(p=0.2\);
Figure~\ref{fig:mnist_hypercleaning_p04} reports the higher-corruption case \(p=0.4\).

\paragraph{Additional results.}
Figure~\ref{fig:mnist_hypercleaning_p04} reports the same constrained MNIST hypercleaning
experiment with a higher label corruption probability \(p=0.4\). The qualitative behavior is consistent
with the \(p=0.2\) setting: F2CBA decreases the validation loss quickly at the beginning but plateaus,
whereas BMFO continues to improve and reaches a lower final validation loss. The same trend holds
when performance is measured by the number of gradient evaluations, indicating that the improvement
is not only an iteration-count effect.
\begin{figure}[h]
    \centering
    \includegraphics[width=0.75\linewidth]{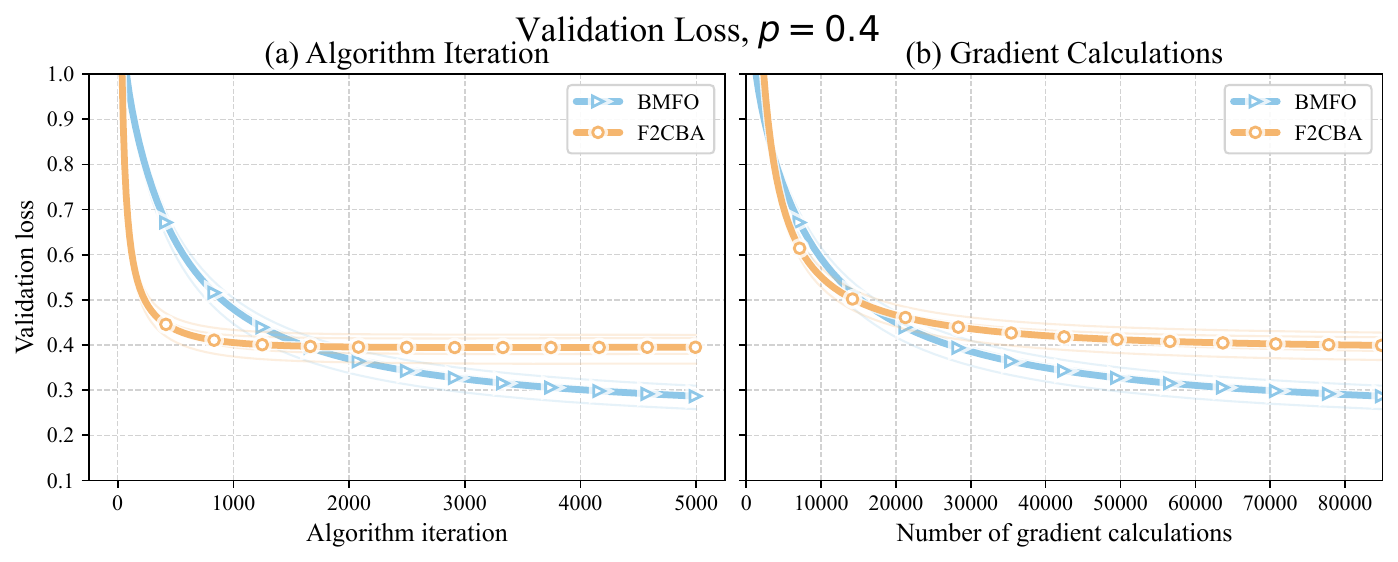}
\caption{
Constrained MNIST hypercleaning with label corruption probability \(p=0.4\).
Panel~(a) reports validation cross-entropy versus algorithm iterations, and panel~(b) reports
validation cross-entropy versus the number of gradient evaluations. Under the higher corruption
level, BMFO continues to reduce validation loss after F2CBA plateaus, and attains a lower final loss
both per iteration and per gradient evaluation.
}
\label{fig:mnist_hypercleaning_p04}
\end{figure}

\section{Other Supporting Lemmas}
\label{app:supporting_lemmas}

This appendix collects auxiliary estimates used in the proxy-bias and tracker-recursion arguments.
Lemma~\ref{lem:proxy_direction_linearization} gives a local linearization of the proxy direction around
the exact barrierized center, which is the main algebraic tool behind the \(O(1/\lambda)\)
proxy-gradient bias. Lemma~\ref{lem:proxy_jacobian_lipschitz} records the local smoothness of
the proxy minimizer map \(x\mapsto y_{\lambda,\mu}^\star(x)\). The final two lemmas,
Lemmas~\ref{lem:exact_center_cross_bound} and~\ref{lem:proxy_center_cross_bound}, turn exact and
proxy center drift into inner-product bounds against arbitrary tracker-error vectors; these are used to
control the cross terms in the moving-anchor recursions.

\subsection{Proxy-gradient Linearization}

\begin{lemma}[Local proxy-gradient linearization]
\label{lem:proxy_direction_linearization}
Fix \(\lambda>0\) and \(x\in X\). For any \(y\in T_{\eta,\mu}(x)\), we have
\begin{align}
&\Big\|
\nabla F_\mu(x)
-\nabla_x L_{\lambda,\mu}(x,y)
+\nabla^2_{xy}\psi_\mu(x,y_\mu^\star(x))^\top
\big(\nabla^2_{yy}\psi_\mu(x,y_\mu^\star(x))\big)^{-1}
\nabla_y L_{\lambda,\mu}(x,y)
\Big\|_2
\nonumber\\
&\qquad\le
2\Big(\frac{\ell_{\psi,1}^\eta}{\rho_\psi^\eta}\Big)
\|y-y_\mu^\star(x)\|_{y_\mu^\star(x)}
\Big(
\ell_{f,1}^\eta
+
\lambda\min\{
2\ell_{\psi,1}^\eta,\,
\ell_{\psi,2}^\eta\|y-y_\mu^\star(x)\|_{y_\mu^\star(x)}
\}
\Big).
\label{eq:proxy_linearization_bound}
\end{align}
\end{lemma}
\begin{proof}
Fix \(x\in X\). Since
\(\nabla_y\psi_\mu(x,y_\mu^\star(x))=0\) and
\(\nabla^2_{yy}\psi_\mu(x,y_\mu^\star(x))\succ0\) on the maintained tube, the implicit function
theorem gives
\begin{equation}
\label{eq:support_Dy_mu}
Dy_\mu^\star(x)
=
-\big(\nabla^2_{yy}\psi_\mu(x,y_\mu^\star(x))\big)^{-1}
\nabla^2_{xy}\psi_\mu(x,y_\mu^\star(x)).
\end{equation}
Hence, by the chain rule,
\begin{align}
\nabla F_\mu(x)
&=
\nabla_x f(x,y_\mu^\star(x))
+
\big(Dy_\mu^\star(x)\big)^\top
\nabla_y f(x,y_\mu^\star(x))
\nonumber\\
&=
\nabla_x f(x,y_\mu^\star(x))
-
\nabla^2_{xy}\psi_\mu(x,y_\mu^\star(x))^\top
\big(\nabla^2_{yy}\psi_\mu(x,y_\mu^\star(x))\big)^{-1}
\nabla_y f(x,y_\mu^\star(x)).
\label{eq:support_gradF_mu}
\end{align}

Next, using
\(L_{\lambda,\mu}(x,y)=f(x,y)+\lambda(\psi_\mu(x,y)-\psi_\mu^\star(x))\) and
\(\nabla_x\psi_\mu^\star(x)=\nabla_x\psi_\mu(x,y_\mu^\star(x))\), we have
\begin{align}
\nabla_x L_{\lambda,\mu}(x,y)
&=
\nabla_x f(x,y)
+
\lambda\big(
\nabla_x\psi_\mu(x,y)-\nabla_x\psi_\mu(x,y_\mu^\star(x))
\big),
\nonumber\\
\nabla_y L_{\lambda,\mu}(x,y)
&=
\nabla_y f(x,y)+\lambda\nabla_y\psi_\mu(x,y).
\label{eq:support_proxy_derivatives}
\end{align}

Combining \eqref{eq:support_gradF_mu}--\eqref{eq:support_proxy_derivatives} gives
\begin{align}
&\nabla F_\mu(x)
-\nabla_x L_{\lambda,\mu}(x,y)
+
\nabla^2_{xy}\psi_\mu(x,y_\mu^\star(x))^\top
\big(\nabla^2_{yy}\psi_\mu(x,y_\mu^\star(x))\big)^{-1}
\nabla_y L_{\lambda,\mu}(x,y)
\nonumber\\
&=
\underbrace{
\nabla_x f(x,y_\mu^\star(x))-\nabla_x f(x,y)
}_{T_1}
\nonumber\\
&\quad+
\underbrace{
\nabla^2_{xy}\psi_\mu(x,y_\mu^\star(x))^\top
\big(\nabla^2_{yy}\psi_\mu(x,y_\mu^\star(x))\big)^{-1}
\big(
\nabla_y f(x,y)-\nabla_y f(x,y_\mu^\star(x))
\big)
}_{T_2}
\nonumber\\
&\quad
-\lambda
\underbrace{
\Big(
\nabla_x\psi_\mu(x,y)-\nabla_x\psi_\mu(x,y_\mu^\star(x))
-\nabla^2_{xy}\psi_\mu(x,y_\mu^\star(x))^\top
(y-y_\mu^\star(x))
\Big)
}_{R_x}
\nonumber\\
&\quad
+\lambda
\underbrace{
\nabla^2_{xy}\psi_\mu(x,y_\mu^\star(x))^\top
\big(\nabla^2_{yy}\psi_\mu(x,y_\mu^\star(x))\big)^{-1}
}_{\text{linear correction}}
\nonumber\\
&\qquad\qquad\qquad\cdot
\underbrace{
\Big(
\nabla_y\psi_\mu(x,y)-\nabla_y\psi_\mu(x,y_\mu^\star(x))
-\nabla^2_{yy}\psi_\mu(x,y_\mu^\star(x))(y-y_\mu^\star(x))
\Big)
}_{\widetilde R_y}.
\label{eq:support_proxy_decomposition}
\end{align}
The linear \(\lambda\nabla^2_{xy}\psi_\mu(x,y_\mu^\star(x))^\top(y-y_\mu^\star(x))\) terms cancel.

\proofstep{Operator-norm in the anchored Dikin geometry:}
We will use the induced norms
\[
\|M\|_{2\to y_\mu^\star(x),*}
:=
\sup_{\|u\|_2=1}\|Mu\|_{y_\mu^\star(x),*},
\qquad
\|M\|_{y_\mu^\star(x)\to 2}
:=
\sup_{\|v\|_{y_\mu^\star(x)}=1}\|Mv\|_2.
\]
By duality,
\[
\|M^\top\|_{y_\mu^\star(x)\to 2}
=
\|M\|_{2\to y_\mu^\star(x),*}.
\]
Proposition~\ref{prop:derived_dikin_regularity} implies
\begin{equation}
\label{eq:support_cross_operator_bound}
\|\nabla^2_{xy}\psi_\mu(x,y_\mu^\star(x))\|_{2\to y_\mu^\star(x),*}
\le
\ell_{\psi,1}^\eta,
\end{equation}
and therefore
\[
\|\nabla^2_{xy}\psi_\mu(x,y_\mu^\star(x))^\top\|_{y_\mu^\star(x)\to 2}
\le
\ell_{\psi,1}^\eta .
\]
Also, since
\[
\nabla^2_{yy}\psi_\mu(x,y_\mu^\star(x))
\succeq
\rho_\psi^\eta H_\mu^\star(x),
\]
we have
\begin{equation}
\label{eq:support_hinv_operator_bound}
\big\|
\big(\nabla^2_{yy}\psi_\mu(x,y_\mu^\star(x))\big)^{-1}w
\big\|_{y_\mu^\star(x)}
\le
\frac{1}{\rho_\psi^\eta}\|w\|_{y_\mu^\star(x),*}.
\end{equation}
Combining the preceding two displays yields
\begin{equation}
\label{eq:support_mixed_operator_bound}
\Big\|
\nabla^2_{xy}\psi_\mu(x,y_\mu^\star(x))^\top
\big(\nabla^2_{yy}\psi_\mu(x,y_\mu^\star(x))\big)^{-1}
\Big\|_{y_\mu^\star(x),*\to 2}
\le
\frac{\ell_{\psi,1}^\eta}{\rho_\psi^\eta}.
\end{equation}

\proofstep{Bound \(T_1\) and \(T_2\).}
By the local bound for \(f\), with Euclidean--Dikin conversion factors absorbed into
\(\ell_{f,1}^\eta\),
\begin{equation}
\label{eq:support_T1_bound}
\|T_1\|_2
\le
\ell_{f,1}^\eta\|y-y_\mu^\star(x)\|_{y_\mu^\star(x)}.
\end{equation}
For \(T_2\), combine \eqref{eq:support_mixed_operator_bound} with the local Lipschitz bound for
\(\nabla_y f\):
\[
\|T_2\|_2
\le
\frac{\ell_{\psi,1}^\eta}{\rho_\psi^\eta}
\|\nabla_y f(x,y)-\nabla_y f(x,y_\mu^\star(x))\|_{y_\mu^\star(x),*}
\le
\frac{\ell_{\psi,1}^\eta}{\rho_\psi^\eta}
\ell_{f,1}^\eta
\|y-y_\mu^\star(x)\|_{y_\mu^\star(x)} .
\]
Since \(\ell_{\psi,1}^\eta/\rho_\psi^\eta\ge1\), \eqref{eq:support_T1_bound} gives
\begin{equation}
\label{eq:support_f_terms_bound}
\|T_1\|_2+\|T_2\|_2
\le
2\Big(\frac{\ell_{\psi,1}^\eta}{\rho_\psi^\eta}\Big)
\ell_{f,1}^\eta
\|y-y_\mu^\star(x)\|_{y_\mu^\star(x)} .
\end{equation}

\proofstep{Bound \(R_x\) and \(\widetilde R_y\).}
For \(R_x\), the integral remainder form gives
\begin{align}
R_x
&=
\int_0^1
\Big(
\nabla^2_{xy}\psi_\mu(x,y_\mu^\star(x)+t(y-y_\mu^\star(x)))
-
\nabla^2_{xy}\psi_\mu(x,y_\mu^\star(x))
\Big)^\top
\nonumber\\
&\hspace{6.5cm}\cdot
(y-y_\mu^\star(x))\,dt .
\label{eq:support_Rx_integral}
\end{align}
By duality, for any matrix \(M\),
\[
\|M^\top(y-y_\mu^\star(x))\|_2
\le
\|M\|_{2\to y_\mu^\star(x),*}
\|y-y_\mu^\star(x)\|_{y_\mu^\star(x)}.
\]
Thus
\begin{align}
\|R_x\|_2
&\le
\|y-y_\mu^\star(x)\|_{y_\mu^\star(x)}
\int_0^1
\Big\|
\nabla^2_{xy}\psi_\mu(x,y_\mu^\star(x)+t(y-y_\mu^\star(x)))
\nonumber\\
&\hspace{4.5cm}
-
\nabla^2_{xy}\psi_\mu(x,y_\mu^\star(x))
\Big\|_{2\to y_\mu^\star(x),*}\,dt .
\label{eq:support_Rx_prebound}
\end{align}
On the tube, either use the crude bound
\(\|\nabla^2_{xy}\psi_\mu(x,\cdot)\|_{2\to y_\mu^\star(x),*}\le\ell_{\psi,1}^\eta\) twice, or use the
Hessian-Lipschitz-in-\(y\) bound with constant \(\ell_{\psi,2}^\eta\). Hence
\begin{equation}
\label{eq:support_Rx_bound}
\|R_x\|_2
\le
\|y-y_\mu^\star(x)\|_{y_\mu^\star(x)}
\min\Big\{
2\ell_{\psi,1}^\eta,\,
\ell_{\psi,2}^\eta\|y-y_\mu^\star(x)\|_{y_\mu^\star(x)}
\Big\}.
\end{equation}

Similarly, since \(\nabla_y\psi_\mu(x,y_\mu^\star(x))=0\),
\begin{align}
\widetilde R_y
&=
\nabla_y\psi_\mu(x,y)
-
\nabla_y\psi_\mu(x,y_\mu^\star(x))
-
\nabla^2_{yy}\psi_\mu(x,y_\mu^\star(x))(y-y_\mu^\star(x))
\nonumber\\
&=
\int_0^1
\Big(
\nabla^2_{yy}\psi_\mu(x,y_\mu^\star(x)+t(y-y_\mu^\star(x)))
-
\nabla^2_{yy}\psi_\mu(x,y_\mu^\star(x))
\Big)
\nonumber\\
&\hspace{6.5cm}\cdot
(y-y_\mu^\star(x))\,dt .
\label{eq:support_Ry_integral}
\end{align}
Using the corresponding operator norm
\(\|\cdot\|_{y_\mu^\star(x)\to y_\mu^\star(x),*}\), we get
\begin{equation}
\label{eq:support_Ry_inner_bound}
\|\widetilde R_y\|_{y_\mu^\star(x),*}
\le
\|y-y_\mu^\star(x)\|_{y_\mu^\star(x)}
\min\Big\{
2\ell_{\psi,1}^\eta,\,
\ell_{\psi,2}^\eta\|y-y_\mu^\star(x)\|_{y_\mu^\star(x)}
\Big\}.
\end{equation}
Finally, applying \eqref{eq:support_mixed_operator_bound} to the linear correction multiplying
\(\widetilde R_y\), we obtain
\begin{equation}
\label{eq:support_Ry_bound}
\Big\|
\nabla^2_{xy}\psi_\mu(x,y_\mu^\star(x))^\top
\big(\nabla^2_{yy}\psi_\mu(x,y_\mu^\star(x))\big)^{-1}
\widetilde R_y
\Big\|_2
\le
\frac{\ell_{\psi,1}^\eta}{\rho_\psi^\eta}
\|\widetilde R_y\|_{y_\mu^\star(x),*}.
\end{equation}
Since \(\ell_{\psi,1}^\eta/\rho_\psi^\eta\ge1\), combining
\eqref{eq:support_Rx_bound}--\eqref{eq:support_Ry_bound} gives
\begin{align}
&\lambda
\left(
\|R_x\|_2
+
\Big\|
\nabla^2_{xy}\psi_\mu(x,y_\mu^\star(x))^\top
\big(\nabla^2_{yy}\psi_\mu(x,y_\mu^\star(x))\big)^{-1}
\widetilde R_y
\Big\|_2
\right)
\nonumber\\
&\qquad\le
2\Big(\frac{\ell_{\psi,1}^\eta}{\rho_\psi^\eta}\Big)
\lambda
\|y-y_\mu^\star(x)\|_{y_\mu^\star(x)}
\min\Big\{
2\ell_{\psi,1}^\eta,\,
\ell_{\psi,2}^\eta\|y-y_\mu^\star(x)\|_{y_\mu^\star(x)}
\Big\}.
\label{eq:support_remainder_bound}
\end{align}

Taking Euclidean norms in \eqref{eq:support_proxy_decomposition}, applying the triangle inequality,
and using \eqref{eq:support_f_terms_bound} and \eqref{eq:support_remainder_bound}, we obtain
\eqref{eq:proxy_linearization_bound}.
\end{proof}

\subsection{Local Smoothness}

\begin{lemma}[Local smoothness of \(x\mapsto y_{\lambda,\mu}^\star(x)\)]
\label{lem:proxy_jacobian_lipschitz}
Fix \(\lambda\ge \lambda_0\). For each \(x\in X\), let
\(y_{\lambda,\mu}^\star(x)\) denote the unique minimizer of
\(y\mapsto L_{\lambda,\mu}(x,y)\), and assume the maintained tube condition
\(y_{\lambda,\mu}^\star(x)\in T_{\eta,\mu}^{\lambda}(x)\) for all
\(x\in X\).

Then the map \(x\mapsto y_{\lambda,\mu}^\star(x)\) is differentiable on
\(X\), and its Jacobian satisfies
\begin{equation}
\label{eq:proxy_jacobian_formula}
\nabla y_{\lambda,\mu}^\star(x)
=
-\Big(\nabla^2_{yy}L_{\lambda,\mu}
(x,y_{\lambda,\mu}^\star(x))\Big)^{-1}
\nabla^2_{xy}L_{\lambda,\mu}
(x,y_{\lambda,\mu}^\star(x)).
\end{equation}
Moreover, there exists a finite local constant \(\ell_{\lambda,1}^\eta<\infty\) such that,
for all \(x_1,x_2\in X\),
\begin{equation}
\label{eq:proxy_jacobian_lipschitz}
\big\|
\nabla y_{\lambda,\mu}^\star(x_2)
-
\nabla y_{\lambda,\mu}^\star(x_1)
\big\|_{2\to y_{\lambda,\mu}^\star(x_1)}
\le
\ell_{\lambda,1}^\eta\|x_2-x_1\|_2 .
\end{equation}
One admissible choice is
\begin{equation}
\label{eq:proxy_jacobian_constant}
\ell_{\lambda,1}^\eta
:=
32\left(
\frac{\ell_{f,2}^\eta}{\lambda}
+
\ell_{\psi,2}^\eta
\right)
\frac{(\ell_{\psi,1}^\eta)^2}{(\rho_\psi^\eta)^3}.
\end{equation}
Uniformly over \(\lambda\ge\lambda_0\), one may replace
\(\ell_{f,2}^\eta/\lambda\) by \(\ell_{f,2}^\eta/\lambda_0\).
\end{lemma}

\begin{proof}
Fix \(x_1,x_2\in X\). By optimality, for each visited \(x\),
\[
\nabla_y L_{\lambda,\mu}(x,y_{\lambda,\mu}^\star(x))=0.
\]
By Proposition~\ref{prop:derived_dikin_regularity} and
Lemma~\ref{lem:proxy_objective_curvature}, the Hessian
\(\nabla^2_{yy}L_{\lambda,\mu}(x,y)\) is positive definite on
\(T_{\eta,\mu}^{\lambda}(x)\). Hence the implicit function theorem implies differentiability of
\(x\mapsto y_{\lambda,\mu}^\star(x)\) and yields
\eqref{eq:proxy_jacobian_formula}.

Using \eqref{eq:proxy_jacobian_formula} at \(x_1\) and \(x_2\), we have
\begin{align}
&\nabla y_{\lambda,\mu}^\star(x_2)
-
\nabla y_{\lambda,\mu}^\star(x_1)
\nonumber\\
&=
-\Big(\nabla^2_{yy}L_{\lambda,\mu}
(x_2,y_{\lambda,\mu}^\star(x_2))\Big)^{-1}
\nabla^2_{xy}L_{\lambda,\mu}
(x_2,y_{\lambda,\mu}^\star(x_2))
\nonumber\\
&\quad
+
\Big(\nabla^2_{yy}L_{\lambda,\mu}
(x_1,y_{\lambda,\mu}^\star(x_1))\Big)^{-1}
\nabla^2_{xy}L_{\lambda,\mu}
(x_1,y_{\lambda,\mu}^\star(x_1)).
\label{eq:proxy_jacobian_difference}
\end{align}
To compare the two terms, multiply the difference by
\(\nabla^2_{yy}L_{\lambda,\mu}(x_2,y_{\lambda,\mu}^\star(x_2))\) and add/subtract the usual
cross term. Using the lower curvature bound
\[
\nabla^2_{yy}L_{\lambda,\mu}(x_2,y)
\succeq
\frac{\lambda\rho_\psi^\eta}{2}H_{\lambda,\mu}^\star(x_2),
\]
we obtain
\begin{align}
&\frac{\lambda\rho_\psi^\eta}{2}
\big\|
\nabla y_{\lambda,\mu}^\star(x_2)
-
\nabla y_{\lambda,\mu}^\star(x_1)
\big\|_{2\to y_{\lambda,\mu}^\star(x_2)}
\nonumber\\
&\le
\Big\|
\nabla^2_{xy}L_{\lambda,\mu}
(x_2,y_{\lambda,\mu}^\star(x_2))
-
\nabla^2_{xy}L_{\lambda,\mu}
(x_1,y_{\lambda,\mu}^\star(x_1))
\Big\|_{2\to y_{\lambda,\mu}^\star(x_2),*}
\nonumber\\
&\quad
+
\Big\|
\nabla^2_{yy}L_{\lambda,\mu}
(x_2,y_{\lambda,\mu}^\star(x_2))
-
\nabla^2_{yy}L_{\lambda,\mu}
(x_1,y_{\lambda,\mu}^\star(x_1))
\Big\|_{y_{\lambda,\mu}^\star(x_2)\to y_{\lambda,\mu}^\star(x_2),*}
\nonumber\\
&\qquad\qquad\cdot
\big\|
\nabla y_{\lambda,\mu}^\star(x_1)
\big\|_{2\to y_{\lambda,\mu}^\star(x_2)} .
\label{eq:proxy_jacobian_key_ineq}
\end{align}
The local Lipschitz-Hessian bounds for \(L_{\lambda,\mu}\) give
\[
\Big\|
\nabla^2_{xy}L_{\lambda,\mu}
(x_2,y_{\lambda,\mu}^\star(x_2))
-
\nabla^2_{xy}L_{\lambda,\mu}
(x_1,y_{\lambda,\mu}^\star(x_1))
\Big\|_{2\to y_{\lambda,\mu}^\star(x_2),*}
\]
\[
\le
(\ell_{f,2}^\eta+\lambda\ell_{\psi,2}^\eta)
\Big(
\|x_2-x_1\|_2
+
\|y_{\lambda,\mu}^\star(x_2)-y_{\lambda,\mu}^\star(x_1)\|_{y_{\lambda,\mu}^\star(x_2)}
\Big),
\]
and the same bound holds for the Hessian difference
\[
\nabla^2_{yy}L_{\lambda,\mu}
(x_2,y_{\lambda,\mu}^\star(x_2))
-
\nabla^2_{yy}L_{\lambda,\mu}
(x_1,y_{\lambda,\mu}^\star(x_1))
\]
in the corresponding \(y_{\lambda,\mu}^\star(x_2)\to y_{\lambda,\mu}^\star(x_2),*\) operator norm.
By Proposition~\ref{prop:local_consequences}\textup{(i)} with fixed \(\lambda\),
\[
\|y_{\lambda,\mu}^\star(x_2)-y_{\lambda,\mu}^\star(x_1)\|_{y_{\lambda,\mu}^\star(x_2)}
\le
\ell_{\lambda,0}^\eta\|x_2-x_1\|_2,
\]
after absorbing the fixed anchor-switch constant. Also,
\[
\big\|\nabla y_{\lambda,\mu}^\star(x)\big\|_{2\to y_{\lambda,\mu}^\star(x)}
\le
\ell_{\lambda,0}^\eta ,
\]
again by the local stability of the proxy minimizer map. Therefore
\begin{align}
&\frac{\lambda\rho_\psi^\eta}{2}
\big\|
\nabla y_{\lambda,\mu}^\star(x_2)
-
\nabla y_{\lambda,\mu}^\star(x_1)
\big\|_{2\to y_{\lambda,\mu}^\star(x_2)}
\nonumber\\
&\le
(\ell_{f,2}^\eta+\lambda\ell_{\psi,2}^\eta)
(1+\ell_{\lambda,0}^\eta)^2
\|x_2-x_1\|_2 .
\label{eq:proxy_jacobian_pre_final}
\end{align}
Rearranging gives
\[
\big\|
\nabla y_{\lambda,\mu}^\star(x_2)
-
\nabla y_{\lambda,\mu}^\star(x_1)
\big\|_{2\to y_{\lambda,\mu}^\star(x_2)}
\le
\frac{2(\ell_{f,2}^\eta+\lambda\ell_{\psi,2}^\eta)}
{\lambda\rho_\psi^\eta}
(1+\ell_{\lambda,0}^\eta)^2
\|x_2-x_1\|_2.
\]
Using the proxy stability bound
\[
\ell_{\lambda,0}^\eta
\le
\frac{3\ell_{\psi,1}^\eta}{\rho_\psi^\eta},
\]
and increasing the absolute constant if necessary, we obtain the bound:
\begin{equation}
\label{eq:proxy_jacobian_final}
\big\|
\nabla y_{\lambda,\mu}^\star(x_2)
-
\nabla y_{\lambda,\mu}^\star(x_1)
\big\|_{2\to y_{\lambda,\mu}^\star(x_2)}
\le
32\left(
\frac{\ell_{f,2}^\eta}{\lambda}
+
\ell_{\psi,2}^\eta
\right)
\frac{(\ell_{\psi,1}^\eta)^2}{(\rho_\psi^\eta)^3}
\|x_2-x_1\|_2 .
\end{equation}
Finally, applying the self-concordant anchor-switch comparison between
\(y_{\lambda,\mu}^\star(x_2)\) and \(y_{\lambda,\mu}^\star(x_1)\) only changes the constant by a fixed
\(\eta\)-dependent factor, which is absorbed into
\(\ell_{\lambda,1}^\eta\). This gives \eqref{eq:proxy_jacobian_lipschitz}.
\end{proof}

\paragraph{Convention on anchor-switch constants.}
The constants \(\ell_{*,1}^{\eta}\)  in Definition~\ref{def:barrier_aware} \eqref{eq:main_sched_s3} and \(\ell_{\lambda,1}^{\eta}\) in Lemma~\ref{lem:proxy_jacobian_lipschitz} are only required to be finite
upper bounds in the exact and proxy center-shift estimates. Hence, without loss of generality, we
enlarge \(\ell_{*,1}^{\eta}\) so that \(\ell_{*,1}^{\eta}\ge \max\{1,\ell_{\lambda,1}^{\eta}\}.\) With this convention, the single constant \(\ell_{*,1}^{\eta}\) appearing in the \textit{barrier-aware} schedule
controls both the exact and proxy anchor-switch remainders.

\subsection{Cross Bound}

\begin{lemma}[Cross bound for exact-center drift]
\label{lem:exact_center_cross_bound}
Fix an outer iteration \(k\), and let \(v_k\in\mathbb R^{d_y}\) be any vector chosen before updating
\(x_k\). We use the following local bounds for the exact minimizer map
\(x\mapsto y_\mu^\star(x)\).

\begin{enumerate}
\item[\textup{(J0)}] \textup{(Lipschitz continuity).}
There exists a finite constant \(\ell_{*,0}^\eta<\infty\) such that, for all
\(x\in X\) and all \(u\in\mathbb R^{d_x}\),
\begin{equation}
\label{eq:exact_map_jacobian_bound}
\|\nabla y_\mu^\star(x)u\|_{y_\mu^\star(x)}
\le
\ell_{*,0}^\eta\|u\|_2 .
\end{equation}

\item[\textup{(J1)}] \textup{(Smoothness).}
There exists a finite constant \(\ell_{*,1}^\eta<\infty\) such that
\begin{equation}
\label{eq:exact_map_taylor_remainder}
\Big\|
y_\mu^\star(x_{k+1})-y_\mu^\star(x_k)
-\nabla y_\mu^\star(x_k)(x_{k+1}-x_k)
\Big\|_{y_\mu^\star(x_k)}
\le
\frac{\ell_{*,1}^\eta}{2}\|x_{k+1}-x_k\|_2^2 .
\end{equation}
\end{enumerate}

Then, for any \(j_k>0\),
\begin{align}
\label{eq:exact_center_cross_bound}
&\big\langle v_k,\,
y_\mu^\star(x_{k+1})-y_\mu^\star(x_k)
\big\rangle
\nonumber\\
&\quad\le
\Bigl(
\xi\alpha_k j_k
+
\frac{(\ell_{*,1}^\eta)^2}{2}\xi^2
\bigl(\ell_{f,0}^2\alpha_k^2+4\ell_{g,0}^2\beta_k^2\bigr)
\Bigr)
\|v_k\|_{y_\mu^\star(x_k),*}^2
\nonumber\\
&\qquad
+
\Bigl(
\frac{\xi\alpha_k}{4j_k}(\ell_{*,0}^\eta)^2
+
\frac{\xi^2}{4}\alpha_k^2
\Bigr)
\|q_k^x\|_2^2 .
\end{align}
\end{lemma}

\begin{proof}
We decompose the exact-center displacement into its linear part and Taylor remainder:
\begin{align}
&\big\langle v_k,\,
y_\mu^\star(x_{k+1})-y_\mu^\star(x_k)
\big\rangle
\nonumber\\
&\quad=
\underbrace{
\big\langle v_k,\,
\nabla y_\mu^\star(x_k)(x_{k+1}-x_k)
\big\rangle
}_{(I)}
\nonumber\\
&\qquad+
\underbrace{
\Big\langle v_k,\,
y_\mu^\star(x_{k+1})-y_\mu^\star(x_k)
-\nabla y_\mu^\star(x_k)(x_{k+1}-x_k)
\Big\rangle
}_{(II)} .
\label{eq:exact_center_split}
\end{align}

Since \(x_{k+1}-x_k=-\xi\alpha_k q_k^x\), we have
\[
(I)
=
-\xi\alpha_k
\big\langle v_k,\nabla y_\mu^\star(x_k)q_k^x\big\rangle .
\]
By Cauchy--Schwarz in the anchored Dikin primal/dual pair and
\eqref{eq:exact_map_jacobian_bound},
\[
|(I)|
\le
\xi\alpha_k
\|v_k\|_{y_\mu^\star(x_k),*}
\|\nabla y_\mu^\star(x_k)q_k^x\|_{y_\mu^\star(x_k)}
\le
\xi\alpha_k
\ell_{*,0}^\eta
\|v_k\|_{y_\mu^\star(x_k),*}
\|q_k^x\|_2 .
\]
Applying Young's inequality with parameter \(j_k>0\) gives
\begin{equation}
\label{eq:exact_linear_part_bound}
(I)
\le
\xi\alpha_k j_k\|v_k\|_{y_\mu^\star(x_k),*}^2
+
\frac{\xi\alpha_k}{4j_k}(\ell_{*,0}^\eta)^2\|q_k^x\|_2^2 .
\end{equation}

By \eqref{eq:exact_map_taylor_remainder} and Cauchy--Schwarz,
\begin{align}
(II)
&\le
\|v_k\|_{y_\mu^\star(x_k),*}
\Big\|
y_\mu^\star(x_{k+1})-y_\mu^\star(x_k)
-\nabla y_\mu^\star(x_k)(x_{k+1}-x_k)
\Big\|_{y_\mu^\star(x_k)}
\nonumber\\
&\le
\frac{\ell_{*,1}^\eta}{2}
\|v_k\|_{y_\mu^\star(x_k),*}
\|x_{k+1}-x_k\|_2^2 .
\label{eq:exact_remainder_prebound}
\end{align}

Using \(x_{k+1}-x_k=-\xi\alpha_k q_k^x\) and the bound
\[
\|q_k^x\|_2
\le
\ell_{f,0}+2\lambda_k\ell_{g,0},
\]
we have
\begin{align}
\|x_{k+1}-x_k\|_2^2
&=
\xi^2\alpha_k^2\|q_k^x\|_2^2
\nonumber\\
&\le
\xi^2\alpha_k^2
\bigl(\ell_{f,0}+2\lambda_k\ell_{g,0}\bigr)^2
\nonumber\\
&\le
2\xi^2
\bigl(\ell_{f,0}^2\alpha_k^2+4\ell_{g,0}^2\beta_k^2\bigr).
\label{eq:outer_step_norm_bound_exact}
\end{align}
Combining \eqref{eq:exact_remainder_prebound} and
\eqref{eq:outer_step_norm_bound_exact}, and applying Young's inequality to the remaining product,
yields
\begin{equation}
\label{eq:exact_remainder_bound}
(II)
\le
\frac{(\ell_{*,1}^\eta)^2}{2}\xi^2
\bigl(\ell_{f,0}^2\alpha_k^2+4\ell_{g,0}^2\beta_k^2\bigr)
\|v_k\|_{y_\mu^\star(x_k),*}^2
+
\frac{\xi^2}{4}\alpha_k^2\|q_k^x\|_2^2 .
\end{equation}

Combining \eqref{eq:exact_linear_part_bound} and \eqref{eq:exact_remainder_bound} in
\eqref{eq:exact_center_split} proves \eqref{eq:exact_center_cross_bound}.
\end{proof}

\begin{lemma}[Cross bound for proxy-center drift]
\label{lem:proxy_center_cross_bound}
Fix an outer iteration \(k\), and let \(\delta_k:=\lambda_{k+1}-\lambda_k\ge0\). Let
\(v_k\in\mathbb R^{d_y}\) be any vector chosen before updating \(x_k\). Assume
\(\lambda_k,\lambda_{k+1}\ge\lambda_0\), so that Proposition~\ref{prop:local_consequences} applies.
For the fixed-\(\lambda_k\) proxy map \(x\mapsto y_{\lambda_k,\mu}^\star(x)\), we use the following
local bounds.

\begin{enumerate}
\item[\textup{(J0)}] \textup{(Lipschitz continuity).}
There exists a finite constant \(\ell_{\lambda,0}^\eta<\infty\) such that, for all
\(x\in X\) and all \(u\in\mathbb R^{d_x}\),
\begin{equation}
\label{eq:proxy_map_jacobian_bound}
\|\nabla y_{\lambda_k,\mu}^\star(x)u\|_{y_{\lambda_k,\mu}^\star(x)}
\le
\ell_{\lambda,0}^\eta\|u\|_2 .
\end{equation}

\item[\textup{(J1)}] \textup{(Smoothness).}
There exists a finite constant \(\ell_{\lambda,1}^\eta<\infty\) such that
\begin{equation}
\label{eq:proxy_map_taylor_remainder}
\Big\|
y_{\lambda_k,\mu}^\star(x_{k+1})-y_{\lambda_k,\mu}^\star(x_k)
-\nabla y_{\lambda_k,\mu}^\star(x_k)(x_{k+1}-x_k)
\Big\|_{y_{\lambda_k,\mu}^\star(x_k)}
\le
\frac{\ell_{\lambda,1}^\eta}{2}\|x_{k+1}-x_k\|_2^2 .
\end{equation}
\end{enumerate}

Then, for any \(j_k>0\),
\begin{align}
\label{eq:proxy_center_cross_bound}
&\big\langle v_k,\,
y_{\lambda_{k+1},\mu}^\star(x_{k+1})
-y_{\lambda_k,\mu}^\star(x_k)
\big\rangle
\nonumber\\
&\quad\le
\Bigl(
\frac{\delta_k}{\lambda_k}
+
\xi\alpha_k j_k
+
\frac{(\ell_{\lambda,1}^\eta)^2}{2}\xi^2
\bigl(\ell_{f,0}^2\alpha_k^2+4\ell_{g,0}^2\beta_k^2\bigr)
\Bigr)
\|v_k\|_{y_{\lambda_k,\mu}^\star(x_k),*}^2
\nonumber\\
&\qquad
+
\Bigl(
\frac{\xi\alpha_k}{4j_k}(\ell_{\lambda,0}^\eta)^2
+
\frac{\xi^2}{4}\alpha_k^2
\Bigr)\|q_k^x\|_2^2
+
\frac{\kappa_{y,k}^2(\ell_{f,0}^\eta)^2}{(\rho_\psi^\eta)^2}
\frac{\delta_k}{\lambda_k^3}.
\end{align}
\end{lemma}

\begin{proof}
If \(\delta_k=0\), the pure multiplier-drift term below vanishes. We therefore assume
\(\delta_k>0\); the case \(\delta_k=0\) follows by the same argument after omitting that term.

\proofstep{Decomposition:}
Add and subtract \(y_{\lambda_k,\mu}^\star(x_{k+1})\):
\begin{align}
&\big\langle v_k,\,
y_{\lambda_{k+1},\mu}^\star(x_{k+1})
-y_{\lambda_k,\mu}^\star(x_k)
\big\rangle
\nonumber\\
&\quad=
\underbrace{
\big\langle v_k,\,
y_{\lambda_{k+1},\mu}^\star(x_{k+1})
-y_{\lambda_k,\mu}^\star(x_{k+1})
\big\rangle
}_{(I)}
\nonumber\\
&\qquad+
\underbrace{
\big\langle v_k,\,
y_{\lambda_k,\mu}^\star(x_{k+1})
-y_{\lambda_k,\mu}^\star(x_k)
\big\rangle
}_{(II)} .
\label{eq:proxy_center_split}
\end{align}

\proofstep{Multiplier drift:}
By Cauchy--Schwarz in the anchored Dikin primal/dual pair,
\begin{align}
(I)
&\le
\|v_k\|_{y_{\lambda_k,\mu}^\star(x_k),*}
\big\|
y_{\lambda_{k+1},\mu}^\star(x_{k+1})
-y_{\lambda_k,\mu}^\star(x_{k+1})
\big\|_{y_{\lambda_k,\mu}^\star(x_k)} .
\label{eq:proxy_lambda_drift_cs}
\end{align}
Using the anchor-switch comparison between
\(y_{\lambda_k,\mu}^\star(x_k)\) and \(y_{\lambda_k,\mu}^\star(x_{k+1})\),
\[
\big\|u\big\|_{y_{\lambda_k,\mu}^\star(x_k)}
\le
\kappa_{y,k}\big\|u\big\|_{y_{\lambda_k,\mu}^\star(x_{k+1})},
\]
and applying Proposition~\ref{prop:local_consequences}\textup{(i)} with
\((\lambda_1,x_1)=(\lambda_k,x_{k+1})\) and
\((\lambda_2,x_2)=(\lambda_{k+1},x_{k+1})\), we get
\[
\big\|
y_{\lambda_{k+1},\mu}^\star(x_{k+1})
-y_{\lambda_k,\mu}^\star(x_{k+1})
\big\|_{y_{\lambda_k,\mu}^\star(x_{k+1})}
\le
\frac{2\delta_k\ell_{f,0}^\eta}
{\lambda_k\lambda_{k+1}\rho_\psi^\eta}
\le
\frac{2\delta_k\ell_{f,0}^\eta}
{\lambda_k^2\rho_\psi^\eta}.
\]
Hence
\[
(I)
\le
\|v_k\|_{y_{\lambda_k,\mu}^\star(x_k),*}
\cdot
\kappa_{y,k}
\frac{2\delta_k\ell_{f,0}^\eta}{\lambda_k^2\rho_\psi^\eta}.
\]
Young's inequality with
\[
a:=\|v_k\|_{y_{\lambda_k,\mu}^\star(x_k),*},
\qquad
b:=\kappa_{y,k}\frac{2\delta_k\ell_{f,0}^\eta}{\lambda_k^2\rho_\psi^\eta},
\qquad
c:=\frac{\delta_k}{\lambda_k},
\]
gives
\begin{equation}
\label{eq:proxy_lambda_drift_bound}
(I)
\le
\frac{\delta_k}{\lambda_k}
\|v_k\|_{y_{\lambda_k,\mu}^\star(x_k),*}^2
+
\frac{\kappa_{y,k}^2(\ell_{f,0}^\eta)^2}{(\rho_\psi^\eta)^2}
\frac{\delta_k}{\lambda_k^3}.
\end{equation}

\proofstep{Fixed-\(\lambda_k\) motion:}
Decompose \((II)\) by Taylor expansion of the fixed-\(\lambda_k\) proxy map at \(x_k\):
\begin{align}
(II)
&=
\underbrace{
\big\langle v_k,\,
\nabla y_{\lambda_k,\mu}^\star(x_k)(x_{k+1}-x_k)
\big\rangle
}_{(IIa)}
\nonumber\\
&\quad+
\underbrace{
\Big\langle v_k,\,
y_{\lambda_k,\mu}^\star(x_{k+1})
-y_{\lambda_k,\mu}^\star(x_k)
-\nabla y_{\lambda_k,\mu}^\star(x_k)(x_{k+1}-x_k)
\Big\rangle
}_{(IIb)} .
\label{eq:proxy_x_drift_split}
\end{align}

Since \(x_{k+1}-x_k=-\xi\alpha_kq_k^x\), the linear term satisfies
\[
(IIa)
=
-\xi\alpha_k
\big\langle v_k,\nabla y_{\lambda_k,\mu}^\star(x_k)q_k^x\big\rangle .
\]
By Cauchy--Schwarz and \eqref{eq:proxy_map_jacobian_bound},
\[
|(IIa)|
\le
\xi\alpha_k
\|v_k\|_{y_{\lambda_k,\mu}^\star(x_k),*}
\ell_{\lambda,0}^\eta\|q_k^x\|_2 .
\]
Young's inequality with parameter \(j_k>0\) yields
\begin{equation}
\label{eq:proxy_linear_drift_bound}
(IIa)
\le
\xi\alpha_kj_k
\|v_k\|_{y_{\lambda_k,\mu}^\star(x_k),*}^2
+
\frac{\xi\alpha_k}{4j_k}
(\ell_{\lambda,0}^\eta)^2
\|q_k^x\|_2^2 .
\end{equation}

For the Taylor remainder, \eqref{eq:proxy_map_taylor_remainder} and Cauchy--Schwarz give
\begin{align}
(IIb)
&\le
\|v_k\|_{y_{\lambda_k,\mu}^\star(x_k),*}
\Big\|
y_{\lambda_k,\mu}^\star(x_{k+1})
-y_{\lambda_k,\mu}^\star(x_k)
-\nabla y_{\lambda_k,\mu}^\star(x_k)(x_{k+1}-x_k)
\Big\|_{y_{\lambda_k,\mu}^\star(x_k)}
\nonumber\\
&\le
\frac{\ell_{\lambda,1}^\eta}{2}
\|v_k\|_{y_{\lambda_k,\mu}^\star(x_k),*}
\|x_{k+1}-x_k\|_2^2 .
\label{eq:proxy_remainder_prebound}
\end{align}
As before,
\begin{align}
\|x_{k+1}-x_k\|_2^2
&=
\xi^2\alpha_k^2\|q_k^x\|_2^2
\nonumber\\
&\le
\xi^2\alpha_k^2
\bigl(\ell_{f,0}+2\lambda_k\ell_{g,0}\bigr)^2
\nonumber\\
&\le
2\xi^2
\bigl(\ell_{f,0}^2\alpha_k^2+4\ell_{g,0}^2\beta_k^2\bigr).
\label{eq:outer_step_norm_bound_proxy}
\end{align}
Combining \eqref{eq:proxy_remainder_prebound} and \eqref{eq:outer_step_norm_bound_proxy}, and
absorbing the remaining product by Young's inequality, yields
\begin{equation}
\label{eq:proxy_remainder_bound}
(IIb)
\le
\frac{(\ell_{\lambda,1}^\eta)^2}{2}\xi^2
\bigl(\ell_{f,0}^2\alpha_k^2+4\ell_{g,0}^2\beta_k^2\bigr)
\|v_k\|_{y_{\lambda_k,\mu}^\star(x_k),*}^2
+
\frac{\xi^2}{4}\alpha_k^2\|q_k^x\|_2^2 .
\end{equation}

Combining \eqref{eq:proxy_lambda_drift_bound}, \eqref{eq:proxy_linear_drift_bound}, and
\eqref{eq:proxy_remainder_bound} in \eqref{eq:proxy_center_split} proves
\eqref{eq:proxy_center_cross_bound}.
\end{proof}

%%%%%%%%%%%%%%%%%%%%%%%%%%%%%%%%%%%%%%%%%%%%%%%%%%%%%%%%%%%%

% \newpage
% \input{checklist.tex}

\end{document}